\documentclass[reqno,centertags,10pt]{amsart}

\usepackage{amsmath,amsthm,amssymb}

\usepackage{comment}

\usepackage[shortlabels]{enumitem}

\usepackage{xcolor}
\usepackage[colorlinks=true,urlcolor=blue,citecolor=red,linkcolor=blue,linktocpage,pdfpagelabels,bookmarksnumbered,bookmarksopen]{hyperref}

-\marginparwidth \oddsidemargin -4mm \evensidemargin \oddsidemargin

\DeclareMathOperator{\arccot}{arccot}

\newtheorem{theorem}{Theorem}[section]

\newtheorem{proposition}[theorem]{Proposition}
\newtheorem{lemma}[theorem]{Lemma}
\newtheorem{corollary}[theorem]{Corollary}
\theoremstyle{definition}
\newtheorem{definition}[theorem]{Definition}
\newtheorem{example}[theorem]{Example}

\newtheorem{remark}[theorem]{Remark}

\newtheorem{thm}{Theorem}

\DeclareMathOperator*{\Lip}{Lip}

\allowdisplaybreaks
\numberwithin{equation}{section}

\newcommand{\lb}{\label}

\newcommand{\beq}{\begin{equation}}
\newcommand{\eeq}{\end{equation}}

\newcommand{\bal}{\begin{align}}
\newcommand{\eal}{\end{align}}
\newcommand{\bals}{\begin{align*}}
\newcommand{\eals}{\end{align*}}

\newcommand{\bbN}{{\mathbb{N}}}
\newcommand{\bbR}{{\mathbb{R}}}
\newcommand{\R}{{\mathbb{R}}}

\newcommand{\bbT}{{\mathbb{T}}}
\newcommand{\bbS}{{\mathbb{S}}}

\newcommand{\calT}{{\mathcal T}}

\newcommand{\calL}{{\mathcal L}}
\newcommand{\calQ}{{\mathcal Q}}

\newcommand{\eps}{\varepsilon}

\newcommand{\al}{\alpha}

\def\e{\varepsilon}

\usepackage[T1]{fontenc}
\usepackage{esint}

\begin{document}
\title[Regularity of drift-Hele-Shaw Flow]
{Regularity of Hele-Shaw Flow with Source and Drift}

\author{Inwon Kim and Yuming Paul Zhang}

\address{\noindent Department of Mathematics \\ University of California Los Angeles\\ Los Angeles\\ CA \newline Email: \tt
ikim@math.ucla.edu}

\address{\noindent Department of Mathematics \& Statistics \\ Auburn University \\ Auburn\\ AL \newline Email: \tt
yzhangpaul@auburn.edu}





\begin{abstract} 
In this paper we study the regularity property of Hele-Shaw flow, where source and drift are present in the evolution. More specifically we consider H\"{o}lder continuous source and Lipschitz continuous drift. 
We show that if the free boundary of the solution is locally close to a Lipschitz graph, then it is indeed Lipschitz, given that the Lipschitz constant is small. When there is no drift, our result establishes $C^{1,\gamma}$ regularity of the free boundary by combining our result with the obstacle problem theory. In general, when the source and drift are both smooth, we prove that the solution is non-degenerate, indicating higher regularity of the free boundary.

\end{abstract}

\maketitle


\vspace{5pt}
\vspace{5pt}

\section{Introduction} \lb{S1}

Let $\vec{b}:\bbR^d\to\bbR^d$ be a Lipschitz continuous vector field, and $f:\bbR^d\to [0,\infty)$ be a non-negative H\"{o}lder continuous function.  We consider $u=u(x,t) \geq 0$ solving the Hele-Shaw type problem:
\begin{equation}\lb{1.1}
    \left\{\begin{aligned}
        -\Delta u &=f \quad &&\text{ in }\{u>0\},\\
        u_t&=|\nabla u|^2+\vec{b}\cdot\nabla u\quad &&\text{ on } \partial\{u>0\}.
    \end{aligned}\right.
\end{equation}
We refer to $\partial\{u>0\}$ as the {\it free boundary } of $u$. The second equation is the level set formulation of the velocity law 
\begin{equation}\label{velocity}
V = (-\nabla u - \vec{b})\cdot \nu = |\nabla u| - \vec{b} \cdot \nu \hbox{ on }\partial\{u>0\},
\end{equation}
 where $V$ denotes the velocity of the set $\{u>0\}$ along the outward spatial normal $\nu = \frac{-\nabla u}{|\nabla u|}$ at the given free boundary point $(x,t)\in \partial\{u>0\}$. 

\medskip

{\eqref{1.1} corresponds to the classical {\it Hele-Shaw flow}  \cite{HS1898,richardson1972hele}  when both $f$ and $\vec{b}$ are zero, and with a fixed boundary where $u$ is prescribed a constant value. We will refer to \eqref{1.1} in this setting as the {\it injection problem}, since $u$ here denotes the pressure of the fluid being injected via the fixed boundary.} The general equation \eqref{1.1} can be also written as the continuity equation  
$$\rho_t - \nabla \cdot ((\nabla u+ \vec{b}) \rho) = \tilde{f}\rho,$$ with the density variable $\rho = \chi_{\{u>0\}}$ and growth term $\tilde{f}:= f - \nabla\cdot \vec{b}$. In other words, $\rho$ is transported by the velocity field $-(\nabla u + \vec{b})$ and with the growth term $\tilde{f}$. In this context, $u$ can be understood as the pressure variable, and is generated by the incompressibility constraint $\rho \leq 1$ to transport density that intends to move with drift $-\vec{b}$ and growth rate $\tilde{f}$. Due to this interpretation of the model, \eqref{1.1} has been actively studied in the recent literature, for instance in the context of tumor growth  where cells evolve with contact inhibition, \cite{PQV,David_S,jacobs2022tumor} and in the context of congested population dynamics \cite{maury2010,CKY}. The injection problem, with injection at infinity, can be also understood as the one-phase Muskat problem with gravity, see \cite{Wu23,Hongjie}.


\medskip

Understanding the regularity of free boundary problem is in general an important and challenging task, even in the elliptic setting, due to its potential singularities and degeneracy. For our problem, we expect that pressure gradient $\nabla u$ that drives the velocity law in \eqref{velocity} will regularize the flow, but we must ensure that the flow avoids both degeneracy of the pressure and topological singularities. 
If the Lipschitz constant of the initial free boundary is small so that the pressure is non-degenerate, both local and global-time regularity results are available, see \cite{CJK,CJK2,Hongjie}. General regularity results that involves characterization of topological singularity are available for the injection problem, by its connection to the obstacle problem: see \cite{figalli20} for instance.

\medskip

For our inhomogeneous problem, zooming in at a single point $(x_0,t_0)\in \partial\{u>0\}$ with the hyperbolic scale $\tilde{u}_r(x,t):= r^{-1}u(r(x-x_0), r(t-t_0))$, one formally sees that the source term tends to zero and the drift becomes a constant vector field as $r$ tends to zero. Thus it seems plausible that similar regularity theory as for the injection problem  holds. However, it is unclear how to quantify this heuristical argument, even in the setting of zero drift, as we will discuss in detail below.
\medskip

\medskip

  Our main result shows that {``free boundaries close to a Lipschitz graph are indeed Lipschitz, and the solution is non-degenerate, as long as the Lipschitz constant is small''. Let us give a brief summary of Theorem A and B below:} \medskip

{\bf Main theorem: }{\it Let $u$ solve \eqref{1.1}. Suppose that $u$ is close to a cone-monotone profile {at each time} in a local space-time neighborhood. If the angle of the cone is sufficiently large, then the solution is fully cone-monotone and  the free boundary is Lipschitz.  In addition, if $\vec{b}$ is zero, the free boundary is $C^{1,\gamma}$ for some $0<\gamma<1$. Lastly, if $f$ and $\vec{b}$ are at least $C^3$, the solution is also non-degenerate, namely it features faster-than-linear growth near the free boundary.}

\medskip

Our result extends the celebrated {free boundary regularity theory introduced by Caffarelli et. al. (see the book \cite{CafSal}) as well as its corresponding version for the injection problem \cite{CJK,CJK2}.   In particular our work serves as the first attempt to understand the effect of source and drift on the regularization mechanism of the free boundary evolution. As mentioned above, the presence of a nonzero $f$ alone  necessitates some significant changes in the standard arguments. Our proof relies on  spatially Lipschitz solutions of the injection problem that were constructed in \cite{CJK}, as well as the properties of superharmonic functions given in section \ref{S.4}. 

\medskip

In general,  the Lipschitz regularity of the free boundary and the non-degeneracy of the solutions are the two ingredients of further regularity analysis in aforementioned references. We thus suspect that  the free boundary in our statement is in fact  $C^{1,\gamma}$ in space and time, when $f$ and $\vec{b}$ are smooth. Given the technical nature of these arguments, we do not pursue this next step in full generality, to lay out the main arguments to achieve the basic regularity results as clearly as possible. Our higher regularity result with zero drift already has interesting implications in applications. Indeed there are examples of log-Lipschitz continuous function $f$ with $\vec{b}=0$ {that describes tumor growth with nutrients},  for which numerical experiments reveal immediate dendrite-like growth on the free boundary \cite{kitsunezaki1997interface,PTV, maury14}.  The dynamics behind the generation of such irregularities remain mysterious: our $C^{1,\gamma}$ result with zero drift implies that such fingers appear only at a large scale.

\medskip

 Some remarks on the assumption are in order. Our assumption considers solutions whose level sets are close to Lipschitz graphs up to small scale, which is a natural assumption when considering rough initial data. At least in the case of $\vec{b}=0$ and $f=1$, this assumption is satisfied for a small time for those who start from an initially Lipschitz graph with small Lipschitz constant: see Corollorary~\ref{cor:flat}.  This assumption is also motivated from the well-known {\it waiting time} phenomena, where the initial free boundary does not move for a finite amount of time, delaying its regularization by the pressure force. For the classical Hele-Shaw problem with $f=\vec{b}=0$, it is well-known that there is a waiting time phenomena with  initial free boundary in the shape of a sharp cone \cite{klv} . The same remains true in the presence of the source term $f\in L^{\infty}$: see Example~\ref{cone_sharp} where the free boundary keeps its profile as a sharp cone with a fixed vertex for a unit amount of time. The waiting time phenomena is the main obstacle for regularization of the Hele-Shaw flow for a short-time interval. We show that this phenomena does not occur when cone has sufficiently large angle.

\medskip




\medskip

Let us also briefly discuss the optimality of assumptions on $f$ and $\vec{b}$.  It is not hard to see that the condition is optimal for the drift term: when $\vec{b}$ is not Lipschitz continuous, one can construct an example where the solution starting with a cone as its positive set  maintains the cone shape as its positive set, even developing a cusp at the vertex of the cone (see Example \ref{E.1}).  On the other hand it is less clear whether the regularity of $f$ is sharp for the theorem. The H\"{o}lder regularity of $f$ appears to be close to the optimal condition for the ``flat implies Lipschitz'' result. We will show an example  (see Example \ref{E.2}) where this result is false with merely bounded $f$. We also refer to a counterexample in \cite{blank2001sharp} for the obstacle problem, the time integrated version of our problem, with a continuous $f$ that is not Dini-continuous.  For the non-degeneracy, it remains unclear whether smoothness is required for $f$ and $\vec{b}$: see more discussions on this in Section \ref{S.7}.

\medskip

$\circ$ {\it Regularization mechanism, new challenges and ingredients:}  

\smallskip

 In \eqref{1.1} the support of  the pressure variable $u$ moves along the velocity field $-(\nabla u + \vec{b})$. Due to the elliptic equation $u$ solves in its support, $\nabla u$ acts as the regularizing force in the flow. We largely follow the outline of \cite{CafSal} and \cite{CJK} for our analysis, which involve analyzing the re-scaled function $\tilde{u}_r(x,t)= r^{-1}u(r(x-x_0), r(t-t_0))$ with decreasing scale $r$. The hope is to quantitatively estimate the regularization effect of the solution so that the rescaled solutions resemble the formal blow-up limit at $r\to 0$. There are significant challenges in the analysis that differs from the injection problem. First, the blow-up argument would require that $\tilde{u}_r$ has linear growth at the free boundary, which we do not know a priori and must be shown along the way. For this reason a careful estimate on the rate of growth is carried out in Section \ref{S.3} for superharmonic functions near almost-Lipschitz boundaries.  Second, the inhomogeneous source term affects the directional monotonicity of $\tilde{u}_r$ in small scales: for instance $\e$-monotonicity does not improve to full monotonicity for superharmonic functions away from the free boundary, in contrast to harmonic functions. This makes the preservation of small-scale monotonicity significantly more difficult  as we zoom in.  We go around this by tracking a stronger notion of $\e$-monotonicity and its improvement away from the free boundary along the flow, in Section \ref{S.4} and Section \ref{S.6}. Lastly with non-constant drift, the cone-monotonicity of $\tilde{u}_r$ is further affected by the streamline generated by the drift.  It is not easy to deal with the full effect of the drift: this can be seen by the fact that a direct transformation of $u$ by the streamline generates error terms involving the Hessian of $u$, a potentially very singular term near the free boundary. Hence we instead proceed carefully with our blow-up argument with an approximate streamline, making sure that the regularization created by the pressure is sufficient for the iteration to preserve the  flow's geometry. 
 

\subsection{Statement of results and Outline of the paper}

For $r>0$, we denote
$
\calQ_r:=  B_r\times (-r,r). 
$ 
Let us state first the ``flat to Lipschitz'' result. 

\begin{thm}\lb{T.2.1}
Let $\vec{b}$ be a Lipschitz continuous vector field, and $f$ be a non-negative $\bar\gamma$-H\"{o}lder continuous function with $\bar\gamma\in(0,1)$, and for some $\eps\in(0,1)$, let $a_\eps\equiv 0$ if $f$ is constant and $a_\eps:=\eps^\alpha$ for some small $\al>0$ otherwise. Suppose that $u$ is a continuous viscosity solution to \eqref{1.1} in $\calQ_2$ satisfying
\begin{itemize}
    \item $u$ is $(\eps,a_\eps)$-monotone with respect to $W_{\theta,\mu}$ for some  $\theta\in (0,\frac{\pi}{2})$ and $\mu\in\bbS^{d-1}$,
    \smallskip
    
    \item $m:=\inf_{t\in(-2,2)}u(-\mu,t)>0$.
\end{itemize}
If $\frac{\pi}{2}-\theta$ and $\eps$ are small enough, then $u$ is non-decreasing along all directions of $W_{\theta',\mu}$ for some $\theta'\in (0,\theta)$ in $\calQ_1$. In particular, the free boundary $\Gamma_u(t)\cap B_1$ for each $t\in(-1,1)$ is a Lipschitz continuous graph. Here $\al_0$ only depends on $\bar\gamma$, and $\theta$ and $\theta'$ only depend on  $\bar\gamma$ and the dimension, and $\eps$ also on $m$, $\|u\|_{L^\infty(\calQ_2)}$, $\|\vec{b}\|_{C^1}$ and $\|f\|_{C^{\bar\gamma}}$. {In addition, when $\vec{b}$ is zero, and when $u$ solves \eqref{1.1} in $\R^d\times (-2,2)$, then the free boundary is $C^{1,\gamma}$ in $\calQ_1$ for some $\gamma\in (0,1)$. }
\end{thm}

{We refer to Corollary~\ref{C.6.5} and Remark thereafter for further discussion on the case of $\vec{b}=0$. }

\medskip


{The definition of the $(\e,a)$-monotonicity will be given in Definition \ref{D.2.1}. The $(\e,0)$-monotonicity corresponds to the usual $\e$-monotonicity, which quantifies the scale at which the solution is monotone along a direction. The additional parameter $a$ adds a growth condition at the same scale $\e$. This is to ensure that away from the boundary the solution is directionally monotone even with smaller scales. While $\e$-monotonicity is sufficient to guarantee such ``interior improvement'' for harmonic functions, it is not the case for general $f$: see Remark~\ref{remark_mon} for further discussions. Our condition is also natural. In Lemma \ref{L.4.5} we show that if the free boundary is known to be Lipschitz continuous, then the solution is monotone and satisfies $(\e,a)$-monotonicity near the free boundary for any small $\e$ and $a$.  }


\medskip

In the general setting, we state our non-degeneracy result. 

\begin{thm}\lb{T.2.2}
Under the assumption of Theorem \ref{T.2.1} and further assuming that $f$ is Lipschitz continuous, and
\begin{itemize}
    \item $u_t\geq \vec{b}\cdot\nabla u-Cu\text{ in } \calQ_2\text{ in the viscosity sense,}$ 
\smallskip
\item
$C^{-1}\leq \frac{u(-e_d,t)}{u(-e_d,0)}\leq C \text{ for all }t\in (-2,2)$ for some $C>0$, 
\end{itemize}
then if $\frac{\pi}{2}-\theta$ and $\eps$ are small enough, $u$ is non-degenerate in its positive set $\calQ_1$. In other words, $|\nabla u|$ is uniformly positive up to the free boundary. Here $\theta$ only depends on  $\bar\gamma$ and the dimension, and $\eps$ and the lower bound of $|\nabla u|$ also on $C,m$, $\|u\|_{L^\infty(\calQ_2)}$, $\|\vec{b}\|_{C^1}$ and $\|f\|_{C^{1}}$.
\end{thm}

Our assumption ensures that $u$ does not decrease too fast in the direction of the streamline generated by $\vec{b}$. This assumption holds for solutions of \eqref{1.1} posed in $\R^d\times (0,\infty)$ when $f$ and $\vec{b}$ are smooth, see Corollary 6.6 and Theorem 2.1 in \cite{chu2022}.


\begin{remark}
Our results apply to time-dependent $f$ and {$\vec{b}$} as well, even though we have only considered stationary ones for simplicity. With $f=f(x,t)$ and $\vec{b}=\vec{b}(x,t)$,  Theorem~\ref{T.2.1} continues to hold with straightforward modifications in the proof if $f$ and $\vec{b}$ are continuous in time. The same is true for Theorem~\ref{T.2.2} if $f$ and $\vec{b}$ are Lipschitz continuous in time. 

\end{remark}

Here is a brief outline of the paper. In Section \ref{S.2}, we introduce notations and preliminary properties. In Section \ref{S.3}, we prove several tools that will be used, including interior monotonicity and polynomial growth of superharmonic functions near the free boundary. In particular we obtain estimates on the growth rate of solutions near the free boundary {(Lemmas \ref{L.2.10} and \ref{C.2.8})}. Heuristically speaking, such growth rate translates into a strong elliptic effect, competing against the oscillations caused by the source and drift terms. We end Section \ref{S.3} with some examples discussing the optimality of our conditions and the waiting time phenomena.  Section \ref{S.4} is about superharmonic functions in Lipschitz domains. An important result is  Proposition \ref{L.2.61}, which compares superharmonic functions in a long strip domain with Lipschitz boundary. This boundary Harnack-type result enables us to compare our solutions to a localized harmonic function, ignoring the effect coming from the far-away regions. This result can be viewed as a generalized version of Dahlberg's lemma for harmonic functions, which was crucial for instance in showing that the interior improvement of the monotonicity.
Section \ref{S.5} introduces the sup-convolution and its properties. Finally, we give the proof of Theorem \ref{T.2.1} and Theorem \ref{T.2.2}, respectively, in Section \ref{S.6} and Section \ref{S.7}.

\medskip

\textbf{Acknowledgements.} I. Kim was partially supported by NSF grant DMS-2153254. We would like to thank Alp{\'a}r R. M{\'e}sz{\'a}ros for the initial discussions and helpful comments, motivating our investigation of this problem.

\medskip

\section{Preliminaries}\lb{S.2}

For a space-time function $u:\bbR^d\times [0,\infty)\to [0,\infty)$, we write 
\[
\Omega_u:=\{u(\cdot,\cdot)>0\},\quad \Omega_u(t):=\{u(\cdot,t)>0\},
\]
and
\[
\Gamma_u(t):=\partial \Omega_u(t),\quad \Gamma_u:=\bigcup_t\, \Gamma_u(t)\times\{t\} .
\]
Similarly, for a function $\omega:\bbR^d\to [0,\infty)$, we define
$$
\Omega_\omega:=\{\omega(\cdot)>0\}\quad\hbox{ and }\quad\Gamma_\omega:=\partial\Omega_\omega.
$$

Let us recall the notions of viscosity sub- and supersolutions to \eqref{1.1} from \cite{kim2003}, with trivial modifications due to the drift and source terms and reduced to continuous functions. 
Consider the domain $\Sigma:=D\times (0,T)$ with $T>0$ and $D\subseteq\bbR^d$ open and bounded. 


\begin{definition}\lb{D.21}
A non-negative continuous function $u$ defined in $\Sigma$ is a viscosity subsolution of \eqref{1.1} if for every $\phi\in C^{2,1}_{x,t}(\Sigma)$ such that $u-\phi$ has a local maximum in $\overline{\Omega_u}\cap\{t\leq t_0\}\cap\Sigma$ at $(x_0,t_0)$, then
\begin{align*}
    -(\Delta\phi+f)(x_0,t_0)\leq 0\quad &\text{ if } u(x_0,t_0)>0\\
    (\phi_t-|\nabla\phi|^2-\vec{b}\cdot\nabla\phi)(x_0,t_0)\leq 0\quad &\text{ if } (x_0,t_0)\in\Gamma_u \hbox{ and } -(\Delta\phi+f)(x_0,t_0)>0.
\end{align*}
\end{definition}

The reason for the intersection of the set $\overline{\Omega_u}$ in the definition is for the simple fact that there are no globally smooth function that crosses the solution from above at a free boundary point.

\begin{definition}
A non-negative continuous function $u$ defined in $\Sigma$ is a {\it viscosity supersolution} of \eqref{1.1} if for every $\phi\in C^{2,1}_{x,t}(\Sigma)$ such that $u-\phi$ has a local minimum in $\{t\leq t_0\}\cap\Sigma$ at $(x_0,t_0)$, then
\begin{align*}
    -(\Delta\phi+f)(x_0,t_0)\geq 0\quad &\text{ if } u(x_0,t_0)>0\\
    (\phi_t-|\nabla\phi|^2-\vec{b}\cdot\nabla\phi)(x_0,t_0)\geq 0\quad &\text{ if } (x_0,t_0)\in\Gamma_u,\, |\nabla\phi(x_0,t_0)|\neq 0 \text{ and }-(\Delta\phi+f)(x_0,t_0)<0.
 \end{align*}
\end{definition}

\begin{definition}
We say that a continuous non-negative function $u$ is a {\it viscosity solution} of \eqref{1.1} if $u$ is both a viscosity subsolution and a viscosity supersolution of \eqref{1.1}.
\end{definition}

To state the comparison principle, we need the following definition:
\begin{definition}
We say that a pair of functions $u_0,v_0:\overline{D}\to [0,\infty)$ are strictly separated (denoted by $u_0\prec v_0$) in $D$  if $u_0(x)<v_0(x)$ in $\overline{\Omega_{u_0}}\cap\overline{D}$. This says that the supports of the two functions are separated and in the support of the smaller function, the two functions are strictly ordered.
\end{definition}

Below we recall the comparison principle \cite{kim2, CJK}. 

\begin{lemma}\lb{L.cp}
Let $u,v$ be respectively viscosity sub- and supersolutions in $\Sigma=D\times (0,T)$ with initial data $u_0\prec v_0$ in $D$. In addition suppose that $\limsup_{t\to0^+}\Omega_u(t) = \Omega_{u_0}$. If $u\leq v$ on $\partial D\times (0,T)$ and $u<v$ on $(\partial D\times (0,T) )\cap \overline{\Omega_u}$, then $u(\cdot,t)\prec v(\cdot,t)$ in $D$ for all $t\in [0,T)$.
\end{lemma}


{Parallel argument as in Lemma 2.5 \cite{kim2} yields that the requirement at the free boundary in Definition \ref{D.21} can be simplified for testing against functions with nonzero gradient.}

\begin{lemma}\lb{L.2.66}
Let $u$ be a continuous viscosity subsolution of \eqref{1.1} in $\Sigma$, and $(x_0,t_0)\in\Gamma_u\cap\Sigma$. Let $\phi\in C^{2,1}_{x,t}(\Sigma)$ such that $u-\phi$ has a local maximum in $\overline{\Omega_u}\cap\{t\leq t_0\}\cap\Sigma$ at $(x_0,t_0)$ and $|D\phi(x_0,t_0)|\neq 0$. Then
\[
(\phi_t-|\nabla\phi|^2-\vec{b}\cdot\nabla\phi)(x_0,t_0)\leq 0.
\]
\end{lemma}


\subsection{Monotonicity assumption}

For two vectors $\nu,\mu \in  \bbR^d \setminus\{0\}$,  the angle between them is denoted as
\beq\lb{2.0}
\langle \nu,\mu\rangle:=\arccos\left(\frac{\nu\cdot\mu}{|\nu||\mu|}\right)\in [0,\pi].
\eeq
We denote a spacial cone to direction $\mu\in\mathbb{S}^{d-1}$ with opening $2\theta$ for $\theta\in [0,\frac\pi2]$ as
\begin{equation}\label{2.1}
W_{\theta,\mu} := \left\{p\in\mathbb{R}^{d}: \,\langle p,\mu\rangle \leq \theta\right\}.
\end{equation} 
Our basic hypothesis will be a monotonicity with respect to the cone $W_{\theta,\mu}$.

For a space-time function $u:\bbR^d\times [0,\infty)\to [0,\infty)$, we write 
\[
\Omega_u:=\{u(\cdot,\cdot)>0\},\quad \Omega_u(t):=\{u(\cdot,t)>0\},
\]
and
\[
\Gamma_u(t):=\partial \Omega_u(t),\quad \Gamma_u:=\bigcup_t\, \Gamma_u(t)\times\{t\} .
\]
Similarly, for a function $\omega:\bbR^d\to [0,\infty)$, we define
$$
\Omega_\omega:=\{\omega(\cdot)>0\}\quad\hbox{ and }\quad\Gamma_\omega:=\partial\Omega_\omega.
$$

\begin{definition}\lb{D.2.1}
Let $\Omega\subseteq\bbR^d$, $\theta\in [0,\frac\pi2]$, $\mu\in\bbS^{d-1}$, $\eps\in [0,1)$ and $a\geq 0$. 
We say that a continuous function $\omega:\Omega\to \bbR$  is {\it $(\eps,a)$-monotone} 
with respect to a cone $W_{\theta,\mu}$ in $D\subseteq\Omega$ if for every $\eps'\geq \eps$ and $x\in D$ we have
\[
(1+a\eps)\,\omega(x)\leq \inf_{y\in B_{\eps'\sin\theta}(x)\cap \,\Omega}\omega(y+\eps'\mu)
\]

\end{definition}

Here we need to assume that the solution also grows slightly in the monotone direction, which amounts to $(\e,\eps^\alpha)$-monotonicity, to reach the same conclusion (which is proved in Lemma \ref{L.2.4} and its remark): see Example \ref{E.3}, where the interior monotonicity fails with just $\e$-monotonicity. In Lemma \ref{L.4.5}, we show that if the free boundary is known to be Lipschitz continuous, then the solution is monotone and satisfying the $(\e,a)$-monotonicity for some $a>0$ and for any small $\eps>0$ near the free boundary.

\begin{remark}\label{remark_mon}


1. It is by now a well-known fact that the $(\eps,0)$-monotonicity of a positive harmonic function leads to full monotonicity in a smaller neighborhood, see for instance \cite[Corollary 11.16]{CafSal}. This fact is essential in the regularity analysis for solutions of \eqref{1.1} with $f=0$, since the stronger monotonicity in the positive set propagates to the free boundary so that its small-scale oscillation diminishes in unit time scale.  However when $f$ is present and $f$ is not a constant, this is not true. Indeed, in such cases, $\nabla_{\mu} \omega$ does not necessarily have a sign even if $\e$ is small compared to the $C^n$ norm of $f$ for any $n\geq 1$, see Example \ref{E.3}. Thus the assumption of $a\neq 0$ is sharp when $f$ is not a constant. With $(\e,a)$ monotonicity, the interior full monotonicity is shown in  Lemma \ref{L.2.4} and its remark.

\smallskip

2. If $f$ is a non-negative constant, our results hold even if $\alpha=\infty$. 
We refer readers to the remarks after Lemma \ref{L.2.4} for the detailed discussion.
\end{remark}

Below for any $\bar\gamma$-H\"{o}lder continuous function (with $\bar\gamma\in (0,1)$) $g:\Omega\to\bbR$ with $\Omega\subseteq\bbR^d$ an open set, we denote its $\bar\gamma$-H\"{o}lder seminorm and $\bar\gamma$-H\"{o}lder norm, respectively, as
\[
\|g\|_{C^{0,\bar\gamma}(\Omega)}:=\sup_{x,y\in\Omega,x\neq y}\frac{|g(x)-g(y)|}{|x-y|^{\bar\gamma}}\quad\text{and}\quad\|g\|_{C^{\bar\gamma}(\Omega)}:=\|g\|_{L^\infty(\Omega)}+\|g\|_{C^{0,\bar\gamma}(\Omega)}.
\]
When there is no ambiguity regarding the domain, we will drop $\Omega$ from the notations of $C^{0,\bar\gamma}(\Omega)$ and $C^{\bar\gamma}(\Omega)$, and we will simply write  
$\|g\|_\infty:=\|g\|_{L^\infty(\Omega)}$, and the Lipschitz constant $\|g\|_{\Lip}:=\|g\|_{C^{0,1}}$.

\subsection{Properties of harmonic and superharmonic functions}
First we recall the well-known Dahlberg lemma.

\begin{lemma}{~\rm (\cite{dah})}\lb{L.dah}
Let $\omega_1,\omega_2$ be two non-negative harmonic functions in a domain $D\subseteq \bbR^d$ of the form
\[
\left\{(x',x_d)\in \bbR^{d-1}\times\bbR\,:\, |x'|<2,\,|x_d|< 2\bar M,\, x_d<g(x')\right\}
\]
with $g\colon\R^{d-1}\to\R$ a Lipschitz function with Lipschitz constant less than $\bar{M}$ and $g(0)=0$. Assume further that $\omega_1=\omega_2=0$ along the graph of $g$. Then, there exists $C>1$ depending only on $d,\bar{M}$ such that
\[
\frac{1}{C}\leq \frac{\omega_1(x',x_d)}{\omega_2(x',x_d)}\cdot \frac{\omega_2(0,\bar M)}{\omega_1(0,\bar M)}\leq C
\]
in $\left\{(x',x_d)\,:\, |x'|<1,\,|x_d|<\bar M,\, x_d<g(x')\right\}$.
\end{lemma}

The following lemma follows from Dahlberg's Lemma and the explicit form of harmonic functions in a cone domain. While the proof is basic, we present it here given the importance of the constant $\theta_{\beta}$ in our analysis ($\theta'_\beta$ will only be used in Lemma \ref{L.4.5}). 

\begin{lemma}\lb{L.3.1}
For given $\theta\in (0,\pi)$, $\mu\in\bbS^{d-1}$, consider a harmonic function $\omega$ in $W_{\theta,\mu}\cap B_2$ such that $\sup_{ W_{\theta,\mu}\cap B_1 }\omega= 1$  and $\omega=0$ on $\partial W_{\theta,\mu}\cap B_2$. Then there exists $c\in (0,1)$ such that for any $\beta\in (1,2)$, there are $\theta_{\beta},\theta'_\beta\in (0,\frac{\pi}{2})$ (which are continuous and monotonely decreasing in $\beta\in(1,2)$, and converge to $\frac{\pi}{2}$ as $\beta\to  1$) such that  we have
\begin{equation}\label{angle_1}
\omega(s\mu)\geq c\,s^\beta\quad\text{ for all }s\in (0,1) \hbox{ if } \theta \geq \theta_{\beta}
\end{equation}
and
\begin{equation}\label{angle_2}
\omega(s\mu)\leq s^{2-\beta}/c \quad\text{ for all }s\in (0,1) \hbox{ if } \theta \leq \pi-\theta'_\beta.
\end{equation}

\end{lemma}
\begin{proof}
This result is a direct consequence of \cite[Theorem 1.1]{anc}. The theorem proves the existence of a harmonic function in $W_{\theta,\mu}$ such that it vanishes on the boundary of $W_{\theta,\mu}$. Moreover, the harmonic function is of the following form
\[
h(r\vartheta)=c\,r^{\beta_\theta}\varphi(\vartheta)
\]
where $c,r>0$, $\vartheta\in\Sigma_{\theta}$ with $\Sigma_{\theta}:=\bbS^{d-1}\cap W_{\theta,\mu}$, and $\varphi$ is a positive function in $\Sigma_{\theta}$ vanishing on $\partial \Sigma_{\theta}$. The constant $\beta_\theta>0$ is given by
\[
\beta_\theta:=\frac{-d+2+\sqrt{(d-2)^2+4\lambda_1(\Sigma_{\theta})}}{2}
\]
where $\lambda_1(\Sigma_{\theta})$ denotes the first eigenvalue of the opposite of the Dirichlet Laplacian in $\Sigma_{\theta}$, i.e.
\[
\lambda_1(\Sigma_\theta) = \inf\left\{\int_{\bbS^{d-1}} |\nabla u|^2 d\sigma\,:\, u\in C_c^1(\Sigma_\theta),\,\int_{\bbS^{d-1}}|u|^2d\sigma\geq 1\right\},
\]
with $\sigma$ the standard Riemannian spherical measure in $\bbS^{d-1}$.
It is not hard to see that $\beta_\theta$ is non-increasing in $\theta$, and $\beta_{\frac{\pi}{2}}=1$ (since $h=x\cdot\mu$ is a positive harmonic function in $W_{\frac{\pi}{2},\mu}$).
We refer readers to \cite{betz1983bounds} for several bounds of $\lambda_1(\Sigma_\theta)$.
Since $\lambda_1(\Sigma_{\theta})$ and $\beta_{\theta}$ depend continuously on $\theta$, $\beta_\theta$ can be arbitrarily close to $1$ if $\theta$ is large (close to $\frac{\pi}{2}$). The conclusions follow immediately from Harnack's inequality and Dahlberg's lemma.
\end{proof}

\begin{remark}
When $d=2$ the formula can be written as
\[
\theta_{\beta}=\frac{\pi}{2\beta}\quad\text{and}\quad \theta'_\beta=\max\left[\pi-\frac{\pi}{2(2-\beta)},0\right].
\]
In particular one can deduce that $\theta_{\beta} \geq \theta_{2} \geq \frac\pi4$ when $d\geq 2$, by comparison principle for harmonic functions.
\end{remark}

Next we show some properties of superharmonic functions.

\begin{lemma}\lb{l.2.2}
Let $f:\bbR^d\to [0,\infty)$ be 
continuous, let $r>0$ and let $\omega:\overline{B_{2 r}}\to [0,\infty)$, $\omega\in C^2(\overline{B_{2 r}})$  be a classical solution to 
$$-\Delta \omega=f,\ \ {\rm{in}}\ B_{2 r}.$$
Then there exists a constant $C>0$, depending only on the dimension, such that for all $x\in B_{r}$,
\begin{align*}
    \omega(x)\leq C\omega(0)+Cr^2\|f\|_{L^\infty(B_{2r})},\quad\quad
    |\nabla \omega(x)|\leq Cr^{-1}\omega(0)+ Cr\| f\|_{L^\infty(B_{2r})}.
\end{align*}

Moreover if $\nabla_\mu\omega\geq 0$ for some $\mu\in\bbS^{d-1}$ in $B_{2r}$, then for all $x\in B_r$,
\[
\nabla_\mu\omega(x)\leq C\nabla_\mu\omega(0)+Cr\|f\|_{L^\infty(B_{2r})}.
\]
\end{lemma}

\begin{proof}
Set $\tilde{\omega}(x):=\omega(rx)$, and so
$-\Delta \tilde\omega=\tilde f$ in $B_2$ with $\tilde f(x):=r^2f(rx)$.
Let $G$ be the Green's function of Laplacian in $B_{2}$. Then, we have the representation formula (see e.g., \cite{evans})
\beq\lb{2.31}
\tilde\omega(x)=-\int_{\partial B_{2}}\tilde\omega(y) \partial_nG(x,y)d\sigma(y)+\int_{B_{2 }}\tilde f(y)G(x,y)dy,
\eeq
where $n$ denotes the outward pointing unit normal to $\partial B_{1}$. Notice that there exists $C=C(d)>0$ such that 
\beq\lb{21.1}
\sup_{x\in B_1}\left(\int_{B_{2 }}G(x,z)dz+\int_{B_{2 }}|\nabla_x G(x,z)|dz\right)\leq C,
\eeq
and
$
0<-\partial_nG(x,y)\leq  -C\partial_nG(0,y)
$  for $(x,y)\in B_1\times\partial B_{2 }$.
Therefore, also using that $\omega\geq 0$ and \eqref{2.31} with $x=0$, we get for $x\in B_1$ that
\begin{align*}
\tilde\omega(x)&\leq -C\int_{\partial B_{2}}\tilde\omega(y) \partial_nG(0,y)d\sigma(y)+C\int_{B_{2 }}\tilde f(y)G(0,y)dy+(C+1)\sup_{B_1}\|\tilde f\|_{L^{\infty}(B_{1})}\int_{B_{ 2}}G(\cdot,y)dy\\
&\leq C\tilde\omega(0)+C(C+1)\|\tilde f\|_{L^{\infty}(B_{2})}
\end{align*}
By rewriting this estimate for $\omega$ and $f$, this yields the first inequality of the conclusion after enlarging $C$.

Next since $|\nabla \partial_nG(x,y)|\leq -C\partial_nG(0,y)$ for any $(x,y)\in B_{1}\times \partial B_{2 }$, taking derivatives on both sides of \eqref{2.31} yields
\begin{align*}
|\nabla \tilde\omega(x)| &\leq -C\int_{\partial B_{2 }}\tilde\omega(y) \partial_nG(0,y)d\sigma(y)
+\left|\int_{B_{2}} \tilde f(y)\nabla_x G(x,y)dy\right|\leq C\tilde{\omega}(0)+C\| \tilde f\|_{L^{\infty}(B_{2})}
\end{align*}
where in the last inequality we used \eqref{2.31} with $x=0$ and \eqref{21.1}.
This then implies the second inequality, again, by using the definition of $\tilde\omega$ and $\tilde f$.

For the last claim, without loss of generality, we assume that $\omega$ is $C^2$ in a neighbourhood of $B_{2r}$. Taking derivatives on both sides of \eqref{2.31} and using $\nabla_\mu\omega\geq 0$ yield
\begin{align*}
\nabla_\mu \tilde\omega(x) &\leq -\int_{\partial B_{2 }} \nabla_\mu\tilde\omega(y)\partial_n G(x,y)d\sigma(y)+\int_{B_{2 }}\tilde f(y)|\nabla G(x,y)|dy\\
&\leq -C\int_{\partial B_{2 }} \nabla_\mu\tilde\omega(y)\partial_n G(0,y)d\sigma(y)+C\|\tilde{f}\|_{L^{\infty}(B_{2})}\leq C\nabla_\mu\tilde{\omega}(0)+C\| \tilde f\|_{L^{\infty}(B_{2})}
\end{align*}
which implies the last inequality.
\end{proof}

\section{Monotonicity properties, Streamlines, and Examples}\lb{S.3}

In this section, we prove several tools that will be used to prove the main theorems, and we discuss by examples the optimality of our monotonicity assumptions and the formation of cusps on the free boundary.  

\subsection{Interior monotonicity}

The goal of this section is to show that if a superharmonic function is $(\eps,\eps^\alpha)$-monotone, then under some assumptions it is fully monotone in the interior. 

The {corresponding} result with $f\equiv 0$, $\vec{b}\equiv 0$, and $(\eps,0)$-monotonicity is proved in the book of Caffarelli and Salsa \cite[Corollary 11.16]{CafSal}. Here we need $\eps^\alpha$ to be positive to compensate the possible loss of monotonicity caused the source function $f$. 
One important ingredient of the proof in \cite[Corollary 11.16]{CafSal} is the Harnack inequality, which is applied to
$h:=\omega(x)-\omega(x-\eps\mu)$. However when $f\neq 0$, $h$ solves a Poisson equation with the source term $f(x-\eps\mu)-f(x)$ which can be negative at some points, and in such cases the Harnack inequality might fail (because for example, $h:=x^2$ solves $-\Delta h=-2$ and $h(x)\geq0$ with equality holds if and only if $x=0$). To overcome the problem, we estimate carefully the ``error'' from the source term in the lemma below. We will later combine this lemma with Lemma \ref{L.2.10}, which provides a lower bound for $\omega$, to conclude the interior monotonicity.
Below we use the convention that $\eps^\infty=0$ for $\eps\in (0,1)$.



\begin{lemma}\lb{L.2.4}
Let $f\geq 0$ be $\bar\gamma$-H\"{o}lder continuous on $\overline{B_1}$ for some $\bar\gamma\in(0,1)$, and $\alpha\in [0,\infty]$ and $\eps,\kappa_1\in (0,1)$. There exists $C=C(d)>0$ such that the following holds for all $\eps$ small enough (depending only on $d,\alpha,\kappa_1$). If $\omega$ is a non-negative solution to $
-\Delta \omega=f$ in $B_{\eps^{1-\kappa_1}}$, and $\omega$ is $(\eps,\eps^\alpha)$-monotone with respect to $W_{0,\mu}$ for $\mu\in\bbS^{d-1}$,
then
\[
\nabla_\mu\omega(x)\geq \eps^\alpha(1-C\eps^{\kappa_1})\omega(x)-C\eps^{1+{\bar\gamma}-\kappa_1}\|f\|_{C^{0,\bar\gamma}(B_1)}\quad{\mathrm{for\ all\ }}x\in B_{\eps}. 
\]

\end{lemma}
\begin{proof}
Let us denote
$\delta:=\eps^{\alpha+1}<1$. We will only show the conclusion for $x=0$, and the general case of $x\in B_\eps$ follows the same. For $s\in [\eps,2\eps]$, define
\beq\lb{21.2}
h_s(x):=\omega(x+s\mu)-(1+\delta)\omega(x),
\eeq
and it follows from the $(\eps,\eps^\alpha)$-monotonicity assumption that 
$h_s\geq 0$. 
Using the monotonicity again yields for $s\in [\eps, 2\eps]$, 
\beq\lb{2.11}
\begin{aligned}
\sum_{i=0}^2 (1+\delta)^{-i} h_\eps(x+i\eps\mu)&= (1+\delta)^{-2}\omega(x+3\eps\mu)-(1+\delta)\omega(x)\\
&\geq (1+\delta)^{{-1}}\omega(x+s\mu)-(1+\delta)\omega(x)\geq (1+\delta)^{-1}h_s(x)-{\delta}\omega(x).
\end{aligned}
\eeq
Note that $-\Delta h_s=(1+\delta)f(\cdot)-f(\cdot+s\mu)$ and 
\[
|(1+\delta)f(\cdot)-f(\cdot+s\mu)|\leq \delta\|f\|_\infty+s^{\bar\gamma}\| f\|_{C^{0,\bar\gamma}}.
\]
Hence $h_s\geq 0$ and
Lemma \ref{l.2.2} (after shifting $0$ to any $y\in B_{3\eps}$) yield for some $C >0$ (if $\eps$ is small) and any $s\in [\eps,2\eps]$ that
\beq\lb{2.10}
h_s(x)\leq Ch_s(y)+C\eps^2
\delta\|f\|_{\infty}+C\eps^{2+\bar\gamma}\| f\|_{C^{0,\bar\gamma}}\quad \text{ for all }x,y\in B_{3\eps}.
\eeq
This and \eqref{2.11} with $x=0$ yield
\begin{align*}
    h_s(0)&\leq C \sum_{i=0}^2h_\eps(i\eps\mu)+C\delta\omega(0)+C\eps^2
\delta\|f\|_{\infty}+C\eps^{2+{\bar\gamma}}\| f\|_{C^{0,\bar\gamma}}\\
&\leq  C h_\eps(0)+C\delta\omega(0)+C\eps^2
\delta\|f\|_{\infty}+C\eps^{2+{\bar\gamma}}\| f\|_{C^{0,\bar\gamma}}.
\end{align*}
Next, by Lemma \ref{l.2.2} again, for $s\in[\eps,2\eps]$ and $r:=\frac12\eps^{1-\kappa_1}$ with $\kappa_1\in(0,1)$, we get
\beq\lb{2.12}
\begin{aligned}
|\nabla h_s(0)|&\leq Cr^{-1}h_s(0)+Cr\delta\|  f\|_{\infty}+C r\eps^{\bar\gamma}\| f\|_{C^{0,\bar\gamma}}\\
&\leq Cr^{-1}h_\eps(0)+Cr^{-1}\delta\omega(0)+Cr\delta\|  f\|_{\infty}+C r\eps^{\bar\gamma}\| f\|_{C^{0,\bar\gamma}}.
\end{aligned}
\eeq

Now we estimate $h_\eps(0)$.
We obtain from \eqref{21.2} and \eqref{2.10} with $s=\eps$ that
\begin{align*}
    h_\eps(0)&\leq C(\omega(2\eps\mu)-(1+\delta)\omega(\eps\mu))+c_{\eps,f},
\end{align*}
where $c_{\eps,f}:=C\eps^2
\delta\|f\|_{\infty}+C\eps^{2+{\bar\gamma}}\| f\|_{C^{0,\bar\gamma}}$.
Since $\nabla_\mu h_s(0)=\nabla_\mu \omega(s\mu)-(1+\delta) \nabla_\mu \omega(0)$, this implies
\beq\lb{2.22}
\begin{aligned} 
   h_\eps (0)&\leq C\left(\int_{\eps}^{2\eps}\nabla_\mu\omega(s\mu)ds-\delta\omega(\eps\mu)\right)+c_{\eps,f}\\
     &\leq C\left(\int_{\eps}^{2\eps} |\nabla h_s(0)|ds+\eps(1+\delta) \nabla_\mu\omega(0)-(1+\delta)\delta\omega(0)\right)+c_{\eps,f},
\end{aligned}
\eeq
where we also used $\omega(\eps\mu)\geq (1+\delta)\omega(0)$. Then by \eqref{2.12} with $r=\frac12\eps^{1-\kappa_1}$ and the definitions of $c_{\eps,f}$ and $\delta$, we obtain for some $C=C(d)>0$,
\begin{align*} 
   h_\eps(0)\leq C\eps^{\kappa_1} h_\eps(0)+C\eps^{1+\kappa_1+\alpha} \omega(0)+C(1+\delta)\left(\eps \nabla_\mu\omega(0)-\delta\omega(0)\right)+C\eps^{2+{\bar\gamma}-\kappa_1}(\eps^\alpha\|f\|_{\infty}+\| f\|_{C^{0,\bar\gamma}}).
\end{align*}
Using $h_\eps\geq 0$, the above estimate yields for all $\eps>0$ small enough,
\beq\lb{2.13}
\nabla_\mu\omega(0)\geq \eps^\alpha(1-C\eps^{\kappa_1})\omega(0)-C(1+\eps^\alpha)\eps^{1+{\bar\gamma}-\kappa_1}\|f\|_{C^{0,\bar\gamma}}.
\eeq
This yields the conclusion for $x=0$.
\end{proof}



\begin{remark}\lb{R.3.2}
1. It is clear that $(\eps,\eps^\alpha)$-monotonicity with respect to $W_{\theta,\mu}$ for some $\theta\geq 0$ and $\mu\in\bbS^{d-1}$ implies $(\eps,\eps^{\alpha'})$-monotonicity with respect to $W_{0,\mu}$ for all $0\leq \alpha\leq\alpha'$.

\smallskip

2. Let $\omega$ be from Lemma \ref{L.2.4}. If either $\alpha\neq\infty$ or $f$ is constant, and $\eps>0$ is small enough such that $\eps^\alpha\omega(\cdot) \geq 2C\eps^{1+{\bar\gamma}-\kappa_1}\|f\|_{C^{0,\bar\gamma}}$ and $C\eps^{\kappa_1}<\frac12$, then $ \omega(s \mu)$ is non-decreasing in $s$ for all $s\in (-\eps,\eps)$.

\smallskip

3. Furthermore, if $\eps^\alpha\omega(\cdot)\geq C \eps^{1+{\bar\gamma}-2\kappa_1}\|f\|_{C^{0,\bar\gamma}}$ in $B_{2\eps^{1-\kappa_1}}$, 
then for some larger $C>0$ and any $j \in (0,1)$, $\omega$ is $(j \eps,\eps^\alpha(1-C\eps^{\kappa_1}))$-monotone with respect to $W_{0,\mu}$ in $B_{\eps^{1-\kappa_1}}$. 

\end{remark}

\subsection{Polynomial growth near the free boundary}

The goal of this section is to show that a superharmonic function which has cone monotonicity up to $\e$-scale has a polynomial growth bound up to the same scale.  The growth rate lower bound will be used in competition to the irregularity of the source term, to show that the regularity propagates to the boundary over time (see Lemma \ref{L.5.2}, Proposition \ref{L.2.12} and Theorem \ref{T.6.1}).
This bound can be improved to a linear rate once we obtain full monotonicity,  later in Section  \ref{S.7}.



Next lemma provides a lower bound for the growth rate of  $(\eps,0)$-monotone superharmonic functions. 
\begin{lemma}\lb{L.2.10}
Let $\mu\in\bbS^{d-1}$, and let $\omega\geq 0$ be a continuous function in $B_2$ such that 
\[
-\Delta\omega\geq 0\text{ in }\Omega_\omega\cap B_2,\quad 0\in\Gamma_\omega=\partial\Omega_\omega,\quad \omega(\mu)\geq 1,
\]
and $\omega$ is $(\eps,0)$-monotone with respect to $W_{\theta,\mu}$ in $B_2$ for some $\eps$ small enough. Then for some dimensional constant $c>0$ and for any $\beta\in (1,2)$, if $ \theta\geq \theta_{\beta}$ (with $\theta_{\beta}$ given in Lemma \ref{L.3.1}) we have
\[
\omega(x)\geq c\,d(x,\Gamma_\omega)^\beta 
\]
for all $x\in B_{1}\cap\Omega_\omega$ satisfying $d(x,\Gamma_\omega)\geq 2\eps$.
\end{lemma}

\begin{proof}
For each $x\in B_{1}\cap\Omega_\omega$ satisfying $d(x,\Gamma_\omega)\geq 2\eps$, there is $x_0\in \Gamma_\omega\cap B_2$ such that $x= x_0+s\mu$ with $s\geq d(x,\Gamma_\omega)\geq 2\eps$. Note that it follows from the monotonicity assumption and $\{0,x_0\}\subset\Gamma_\omega$ that $\omega>0$ in $((x_0+\eps\mu+W_{\theta,\mu})\cup(\eps\mu+W_{\theta,\mu}))\cap B_2$. Thus Harnack's inequality and $\omega(\mu)\geq 1$ yield $\omega(x+\frac{1}{2}\mu)\geq c$ for some dimensional constant $c>0$.  
Then by comparing $\omega$ with a non-negative harmonic function whose support is $x_0+\eps\mu+W_{\theta,\mu}$, Lemma \ref{L.dah} and Lemma \ref{L.3.1} yield for some dimensional $c'>0$ we have 
\[
\omega(x)\geq c' (s-\eps)^\beta\geq 4^{-1}c' d(x,\Gamma_\omega)^\beta\quad \text{ whenever }\theta\geq \theta_{\beta}.
\]
\end{proof}



For the next lemma the growth rate bound is obtained excluding only a small portion of the original domain $B_1$, with the expanse of restricting to the near boundary region. 

\begin{lemma}\lb{C.2.8}
Under the assumptions of Lemma \ref{L.2.10} except that $\omega$ is only assumed to be $(\eps,0)$-monotone with respect to $W_{\theta,\mu}$ in $B_1$ (instead of $B_2$), then for some $c=c(d)>0$ and for any $\beta\in (1,2)$, if $ \theta\geq \theta_{\beta}$ we have
\[
\omega(x)\geq c\,d(x,\Gamma_\omega)^\beta 
\]
for all $x\in B_{1-\eps^{1/2}}\cap\Omega_\omega$ satisfying $d(x,\Gamma_\omega) \in [2\eps,\eps^\frac12]$. 
\end{lemma}
\begin{proof}
For any $x\in B_{1-\eps^{1/2}}\cap\Omega_\omega$ satisfying $d(x,\Gamma_\omega) \in [2\eps,\eps^\frac12]$, there exists $x_0\in  \Gamma_\omega\cap B_1$ such that $x=x_0+s\mu$ with $s\geq 2\eps$. Note that this is not true if $d(x,\Gamma_{\omega})>>\eps^{1/2}$. With this $x_0\in  \Gamma_\omega\cap B_1$, we can conclude the proof the same as in Lemma \ref{L.2.10}. 
\end{proof}




\subsection{Streamlines}
Here we introduce \textit{streamlines} associated with the drift term, which yields an important monotonicity property for our flow. They are defined as the unique solution $X(t;x_0)$ of the ODE
\begin{equation}\label{ode}
\left\{\begin{aligned}
    &\partial_t X(t;x_0)={-}\vec{b}(X(t;x_0)), \quad t\in \bbR,\\
    &X(0;x_0)=x_0.
    \end{aligned}\right.
\end{equation}
We write $X(t):=X(t;0)$. In order to analyze the solution along one streamline that passes through $(0,0)$, we define 
\beq\lb{3.0}
\bar{u}(x,t):=u(x+X(t),t).
\eeq
Then $\bar{u}$ satisfies
\begin{equation}\lb{3.1}
    \left\{\begin{aligned}
        -\Delta \bar u &=f_0(x,t)\quad  &&\text{ in }\{\bar u>0\},\\
        \bar u_t&=|\nabla \bar u|^2+\vec{b}_0(x,t)\cdot\nabla \bar u\quad  &&\text{ on }\partial\{\bar u>0\},
    \end{aligned}\right.
\end{equation}
where 
\beq\lb{3.1'}
f_0(x,t):=f(x+X(t)),\quad \vec{b}_0(x,t):=\vec{b}(x+X(t))-\vec{b}(X(t)).
\eeq


It was shown in \cite[Lemma 3.5]{kimsingular} for the drift porous medium equation that $\{u>0\}=:\Omega_u$ is non-decreasing along the streamlines. 
The same holds in our case. 


\begin{lemma}\lb{L.2.9}
If $(x_0,t_0)\in \Omega_u$, then $(X(t;x_0),t+t_0)\in\Omega_u$ for all $t>0$. 
\end{lemma}
\begin{proof}
Let us assume $(x_0,t_0)=(0,0)$. By continuity of the solution, suppose that for some $z\in B_1$ we have $u(z,t)\geq c>0$ for all $t\in [0,\tau)$ with some small $\tau>0$. Let $D_0$ be any strict open subset of $\Omega_u(0)\cap B_1$, and then for $t\in (0,\tau)$ define
\[
D_t:=\{X(t;x)\,:\, x\in D_0\}\cap B_1.
\]
We can assume that $z\in D_t$ for $t\in [0,\tau]$.
Let $v(\cdot,t)$ be the largest subharmonic function in $D_t\backslash\{z\}$ such that $v(\cdot,t)=0$ on $\partial D_t$ and $v(z,t)= c$. It is clear that $v\prec u$ at $t=0$ and $v< u$ on $(\overline{\Omega_u(t)}\cap \partial B_1)\cup\{z\}$ for $t\in (0,\tau)$.

We claim that that $v$ is a viscosity subsolution to \eqref{1.1} in $(B_1\backslash\{z\})\times (0,\tau)$. Let us only verify the free boundary condition. Suppose for a smooth function $\phi\in C_{x,t}^{2,1}$ such that $v-\phi $ has a local maximum in $\overline{\Omega_v}\cap\{t\leq t_0\}$ that equals to $0$ at $(x_0,t_0)\in\Gamma_v$ and $x_0\notin \partial B_1$. Note that by the definition of $D_t$, $\phi(x_0,t_0)\leq \phi(X(-\eps;x_0),t_0-\eps)$ for all $\eps$ sufficiently small. 
Therefore 
$\phi_t\leq \vec{b}\cdot\nabla\phi$ at $(x_0,t_0)$, and thus we can conclude with the claim. 

Then the comparison principle (Lemma \ref{L.cp}) yields $v\leq u$. Note that $\{(x,t)\,:\,x\in D_t,t\in [0,\tau)\}$ is non-decreasing along streamlines and $D_0$ can be arbitrarily close to $\Omega_u(0)\cap B_1$. So $\Omega_u$ is non-decreasing along streamlines for $t\in (0,\tau)$, and then the same holds for all positive time.
\end{proof}

\subsection{Lipschitz space-time neighborhood of the free boundary}
In this subsection we show that if the solution $u$ to \eqref{1.1} is $(\eps,0)$-monotone in space, then there exists a Lipschitz space-time neighborhood of the free boundary of $u$. The interesting feature lies in the time variable component of the Lemma: for the space variable it can be derived from a geometric argument, for instance see Proposition 11.14 in \cite{CafSal}. 
This Lipschitz set will be used as the region where we do comparison later. 
For simplicity of discussions, we take $\mu:=-e_d$ below.

\begin{lemma}\lb{L.3.11}
Suppose $u, f,\vec{b}$ satisfy \eqref{1.1}, and they are uniformly bounded by $L$ in $\calQ_2=B_2\times (-2,2)$ for some $L\geq 1$. If $u$ is $(\eps,0)$-monotone with respect to $W_{\theta,-e_d}$ for some $\theta\in (0,\frac{\pi}{2})$ in $\calQ_2$, then for any $r\in [4\eps,\frac14]$ there exists a Lipschitz continuous function $\Phi_r:\bbR^{d}\to\bbR$ such that
\[
\Gamma_u(t) \cap B_{3/2}\subseteq \{(x',x_d)\in B_{3/2}\,:\, |\Phi_r(x',t)-x_d|<r\}
\]
for all $t\in (-2,2)$. Moreover, $\Phi_r$ is $\cot\theta$-Lipschitz continuous in space and $C/r$-Lipschitz continuous in time for some $C=C(L,\theta)>0$.

\end{lemma}
\begin{proof}
From the $(\eps,0)$-monotonicity assumption, it follows from Proposition 11.14 \cite{CafSal} that for each $t\in (-2,2)$, $\Gamma_u(t)$ is contained in a $(1-\sin\theta)\eps$-neighborhood of the graph of a Lipschitz function, with Lipschitz constant $\cot\theta$. Therefore we can find a Lipschitz function $\phi^t:\bbR^{d-1}\to\bbR$ with the same Lipschitz constant such that
\beq\lb{2.93}
\Gamma_u(t)\cap B_2\subseteq \{(x',x_d)\in B_2\,:\, |\phi_t(x')-x_d|<\eps\}.
\eeq

{\it Claim. If $r\in (0, \frac14]$ and $u(\cdot,t_0)=0$ in $B_{r}(x_0)$ for some $(x_0,t_0)\in  B_{3/2}\times (-2,2)$, then $u(x_0,t_0+t)=0$ for all $t\leq cr^2$ for some $c=c(L)>0$. }

{\it Proof of claim.} We use a barrier argument to prove the claim for $d\geq 3$ (the proof for $d=2$ is similar). Also suppose, without loss of generality, that $t_0=0$ and $x_0=0$. For some $A\geq 1$ to be determined, let 
\[
 w(x,t):=a_t-2^{-1}L|x|^2-b_t|x|^{2-d}\quad \text{ in }\Sigma:=\{(x,t)\,:\,x \in B_{1/2}\backslash B_{r_t},\, t\in [0,r^2/(2A)]\}
\]
where
\[
a_t:=1+8^{-1}L+b_t2^{d-2},\quad b_t:=\frac{8+L-4Lr_t^2}{8r_t^{2-d}-2^{d+1}},\quad r_t:=r-Ar^{-1}t.
\]
Then it is straightforward to verify that for $t\in [0,(2A)^{-1}r^2]$, $-\Delta w(\cdot,t)=d L$, $w(\cdot,t)={1}$ on $\partial B_{1/2}$ and $w(\cdot,t)=0$ on $\partial B_{r_t}$. Moreover for these $t$, 
\[
|\nabla w(x,t)|\leq L|x|+(d-2)b_t|x|^{1-d}\leq C/r\quad \text{ for }x\in B_{1/2}\backslash B_{r_t},
\]
as $1/2<1/r$, with $C>0$ only depending on $d,L$. Therefore, using that $\frac{d}{dt}r_t=-Ar^{-1}$ and by picking $A:=C+L$, we get that $w$ is a supersolution to \eqref{1.1}. So the assumptions and the comparison principle yield $u\leq w$ in $\Sigma$. Since $w(\cdot,t)=0$ on $\partial B_{r_t}$ for all $t\in [0,(2A)^{-1}r^2]$, we proved the claim with $c:=(2A)^{-1}$.

Now  for each $x'\in \bbR^{d-1}$ satisfying $|x'|\leq \frac32$, since $u( (x',\phi_t(x')+\eps),t)=0$, the $(\eps,0)$-monotonicity yields $u(\cdot,t)=0$ in $B_{r\sin\theta}((x',\phi_t(x')+r+\eps))$ for all $r\geq\eps$. 
Hence the above claim implies
\[
u\left((x',\phi_t(x')+r+\eps),t+s\right)=0\quad\text{ for all }s\in [0,c_\theta r^2]
\]
where $c_\theta:=c\sin\theta $. 
This yields
\beq\lb{2.91}
\phi_{t+s}(x')\leq \phi_t(x')+r+2\eps\quad\text{ for all }s\in[0,c_\theta r^2].
\eeq
On the other hand, since $\Omega_u$ is non-decreasing along streamlines and $|\vec{b}|\leq L$, we obtain
\beq\lb{2.92}
\phi_{t+ s}(x')\geq \phi_t(x')-c_\theta Lr^2-2\eps \quad\text{ for all }s\in[0,c_\theta r^2].
\eeq

Let $r\in [\eps,\frac{1}{4}]$
and we use $\phi_t$ to construct a Lipschitz space-time function $\Phi_r$. Let $t_0:=-2$, and define iteratively for $k\in\bbN$ that $t_k:=t_0+kc_\theta r^2$, and $\Phi_r(x',t_k):=\phi_{t_k}(x')$. Then we extend $\Phi_r(x',\cdot)$ to all $t\in (-2,2)$ by linear interpolation. We see that $\Phi_r$ is $\cot\theta$-Lipstchiz continuous in space and $2(c_\theta r)^{-1}$-Lipschitz continuous in time.
Finally, \eqref{2.93}, \eqref{2.91} and \eqref{2.92} yield that
\[
\Gamma_u(t) \cap B_{3/2}\subseteq \{(x',x_d)\in B_{3/2}\,:\, |\Phi_r(x',t)-x_d|<r+3\eps\}
\]
which finishes the proof with $r+3\eps$ in place of $r$.
\end{proof}

\subsection{Examples: Waiting time, Formation of cusps and discussion of optimality for the monotonicity assumption}  


First let us show an example of  the {\it waiting time} phenomena, where the initial free boundary and stays as a sharp cone with fixed vertex for a finite amount of time. 

\begin{example}\lb{cone_sharp} 
Let  $f=1$ and $\vec{b}\equiv 0$. We only sketch the proof for $d=2$, by constructing an increasing-in-time supersolution whose free boundary has one point that does not move for a short positive time.

For some $k>\frac{3}{2}$ and $c>0$, and for $t<k/c$, consider
\[
\Phi(x,t)=\left(x_2^2-k x_1^2+ctx_1^2\right)_+ \text{ if }x_2\leq 0, \text{ otherwise }0.
\]
Then when $t=0$, the support of $\Phi$ is a cone of angle $(\pi- 2\arctan \sqrt{k})$, and as time increases the cone enlarges but always with the origin as its vertex.
Indeed, the free boundary $\Gamma_\Phi(t)=\{(x,t)\,|\, x_2=-\sqrt{k-ct}|x_1|\}$. Now we show that $\Phi$ is a viscosity supersolution. Direct computation yields on $\Gamma_\Phi(t)\backslash\{0\}$,
\[
\Phi_t-|\nabla\Phi|^2=cx_1^2-4x_2^2-4(k-ct)^2x_1^2=\left(c+1-(2k-2ct+1)^2\right)x_1^2\geq 0.
\]
While at the origin, suppose $\Phi-\phi$ with $\phi$ a test function has a local minimum for $t\leq t_0$ at $(0,t_0)$. Since $\Phi(x,t_0)=0$ in a cone of angle $>\pi$, then $\phi_t(0,t_0)\leq 0$ and $|\nabla \phi(0,t_0)|=0$. We checked the free boundary condition.
Next, inside the support of $\Phi$, we have
\[
-\Delta\Phi=-2+2k-2ct\geq 1=f \text{ if }t\leq (2k-3)/(2c).
\]
Thus we conclude that $\Phi$ is a viscosity supersolution in $\R^d\times [0,\frac{2k-3}{2c})$.
\end{example}

In the next example, we show that the free boundary of solutions starting with a cone as its positive set 
develops a cusp at the vertex of the cone if the vector field is only H\"{o}lder continuous. 
\begin{example}\lb{E.1}
We only consider space dimension $2$ and we use the polar coordinates $r,\theta$ such that $(x_1,x_2)=(r\cos\theta,r\sin\theta)$. 
In the example we take $f\equiv0$, and $\vec{b}$ to be of the form $\vec{b}=(C_0|x_2|^{\gamma_0-1},0)$ with $C_0>1$ and $\gamma_0\in (1,2)$.

First we show that the support of the solution is contained in a shrinking cone when $C_0$ is large.
For $t\in [0,1]$, let
\[
\Gamma_t':=\{|\theta|=\theta_t\}\quad\text{ where }\theta_t:=(1-t)\frac\pi{2\gamma_0}+t\frac{\pi}{2\gamma_1}\in (0,\frac{\pi}{2})\text{ and }\gamma_1>\gamma_0.
\]
The opening of the cones $\{|\theta|<\theta_t\}$ shrinks from $\theta_0$ to $\theta_1$ for $t\in [0,1]$. For each $t$, let $\varphi^t=r^{\gamma_t}(\cos(\gamma_t\theta))_+$ with $\gamma_t:=\frac{\pi}{2\theta_t}>1$. It is easy to see that $\Delta \varphi^t=0$ in $\{\varphi^t>0\}=\{|\theta|<\theta_t\}$, and
\[
|\nabla \varphi^t|=r^{\gamma_t-1}\sqrt{\cos^2(\gamma_t\theta)+{\gamma_t}^2\sin^2({\gamma_t}\theta)}\Big|_{|\theta|=\theta_t}={\gamma_t} r^{{\gamma_t}-1}\quad\text{ on $|\theta|=\theta_t$. }
\]
By direct computations, the outer normal direction of $\Gamma_t'$ is $\nu_t'=(-\sin\theta_t,\pm\cos\theta_t)$, and  the normal velocity of $\Gamma_t'$ at $(r,\pm\theta_t)$ equals to
$
V'(r,\pm\theta_t)=-(\frac{\pi}{2\gamma_0}-\frac{\pi}{2\gamma_1})r.
$
We obtain on $\Gamma_t'\cap B_1$,
\[
V'-|\nabla\varphi^t|-\vec{b}\cdot\nu_t'=-(\frac{\pi}{2\gamma_0}-\frac{\pi}{2\gamma_1})r-\gamma_tr^{\gamma_t-1}+C_0|r\cos\theta_t|^{\gamma_0-1}\sin\theta_t 
\]
which is non-negative if $C_0$ is large enough, due to $\gamma_0\leq\gamma_t$.
So $\varphi^t$ is a supersolution to \eqref{1.1} in $B_1\times (0,1)$.

Next let $u$ be a solution with initial data $\leq \varphi^0$ and with boundary value $\leq\varphi^t$ on $\partial B_1\times\{t>0\}$, then the origin is on $\Gamma_u$ by Lemma \ref{L.2.9} and $\vec{b}(0)=0$. We claim that the comparison principle (Lemma \ref{L.cp}) yields $u\leq\varphi^t$ and so $\Omega_u(t)$ is contained in $\{\varphi^t>0\}=\{|\theta|<\theta_t\}$. To justify the use of the comparison principle, by the choice of $\vec{b}$, we first compare $u$ with $\varphi^t(x_1+\delta,x_2)$ for $\delta>0$ (the two functions are strictly separated) and then passing $\delta\to 0$ yields the desired inequality $u\leq \varphi^t$.

\medskip

Now we start with $t=1$ and a solution $u$ such that $(\Omega_u(1)\cap B_1)\subseteq\{|\theta|<\theta_1\}$, and show the formation of cusps. Assume $\gamma_1>2$ and $\sigma:=\frac2{\gamma_1}\in (\gamma_0-1,1)$. For $t\in (1,2)$, define
\[
\Gamma_t:=\{x_1=g(x_2,t)\}\quad \text{ where $g(x_2,t):=|x_2|\cot\theta_1+(t-1)|x_2|^{\sigma}$.} 
\]
So a cusp develops at the vertex of the set $\{x_1>g(x_2,t)\}$ when $t>1$. For each $t\in (1,2)$, let $\phi^t$ be a harmonic function in $\{x_1>g(x_2,t)\}$ with $0$ boundary condition and $\phi^t(\frac12,0)=1$.
If we can show that $\phi^t(x_1,x_2)$ is a supersolution for $t\in (1,2)$, then after further assuming $u$ to be smaller on $\partial B_1$ and by the comparison principle (which can be justified similarly as before), the support of $u$ is contained in cusps for $t\in (1,2)$, which shows the formation of cusps.

To show that $\phi^t$ is a supersolution, it suffices to verify the free boundary condition on $\Gamma_t\cap B_1$.
Note that the curvature of $\Gamma_t$ at point $(g(x_2,t),x_2)$ satisfies
\[
\frac{|\partial_{x_2}^2 g({x_2},t)|}{(1+|\partial_{x_2} g({x_2},t)|^2)^{3/2}}\lesssim \frac{1}{|{x_2}|}\quad \text{ uniformly for all }|{x_2}|<1\text{ and }t\in (1,2).
\]
For any fixed $(y_1,y_2)\in \Gamma_t$, let us consider $\tilde{\phi}^t(x_1,x_2):=\phi^t(|y_2|x_1+y_1,|y_2|x_2+y_2)$. Then the free boundary of $\tilde{\phi}_t$ is a graph of finite curvature in a unit neighbourhood of the origin.  
Thus it follows from Lemma \ref{L.dah} (by comparing with radially symmetric harmonic functions, see also \cite{JK1}) that for some $c>0$, 
\[
|\nabla \tilde{\phi}^t(0,0)|\leq c\,\tilde{\phi}^t(0,-y_2/|y_2|).
\]
After scaling back, we get
\[
|\nabla \phi^t(x_1,x_2)|\leq c|\phi^t(x_1,0)|/|x_2|\leq cx_1^{\gamma_1}/|x_2| \quad \text{ on }\Gamma_t\cap B_1.
\]
In the last inequality we used
$\phi^t(x_1,0)\lesssim \varphi^1(x_1,0)\lesssim x_1^{\gamma_1}$, which is due to the support of $\phi^t$ is contained in $\{|\theta|<\theta_1\}$ and Lemma \ref{L.3.1}.  
Moreover, by direct computation,
\[
\vec{b}\cdot\nu_t=(C_0|{x_2}|^{{\gamma_0}-1},0)\cdot\frac{(-1,\pm (\cot\theta_1+\sigma (t-1)|{x_2}|^{\sigma-1}))}{\sqrt{1+(\cot\theta_1+\sigma (t-1)|{x_2}|^{\sigma-1})^2}}\approx\frac{-C_0|{x_2}|^{{\gamma_0}-1}}{\sqrt{1+(t-1)^2|{x_2}|^{2\sigma-2}}}
\]
where $\nu_t$ denotes the unit normal direction on $\Gamma_t$. The normal velocity of $\Gamma_t\cap B_1$ is
\[
V:=(|{x_2}|^{\sigma},0)\cdot\frac{(-1,\pm (\cot\theta_1+\sigma (t-1)|{x_2}|^{\sigma-1}))}{\sqrt{1+(\cot\theta_1+\sigma (t-1)|{x_2}|^{\sigma-1})^2}}\approx\frac{-|{x_2}|^{\sigma}}{\sqrt{1+(t-1)^2|{x_2}|^{2\sigma-2}}}\gtrsim \frac{-|{x_2}|^{{\gamma_0}-1}}{\sqrt{1+(t-1)^2|{x_2}|^{2\sigma-2}}}
\]
where the last inequality is due to $\sigma>\gamma_0-1$.

It remains to show that, if $C_0$ is large enough, then
\beq\lb{E.11}
V\geq \vec{b}\cdot\nu'+|\nabla\phi^t|\quad \text{ on } \Gamma_t\cap B_1.
\eeq
If $(t-1)|{x_2}|^{\sigma-1}\geq 1$, then $(t-1)|{x_2}|^\sigma\approx x_1$ on $\Gamma_t$. Due to $\gamma_0-\sigma<1$ and $\gamma_1\sigma=2$, we have for $t\in (1,2)$,
\[
V-\vec{b}\cdot\nu'\approx {C_0|{x_2}|^{\gamma_0-\sigma}}/{(t-1)}\gtrsim C_0|{x_2}|\quad\text{and}\quad |\nabla\phi^t|\lesssim (t-1)^{\gamma_1}|{x_2}|^{\gamma_1\sigma-1}\lesssim |{x_2}| .
\]
While if $(t-1)|{x_2}|^{\sigma-1}\leq 1$, then $|{x_2}|\approx x_1$ on $\Gamma_t$, and so, by $1<\gamma_0<\gamma_1$,
\[
V-\vec{b}\cdot\nu'\approx C_0|{x_2}|^{\gamma_0-1}\quad\text{and}\quad |\nabla\phi_t|\lesssim |{x_2}|^{\gamma_1-1}\lesssim |{x_2}|^{\gamma_0-1}.
\]
These imply \eqref{E.11}, and we conclude with the formation of cusps.
\end{example}


Finally, we show that $(\eps,0)$-monotonicity with respect to $W_{\theta,\mu}$ in a large neighborhood does not imply that the solution is monotone along the direction $\mu$ in smaller neighborhood. Here $\theta>0$ can be large and the source term $f$ is smooth.

\begin{example}\lb{E.3}
Fix a small $\delta>0$, and any $\theta\in (0,\frac\pi2)$ and $n\in\bbN$. Take a smooth function $f\geq 0$ such that $f$ is radially decreasing, $f$ is supported in $B_{2\delta}$, and $f\equiv\delta^n$ in $B_\delta$. Then we can assume that $f$ is uniformly bounded in $C^n$ norm regardless of the choice  $\delta$. 
Now let $\phi_1:\bbR^2\supseteq \overline{B_1}\to\bbR$ solve
\[
-\Delta\phi_1=f\text{ in }B_1\quad\text{and}\quad\phi_1=0\text{ on }\partial B_1.
\]
Note that $\phi_1(x)=\int_{B_{2\delta}}\frac{1}{2\pi}\ln|x-y|f(y)dy$. Hence by direct computations,
\beq\lb{E1}
\sup_{B_{2\delta}}|\nabla \phi_1|\geq \delta^{n+1}/C\quad\text{and}\quad \phi_1\in (0, C\delta^{n+2}|\ln\delta|)\text{ in }B_1
\eeq
for some $C>0$ independent of $\delta$.
Moreover, take
\beq\lb{E2}
\phi(x):=\phi_1(x)+2+\delta^{n+1} x_1/(2C),
\eeq
which is strictly positive, and satisfies $\phi_{x_1}<0$ at some points in $B_{2\delta}$ by \eqref{E1} and the fact that $\phi_1$ is radial. We claim that, with $\e:= \delta^{\frac12}$ and $\delta$ sufficiently small, $\phi$ is $(\eps,0)$-monotone with respect to $W_{\theta,\mu}$ with $\mu$ being the positive $x_1$-direction. Indeed, for any $x,y\in B_1$ and $y\in B_{(\sin\theta)\eps}(x+\eps \mu)$, \eqref{E2} and the second inequality in \eqref{E1} yield
\[
\phi(y)-\phi(x)\geq \delta^{n+1}(y_1-x_1)/(2C)-\phi_1(x)\geq (1-\sin\theta)\eps \delta^{n+1}/(2C)-C\delta^{n+2}|\ln\delta|\geq 0,
\]
after taking $\delta=\eps^2$ to be small enough. Thus this yields the claim.

\end{example}

Finally (still in dimension $2$), we show that $(\eps,\eps^\alpha)$-monotonicity with respect to $W_{\theta,\mu}$ in a large neighborhood does not imply that the solution is monotone along the direction $\mu$ in smaller neighborhood. The source term $f$ is bounded, the constant $\alpha\in (0,1)$, and the solution can be $>>\eps$.

\begin{example}\lb{E.2}
Let $\alpha>0$ be fixed, and let $\min\{0,1-\alpha\}<\kappa<1$ and $\delta:=\eps^{(\alpha+\kappa+1)/2}$. Then take $C,\theta,f$ and $\phi_1$ from the previous example with $n=0$. We define
\[
\phi(x):=\phi_1(x)+\delta (x_1+1)/2C+\eps^{\kappa}.
\]
By \eqref{E1} with $n=0$, we have for sufficiently small $\eps$ that
\beq\lb{E4}
\eps^{\kappa}\leq \phi\leq C\delta^2|\ln \delta|+\delta/C+\eps^\kappa\leq 2\eps^\kappa\quad\text{ in }B_1.
\eeq
From the previous example, $\phi_{x_1}$ does not have a sign in $B_1$. We now show that $\phi$ is $(\eps,\eps^\alpha)$-monotone with respect to $W_{\theta,\mu}$ with $\mu$ denoting the positive $x_1$-direction. Indeed, for $x,y\in B_1$ and $y\in B_{(\sin\theta)\eps}(x+\eps \mu)$, we get from \eqref{E1} and \eqref{E4} that
\begin{align*}
\phi(y)-(1+\eps^{\alpha+1})\phi(x)&\geq \delta(y_1-x_1)/(2C)-\phi_1(x)-\eps^{\alpha+1}\phi(x)\\
&\geq (1-\sin\theta)\eps\delta/(2C)-C\delta^2|\ln\delta|-2\eps^{\alpha+\kappa+1}.
\end{align*}
This is non-negative  when $\eps$ is sufficiently small, due to $\eps^{\alpha+\kappa}<<\delta<<\eps$ by the choice of the parameters.

\smallskip

Note that, later in Proposition \ref{L.2.12}, we will apply the improved interior monotonicity of the solution $u$ in the region that is $\eps^{\gamma_1}$-away (with $\gamma_1<1$ but close to $1$) from the free boundary and it is possible that $u\in (\eps^{1/\sigma},\eps^{\sigma})$ for some $\sigma<1$ in the region. Thus the above example indeed indicates that merely bounded source function is not sufficient for the purpose.
\end{example}



\section{Superharmonic Functions in Lipschitz domains}\lb{S.4}
In this section, motivated by Lemma \ref{L.3.11} we begin with studying superharmonic functions in Lipschitz domains, starting with an important localization result (Proposition~\ref{L.2.61}). Building on this we achieve an important growth estimate for $(\e,\eps^\alpha)$ superharmonic functions, up to a small distance away from the free boundary (Lemma~\ref{L.3.4}). The challenge lies in the potential oscillation of the source term $f$, which could affect the distribution of $\nabla w$ in small scale.

\medskip

Throughout the section we denote $g:\bbR^{d-1}\to\bbR$ to be a Lipschitz continuous function with Lipschitz constant $c_g>0$ such that $g(0)=0$. For any $L\geq 2$, define a strip with width $1$ below the graph of $g$ in $B_L$ as
\[
\Sigma'_L:= B_L\cap \left\{x=(x',x_d)\,:\, g(x')-1<x_d<g(x')\right\},
\]
and denote the bottom part of the boundary as
\[
{\partial_b \Sigma'_L:= B_L\cap \{x=(x',x_d)\,:\, x_d=g(x')-1\}}.
\]
    
    

We consider two non-negative functions $w_{1,L}$ and $w_{2,L}$ such that
$$
\left\{\begin{array}{l}
-\Delta w_{1,L}=0, \quad  -\Delta w_{2,L}=1\hbox{  in }\Sigma'_L;\\ \\
w_{1,L}=1, \quad w_{2,L}=0 \hbox{ on }\partial_b\Sigma'_L; \quad w_{1,L}=w_{2,L}=0 \hbox{ on the rest of }\partial \Sigma'_L.
\end{array}\right.
$$

{Below we will show that the two functions are comparable, uniformly with respect to the width parameter $L$. Such result allows us to study our solutions using the well-known properties of harmonic functions in Lipschitz domains. While such result appears to be of classical nature, we were unable to find a relevant version in the literature. It does not appear to be directly verifiable from the Green's function presentation for each functions. }


\begin{proposition}\lb{L.2.61}
For $w_{1,L},w_{2,L}$ and $g$ given as above, let $L \geq 2$ and $c_g < \cot \theta_{2}$, where $\theta_{2}$ is from \eqref{angle_1}.
Then
\[
       w_{2,L}\leq C w_{1,L} \quad\text{ in }\Sigma'_{L-1} \hbox{ for some } C=C(d,c_g). 
\]
\end{proposition}

Let us remark that if $c_g > \cot\theta_{2}$  the proposition is false. This is because near the vertex (at which $w_{1,L}=0$) of a cone with small opening, $w_{1,L}$ grows much slower than quadratic, while $w_{2,L}$ has a quadratic growth.

\begin{proof}
1. First we claim that  $w_{2,L}$ is bounded in $\Sigma'_{L}$, with the bound depending only on $d$ and $c_g$.
If this is not true, then we have a sequence of Lipschitz functions $g_n$ and the corresponding $w_{2,L_n}^n$ such that  $w_{2,L_n}^n(x_n)=\max_{x\in\Sigma'_L}  w_{2,L_n}^n(x)\to\infty$ as $n\to\infty$. Due to the classical regularity results for harmonic functions in Lipschitz domains (see e.g., \cite{JK}), $w_{2,L_n}^n(\cdot+x_n)/w_{2,L_n}^n(x_n)$ are uniformly continuous with bounds that only depends on $d$ and $c_g$. By taking a locally uniform convergent subsequence of $w_{2,L_n}^n(\cdot+x_n)/w_{2,L_n}^n(x_n)$, we can easily obtain a contradiction because the limiting function is harmonic in some Lipschitz domain {whose dirichlet boundary is a unit distance away from the origin}, and it assumes its maximum value $1$ at the origin, which is not possible. So we can conclude.

\medskip

2.  We now simplify what we need to prove. First, by Dahlberg's lemma (Lemma \ref{L.dah}), there is no loss of generality to assume that $w_{1,L}=1$ on $\partial_b\Sigma'_L\cup (\partial B_L\cap\overline{  \Sigma'_L})$.
 Next we claim that if we can prove the conclusion for  $L=2$, then the general case follows. Indeed since $w_{2,L}$ for all $L>2$ are uniformly bounded (denote the bound as $C_*$), $w_{2,L}\leq C_*w_{1,2}$ 
on the boundary of $\Sigma'_L\cap B_2$. This implies that $w_{2,L}\leq C_*w_{1,2}+w_{2,2}$ on $\Sigma'_L\cap B_2$. Then by the assumption that the conclusion {of the lemma} holds for $L=2$ and by Dahlberg's lemma, we obtain
\[
w_{2,L}\leq C_*w_{1,2}+w_{2,2}\leq C' w_{1,2}\leq C'' w_{1,L}\quad \text{ on } \Sigma'_L\cap B_1.
\]
The same holds on $\Sigma'_{L-1}$ by shifting the functions.

\medskip

3. Now we set $L=2$ and  change the variable
\[
y':=x',\quad y_d:=x_d-g(x')\quad (\text{write }y:=(y',y_d)).
\]
Under the transformation, the Lipschitz boundary $x_d=g(x')$ becomes a flat hyperplane $y_d=0$. The operator $-\Delta$ changes to 
\beq\lb{3.62}
\calL:=\calL_g=-\nabla\cdot ((Dy)^T Dy\nabla )
\eeq
where $Dy$ denotes the Jacobian matrix of the transformation. The operator remains uniformly elliptic since $(Dy)^T Dy$ is bounded, measurable and uniformly {positive definite}.

Working with the new coordinates, let us consider the following two non-negative functions
\beq\lb{3.66}
\left\{
\begin{aligned}
   -\calL w'_1&=0,\quad &-\calL w'_2&=1\quad &&\text{ on }\{ |x'|< 2,\, x_d\in (-1,0)\}=:\calT',\\
   w'_1&=1, \quad   &w'_2&=0 \quad &&\text{ on }\{|x'|\leq 2,\, x_d=-1\},\\
      w'_1&=0, \quad &w'_2&=0, \quad &&\text{ on } \{(|x'|=2,\,x_d\in (-1,0))\text{ or } (|x'|\leq 2,\,x_d=0)\},
\end{aligned}
\right.
\eeq
It suffices to show that 
\begin{equation}\label{reduction}
w'_2\leq Cw'_1  \hbox{ on }\calT'.
\end{equation}

\medskip

3. We would like to further reduce the problem to periodic domains. Let us denote $\bbT^{d-1}$ as the $(d-1)$ dimensional torus, and consider 
\beq\lb{3.61}
\left\{
\begin{aligned}
   -\calL w''_1&=0,\quad &-\calL w''_2&=1\quad &&\text{ on }\calT=\{ x'\in \bbT^{d-1},\, x_d\in (-1,0)\}, \\
   w''_1&=1, \quad   &w''_2&=0 \quad &&\text{ on }\{x'\in \bbT^{d-1},\, x_d=-1\},\\
      w''_1&=0, \quad &w''_2&=0, \quad &&\text{ on } \{ x'\in \bbT^{d-1},\,x_d=0\}.
\end{aligned}
\right.
\eeq

We claim that to show \eqref{reduction} it suffices to show $w''_2\leq Cw''_1 $ on $\calT$.  
To prove the claim, we can construct a Lipschitz function $\tilde{g}:4\bbT^{d-1}\to \bbR$ with Lipschitz constant $c_g$ such that $\tilde{g}\equiv g$ on $\{x'\in\bbR^{d-1}\,:\,|x'|<2\}$. Then the corresponding operator $\calL_{\tilde{g}}$ agrees with $\calL$ on the same region. 
Let us still call solutions from \eqref{3.61} with $\calL_{\tilde{g}}$ and $4\bbT^{d-1}$ in place of $\calL$ and $\bbT^{d-1}$ as $w_1''$ and $w_2''$. 
Then Lemma \ref{L.dah} in pre-transformation coordinates and uniform continuity of $w_1',w_1''$ yield $w_1''\leq Cw_1'$ on $\calT'$, and the comparison principle yields $w_2'\leq w_2''$. Hence $w''_2\leq Cw''_1$ implies \eqref{reduction}, which shows the claim after rescaling.


\medskip

4. Now we proceed to show $w''_2\leq Cw''_1 $ in the periodic domain $\calT$ for $w''_1,w''_2$ from \eqref{3.61}. We will proceed with induction, to approach the boundary of $x_d=0$. Let us denote\[
\calT_k:=\{x\in \calT\,:\, x_d\in (-2^{-k},0)\}.
\]
Since $w''_1>0$ is uniformly bounded away from $0$ when $x_d\in (-1,-\frac12]$ and $w''_2$ is uniformly bounded, there exists $C_1>0$ such that $w''_2\leq C_1w''_1$ in $\calT\backslash \calT_1$. Suppose $w''_2\leq C_kw''_1$ in $\calT\backslash \calT_k$ for some $k\geq 1$ and $C_k>0$. 
Let $\phi_k$ be the unique solution to
\[
\left\{
\begin{aligned}
   -\calL \phi_k&=1\quad\text{ in }\calT_k,\\
   \phi_k&=0 \quad\text{ on }\partial\calT_k.
\end{aligned}
\right.
\]
Then by considering $4^k\phi_k(2^{-k}x)$ in $2^k \calT_k$,  the bound in Step 1. in {pre-transformation coordinates} yields that $ \phi_k\leq C_* 4^{-k}$ for some $C_*>0$ independent of $k$. Since $w''_2\leq C_kw''_1$ on $\partial \calT_k$, we obtain
\[
w''_2\leq C_kw''_1+\phi_k\leq C_kw''_1+C_*4^{-k}\quad\text{ in }\calT.
\]
Since $c_g \leq \cot\theta_{\beta}<\cot\theta_{2}$ for some $\beta<2$,  Lemma \ref{L.2.10} yields that  $w''_1\geq C|x_d|^\beta$. Thus, using that $w''_1\geq C 2^{-(k+1)\beta}$ in $\calT\backslash\calT_{k+1}$, there exists $C'=C'(C_*)>0$ such that
\[
C'w''_1\geq C'2^{-(k+1)\beta}\geq 2^{(2-\beta)k} (4^{-k}C_*)\quad\text{ in }\calT\backslash \calT_{k+1}.
\]
We then obtain 
\[
w''_2\leq C_{k+1}w''_1\quad\text{ in }\calT\backslash\calT_{k+1}=\{x\in\calT\,:\, x_d\in (-1,-2^{-k-1}]\}.
\]
with $C_{k+1}:=C_k+C'2^{(\beta-2)k}$. Because $C'$ is independent of $k$ and $\beta<2$, we have 
\[
\lim_{k\geq 0}C_{k}<\infty,
\]
and therefore $w''_2\leq Cw''_1 $ in $\calT$ for some $C>0$ which finishes the proof.
\end{proof}

Later, instead of applying Proposition \ref{L.2.61} directly, we are going to use the following corollary.
In it, we use $\Sigma_\delta$ which recales the previous $\Sigma'_L$ to unit length but with $\delta$ width (so $\delta\sim 1/L$). For any $\delta\in (0,\frac{1}{2})$, consider the domain
\begin{equation}\label{strip}
\Sigma_{\delta}:=B_1\cap \{x_d\in (g(x')-\delta,g(x'))\},\quad \partial_b\Sigma_\delta:=B_1\cap \{x_d=g(x')-\delta\}.
\end{equation}

\begin{corollary}\lb{C.2.7}
Let $g$ be as in Proposition \ref{L.2.61} with $c_g\leq \theta_{\beta}$ for some $\beta\in (1,2)$. Let $f:\bbR^d\to [0,\infty)$ be 
continuous, and let $\omega$ be a non-negative function solving
\[
-\Delta \omega=f\text{ in }B_2\cap \{x_d<g(x')\},\quad \omega=0\text{ on }B_2\cap \{x_d=g(x')\}, \omega(-e_d) >0. 
\]
Consider $w_1$ and $w_2$ each solving
$$
\left\{\begin{array}{lll}
-\Delta w_1 = 0, &  -\Delta w_2 = f &\hbox{ in } \Sigma_\delta,\\
w_1=\omega ,  & w_2= 0 &\hbox{ on } \partial_b \Sigma_\delta,\\  
w_1=w_2=0 \,\, &&\hbox{ on the rest of }\partial\Sigma_\delta.
\end{array}\right.
$$ 

Then  there exists $C=C(d,\beta)$ such that 
$$
\delta^{\beta-2} w_2 \leq C\frac{\|f\|_\infty}{\omega(-e_d)}w_1\quad\hbox{ in }B_{1-\delta}\cap \Sigma_{\delta}.
$$ 
Moreover, in the same domain we have
$$
w_1 \leq \omega \leq \left(1+C\delta^{2-\beta}\frac{\|f\|_\infty}{\omega(-e_d)}\right) w_1.
$$
\end{corollary}


\begin{proof}
First, for $m:=\omega(-e_d)$, Lemma \ref{L.3.1} yields that $\omega\geq cm\delta^\beta$ on $\partial_b \Sigma_\delta$ for some $c>0$. 
So that $\bar w_1(x):=m^{-1}\delta^{-\beta}w_1(\delta x)$ is harmonic in $\Sigma_\delta/\delta$ and $\bar w_1\geq c$ on $(\partial_b \Sigma_\delta)/\delta$. Note that $\bar w_2(x):=\delta^{-2} w_2(\delta x)$ satisfies
\[
-\Delta \bar w_2\leq \|f\|_\infty \text{ in }\Sigma_\delta/\delta\quad\text{ and }\quad \bar w_2=0 \text{ on }(\partial_b \Sigma_\delta)/\delta.
\]
Thus, applying the comparison principle and Proposition \ref{L.2.61} with $\|f\|_\infty^{-1}\bar{w}_2$ (when $\|f\|_\infty>0$) and $\bar{w}_1$ in place of $w_2$ and $w_1$ yield for some $C>0$,
\[
m\delta^{\beta-2} w_2\leq C\|f\|_\infty w_1 \quad\text{ in }\Sigma_\delta\cap B_{1-\delta}.
\]

 For the second claim, note that $v:=\omega-w_2\geq0$ is a harmonic function in $\Sigma_\delta$, $w_1=\omega=v$ on $\partial_b\Sigma_\delta$ and $w_1\leq v$ on $\partial\Sigma_\delta$. Hence the comparison principle yields $w_1\leq v$ in $\Sigma_\delta$. 
And by Dahlberg's lemma (Lemma \ref{L.dah}), we have
\beq\lb{3.10}
(w_1\leq)\, v\leq Cw_1\quad\text{ in }\Sigma_{\delta/2}\cap B_{1-\delta/2}
\eeq
for some $C=C(d)>1$.
Since $w_1$ and $v$ are harmonic, 
$v-w_1=0 $ on $B_1\cap\partial_b\Sigma_{\delta}$, and $0\leq v-w_1\leq (C-1)w_1$ on $B_{1-\delta/2}\cap \partial_b \Sigma_{\delta/2}$  by \eqref{3.10}, we apply Lemma \ref{L.dah} again to have $v-w_1\leq C'w_1$ in $B_{1-\delta}\cap (\Sigma_{\delta}\backslash \Sigma_{\delta/2})$ for some $C'=C'(d)$. Hence we have \eqref{3.10} holds (with possibly a different $C=C(d)$) in $B_{1-\delta}\cap\Sigma_{\delta}$, which finishes the proof by the established first claim.
\end{proof}


{The following Corollary states a scenario where the assumption of our main theorem is satisfied, with a quantitative range of space and time.}

\begin{corollary}\label{cor:flat}
Let $u$ solve (1.1) with $\vec{b}=0$ and $f\equiv 1$ 
in $B_1(0)\times [0,1]$, with initial free boundary  $\Gamma_0:= \{x_d= g(x')\}$, where $g$ is as given in Proposition 4.1. Assume that $u(-\frac{1}{2}e_d,t)\approx 1$ for $t\in[0,1]$. If $\delta>0$ and the Lipschitz constant of $g$ are small, {then $u$ satisfies the assumptions in Theorem A.} More precisely, there exists $c>0$ independent of $\delta$ such that $u$ is $(\delta^{3/2},1)$-monotone with respect to the cone $W_{\theta, -e_d}$ in $B_{c\delta}(0)\times [0,  t_{c\delta}]$, where $t_s:=\frac{s}{u(-s e_d,0)}$ and $\theta:=\theta_\beta$ is from Corollary \ref{C.2.7}.
\end{corollary}


 \begin{proof}

We only sketch the proof. Let $\Sigma:=[\{x\,|\,  g(x') -C\delta < x_d\} \cap B_1(0)]\times [0,1]$ for some $C$ large.
If we use $u$'s value as the data on the parabolic boundary of $\Sigma$, and solve the homogeneous problem $(HS)$ (that is \eqref{1.1} with $f, \vec{b} \equiv 0$) in $\Sigma$, and if $C$ is large enough (and we fix it), Theorem 5.7 in \cite{CJK} in particular states that the corresponding solution $w$ has spatially Lipchitz free boundary which is monotone in the cone $W_{\theta, -e_d}$ for $t\in [0,t_{\delta}]$. Let us call $\Gamma_t(w)$ as the free boundary of $w(\cdot,t)$, which is a Lipschitz graph with respect to $x'$-variable up to the time $t_{\delta}$. Let $w_1$ solve $-\Delta w_1(\cdot,t)=1$ in the same domain and boundary data as $w(\cdot,t)$. Then by Corollary \ref{C.2.7},  $w\leq w_1 \leq (1+C\delta) w$, and so $|Dw_1| \leq (1+C\delta) |Dw|$ in $\Sigma\cap \{t<t_\delta\}$. 
Since the domain is expanding in time, $w_1-w$ is increasing in time and so we have $(w_1)_t\geq w_t$. Thus there exists $C>0$ such that
\[
\partial_t w_1(\cdot,(1+C\delta)\cdot)\geq (1+C\delta)w_t=(1+C\delta)|\nabla w|^2\geq |\nabla w_1|^2
\]
in the viscosity sense on $\Gamma_t(w)$. In particular, $w_1(x, (1+C\delta)t)$ is a supersolution to \eqref{1.1}. So we get
\[
w_1\leq u\leq w_1(\cdot, (1+C\delta)\cdot).
\]
Since $\{w(\cdot,t)>0\}=\{w_1(\cdot,t)>0\}$,  it follows from Lemma 3.3 \cite{CJK} that for $t\in [0,t_\delta]$ that
\[
\{w_1(\cdot,t)>0\}\cap B_{\delta}\subseteq \{u(\cdot,t)>0\}\cap B_{\delta}\subseteq \{w_1(\cdot,t)>0\}+B_{C\delta^2}
\]
where the last term denotes the $C\delta^2$-neighbourhood of $\{w_1(\cdot,t)>0\}$. So the comparison principle yields
\beq\lb{123}
w_1(x,t)\leq u(x,t)\leq w_1(x+C\delta^2e_n,t).
\eeq

Recall that $w$ is monotone with respect to $W_{\theta,-e_d}$ for $t\in[0,t_\delta]$. It follows from the proof of Lemma \ref{L.4.5} below that there exists $c>0$ such that for any $e\in W_{\theta,-e_d}$, and $(x,t)\in B_{c\delta}\times [0,t_{c\delta}]$,
\[
d(x,\Gamma_t(w_1))\nabla_e(w_1)\geq c\,w_1.
\]
From this, \eqref{123}, and $\eps:=\delta^{3/2}>>\delta^2$, it follows that
\[
(1+\eps)u(x,t)\leq (1+\eps)w_1(x+C\delta^2 e_n,t)\leq (1+\eps)w_1(x+\eps e,t)-c\eps|\nabla w_1(x+\eps e,t)|\leq u(x+\eps e,t)
\]
which finishes the proof.
\end{proof}

The next two lemmas connect $\omega$ and $\nabla\omega$ in terms of its distance from the free boundary. 
These were proved in \cite[Lemma 11.11]{CafSal} for the case $f=0$. {A crucial element in the proof is Harnack inequality for the directional derivative of harmonic functions. In our case Proposition~\ref{L.2.61} is applied to avoid differentiating the source function. }


\begin{lemma}\lb{L.2.88}
Let $\omega$ and $g$ be as in Corollary \ref{C.2.7}, and in addition suppose that $\omega_{x_d}\leq 0$. 
Then there exists $C>0$ such that for all  sufficiently small $\delta$ we have
\[
C d(x,\{y_d=g(y')\})\,\omega_{-x_d}(x)\geq \omega(x)
\]
holds for all $x\in \Sigma_\delta\cap B_{1-\delta}$, {where $\Sigma_\delta$} is given in \eqref{strip}.
\end{lemma}
\begin{proof}

Let us fix a point $(y',g(y'))\in B_{1-\delta}$. For simplicity, we may assume that $y'=0$ and $g(y')=0$. For $r>0$ define $z_r:=g(y')-re_d$. Then for $w_1$ given in Corollary \ref{C.2.7} the following is true due to the boundary Harnack principle \cite[Theorem 11.5]{CafSal} and the remarks (a)(b) in its proof (see also the proof of \cite[Lemma 11.11]{CafSal}):
There exist $C,\sigma>0$ such that 
we have
\beq\lb{3.6}
w_1(\tau z_r)\leq C\tau^\sigma w_1( z_r) \hbox{ for any } r\in (0,\delta/2), \tau \in (0,1). 
\eeq

Next if $\delta^{2-\beta}\|f\|_\infty\leq \omega(-e_d)$, then $w_1\leq \omega\leq Cw_1$ for some $C>1$ in $\Sigma_\delta\cap B_{1-\delta}$ by Corollary \ref{C.2.7}. Thus by taking $\tau>0$ to be small enough (independent of $r$ and $y'$), we obtain from \eqref{3.6} that $\omega(\tau z_r)\leq \frac12\omega(z_r)$. This implies
\beq\lb{2.41}
\frac{1}{2}\omega(z_r)\leq \int_{\tau r}^r \omega_{-x_d}(-se_d)ds \,\,(\leq \omega(z_r)).
\eeq 

Now since $\omega_{x_d}\leq 0$, by applying the last claim of Lemma \ref{l.2.2} for possibly multiple times, we get for all $s\in (\tau r,r)$ that
\[
\omega_{-x_d}(-se_d)\leq C\omega_{-x_d}(z_r)+Cr\| f\|_{L^\infty(B_1)}
\]
with $C>0$ depending on $\tau$ and an upper bound of $c_g$.
So \eqref{2.41} yields
\beq\lb{241}
Cr\omega_{-x_d}(z_r)\geq \omega(z_r)-Cr^2\|f\|_\infty.
\eeq
For $\theta:=\arccot c_g$ ($\geq\theta_{\beta}$ with $\beta\in (1,2)$), we get
\[
 r\sin\theta\leq d(z_r,\{y_d=g(y')\})\leq r.
\]
Also using $\omega(z_r)\geq Cr^\beta\omega(-e_d)$ by Lemma \ref{L.3.1} and the above, we get from \eqref{241} that for some positive constants $C,C'$ depending on $\theta,\omega(-e_d)$ and $\|f\|_\infty$, 
\begin{align*}
Cd(z_r,\{y_d=g(y')\})\omega_{-x_d}(z_r)&\geq \omega(z_r)-Cr^2\|f\|_\infty\\
&\geq \frac12\omega(z_r)+Cr^\beta\omega(-e_d)-C'r^2\omega(-e_d)\geq \frac12\omega(z_r)
\end{align*}
if $r<\delta$ is small enough.
\end{proof}

We now  relax the previous assumption and consider $(\eps,\eps^\alpha)$-monotone funcitons. {Note that the cone of monotonicity needs to be wider as the regularity of $f$ decreases.} 


\begin{lemma}\lb{L.3.4}

Let $f\in C^{\bar{\gamma}}(\R^d)$ for some $\bar{\gamma} \in (0,1]$ be non-negative, and let $\omega\geq 0$ solve $-\Delta \omega=f$ in $B_2\cap\Omega_\omega$ with $0\in\Gamma_\omega$ and $\omega(-e_d)>0$. Suppose that $\alpha\in (0,\frac{\bar\gamma}2)$ and $\kappa_2\in (\frac{\alpha}{\bar\gamma},\frac12)$ if $f$ is not a constant. Otherwise take $\alpha=\infty$ and any $\kappa_2\in (0,\frac12)$ if $f$ is a constant.

  In addition suppose that $\omega$ is $(\eps,\eps^\alpha)$-monotone with respect to $W_{\theta,-e_d}$ in $B_1$ with $\theta >\theta_{1+\bar{\gamma}}$. Then there exists $C=C(\theta,\omega(-e_d), \|f\|_\infty)>0$ such that: for all $\eps$ sufficiently small, 
\begin{equation}\label{conclusion}
C|\nabla\omega(x)|d(x,\Gamma_\omega) \geq \omega(x)\quad\text{ in } B_{1-\eps^{1/2}}\cap\{x: C\eps^{1-\kappa_2} \leq d(x,\Gamma_\omega) \leq \eps^{1/2}\}\cap\Omega_\omega.
\end{equation}
\end{lemma}

\begin{proof}

1. Let $x_0\in B_{1-\eps^{1/2}}\cap\Omega_\omega$ satisfying $\delta_0:=d(x_0,\Gamma_\omega)\in [ 2\eps^{1-\kappa_2},\eps^{1/2}]$. Below we write $a_0:=\omega(x_0)$. Denoting $x= (x', x_d)$, we consider the domain  
\[
D_0:= \left\{x\,:\, |x'-(x_0)'|<8^{-1}{\delta_0}, \, |x_d-(x_0)_d|< 2\delta_0\right\} .
\]
Let us also define
\[
N_\eps:=\{x\in B_{1-\eps^{1/2}}\cap\Omega_\omega\,:\, d(x,\Gamma_\omega)>\eps^{1-\kappa_2}\}.
\]
By Lemma \ref{C.2.8} with some $\beta\in (1,1+\bar{\gamma})$ and the assumption, for some $c>0$ we have
\beq\lb{3.3}
\omega(x)
\geq  c\,\omega(-e_d) \eps^{(1-\kappa_2)\beta}\quad  \hbox{ in } N_{\eps}.
\eeq

Note that when $f$ is not a constant, we have $\eps^{(1-\kappa_2)\beta+\alpha}>>\eps^{1+\bar{\gamma}-\kappa_2}$ and $\eps^{1/2}\geq\eps^{1-\kappa_2}$ due to $\beta\in(1,1+\bar\gamma)$ and $\kappa_2\in (\frac{\alpha}{\bar{\gamma}},\frac12)$. So it follows from the second remark of Remark \ref{R.3.2} that 
\beq\lb{3.5}
\omega(\cdot)\text{ is fully monotone non-decreasing along all directions in $W_{\theta,-e_d}$ in $ N_\eps$}.
\eeq
By our assumption of $(\eps,\eps^\alpha)$-monotonicity and the fact that  $\theta\geq\frac\pi4$, it follows that the set $\{\omega(\cdot)=a_0\}\cap D_0$ is at least $\eps^{1-\kappa_2}$-away from $\Gamma_\omega\cap D_0$, and therefore $\{\omega(\cdot)-a_0=s\}$ for any $s>0$ are Lipschitz hypersurfaces in $D_0$.

\medskip

2. Now let $w_1(\cdot)$ and  $w_1'(\cdot)$ be, respectively, the harmonic functions in $D_0\cap\Omega_\omega$ with $w_1=\omega$ on $\partial (D_0\cap\Omega_\omega)$ and in $D_0$ with $w_1'=\omega$ on $\partial D_0$. From \eqref{3.3} and classical regularity results of elliptic operators, we get for some $c=c(d)>0$,
\[
w_1'(x_0)\geq w_1(x_0)\geq c\,\omega(-e_d)\delta_0^\beta.
\]
Since $w_2:=\omega-w_1$ satisfies $-\Delta w_2=f$ and $w_2= 0$ on $\partial(D_0\cap\Omega_\omega)$, we get
$w_2\leq C\delta_0^2\|f\|_\infty$ for some $C>0$ in $D_0\cap \Omega_\omega$. Therefore, using the fact that $w_2(x_0)\leq C'\delta_0^{2-\beta}w_1(x_0)$ with  $C':=C\|f\|_\infty/(c\omega(-e_d))$, we have
\beq\lb{3.11}
a_0\leq (1+C'\delta_0^{2-\beta})w_1'(x_0).
\eeq

Next, similarly as done in the proof of Lemma 5.6 \cite{CafSal}, let $h^{x}$ (with $x\in D_0$) be the harmonic measure in $D_0$. By the $(\eps,\eps^\alpha)$-monotonicity assumption and $0\in\Gamma_\omega$, we have 
\[
|\partial D_0\cap \{w'_1=0\}|=|\partial D_0\cap \{\omega=0\}|\geq c |{\partial D_0}|\quad \text{ for some } c=c(d,\theta)\in (0,1). 
\]
Hence Lemma 11.9 \cite{CafSal} implies that 
\[
w'_1(x_0)=\int_{\partial D_0} \omega(\sigma) dh^{x_0}(\sigma)\leq (1-c')\max_{\partial D_0}\omega
\]
for some $c'=c'(d,\theta)\in (0,1)$. Thus by taking $\delta_0$ (and so $\eps$) to be small enough and applying \eqref{3.11}, we obtain
\[
a_0\leq (1-c'/2)\max_{\partial D_0} w_1.
\]
Therefore there exists $x_1\in \partial D_0$ such that for $C_0:=\frac{1}{1-c'/2}>1$, 
\beq\lb{2.42}
w_1(x_1)=\omega(x_1)\geq C_0a_0>a_0.
\eeq

\medskip

3. Let us consider the domain 
\[
D_1:= \left\{x\,:\, |x'-(x_0)'|<8^{-1}{\delta_0}, \, -3\delta_0<x_d-(x_0)_d,\, \omega(x)>a_0\right\},
\]
From the full monotonicity \eqref{3.5} that the level sets $\{\omega-a_0=s\}\cap D_1$ for $s>0$ are Lipschitz graphs. Since $x_1\in D_1$, the set $\{\omega>C_0a_0\}\cap D_1$ is at most $C\delta_0$-away from $\Gamma_\omega\cap D_0$.
Since $-\Delta(\omega-a_0)^+=f$ in $ D_1$ and $\omega_{e_d}\leq 0$, 
we can apply Lemma \ref{L.2.88} to $(\omega-a_0)^+$ to obtain
\[
(\omega(x)-a_0)^+\leq C|\nabla\omega(x)|\,d(x_1,\{\omega>a_0\})\leq C'|\nabla\omega(x)|\,d(x_1,\Gamma_\omega)
\]
for all $x\in D_1$ when $\eps$ is sufficiently small.
While we also know from \eqref{2.42} that 
\[
\omega-a_0\geq (1-C_0^{-1})\omega \quad \hbox{ in } \{\omega > C_0a_0\} \cap D_1.
\]
Thus the inequality \eqref{conclusion}   holds for $x\in \{\omega>C_0a_0\}\cap  D_1$. Since $x_0$ is an arbitrary point that is $\eps^{1-\kappa_2}$-away from $\Gamma_\omega$, by shifting $x_0$, $\{\omega>C_0a_0\}\cap  D_1$ contains all points $x\in B_{1-\eps^{1/2}}$ such that $d(x,\Gamma_\omega)\in [C\eps^{1-\kappa_2},\eps^\frac12]$. We finished the proof.
\end{proof}

\subsection{Lipschitz free boundary implies cone monotonicity}

\medskip

Here we show that if the boundary is Lipschitz continuous, then the solutions to $-\Delta \omega = f$ with $0$ boundary data are cone-monotone when sufficiently close to the boundary.  

Let $g$ be a Lipschitz function as given in the beginning of Section \ref{S.4}. 
\begin{lemma}\lb{L.4.5}
Let $D_r:= B_r\cap\{x_d<g(x')\}$ for $r>0$ and let $\omega\geq0$ be a solution to $-\Delta \omega = f$ in $D_1$ such that $\omega=0$ on $g(x')=0$ and $\omega(-\frac{1}{2}e_d)=1$. Then if $c_g\leq \min\{\cot\theta_{\beta},\cot\theta'_\beta\}$ for some $\beta\in (1,2)$ (where $\theta_\beta,\theta'_\beta$ are from Lemma \ref{L.3.1}), then there are $c,r>0$ such that $d(x,\{y_d=g(y')\})\,\omega_{-x_d}\geq c\,\omega$ in $D_r$.
\end{lemma}

\begin{proof}
For some $\delta\in(0,1)$ to be determined, let 
\[
\omega_\delta(x):=a\,\omega (\delta x)\quad\text{ with }a:=1/\omega(-\delta e_d/2).
\]
Then $\omega_\delta$ satisfies $\omega_\delta(-e_d/2)=1$ and $-\Delta\omega_\delta=f_\delta$ with $f_\delta(x):=a\delta^2f(\delta x)$.
By the assumption on $c_g$, Lemma \ref{L.3.1} and Corollary \ref{C.2.7} yield that 
\beq\lb{2.50}
C^{-1}\delta^\beta\leq a\leq C\delta^{2-\beta},\quad \omega_\delta(-\delta e_d)\leq C\delta^{2-\beta}.
\eeq

Now let $h_{1},h_{2}$ be two harmonic functions  in 
$
D^\delta_1:=B_1\cap\{\delta x_d<g(\delta x')\}$ such that 
\begin{align*}
&h_{1}=\omega_\delta,\,h_{2}=1&&\text{ on } \partial B_1\cap\overline{D^\delta_1}\\
&h_{1}=h_{2}=0&&\text{ on }B_1\cap \{\delta x_d=g(\delta x'
)\}. 
\end{align*}
For $y:=-\delta e_d$, Corollary \ref{C.2.7} and \eqref{2.50} yield
\[
\omega_\delta(y)-h_{1}(y)\leq C\|f_\delta\|_\infty h_{1}(y)\leq C\|f_\delta\|_\infty \omega_\delta(y)\leq C\delta^{6-2\beta}.
\]

Next, it follows from the last two lines of the proof of Lemma 11.12 \cite{CafSal} that if $\delta$ is sufficiently small depending only on $c_g$ and $d$,
\[
\nabla_{-x_d}h_{1}(y)\geq c\,\nabla_{-x_d}h_{2}(y)\geq c\, h_{2}(y)/\delta.
\]
where $c$ is a dimensional constant. By Lemma \ref{L.3.1} again,  
\beq\lb{2.51}
\nabla_{-x_d}h_{1}(y)\geq c\, \delta^{\beta-1}.
\eeq
In view of Lemma \ref{l.2.2} and $|f_\delta|\leq C\delta^{4-\beta}$,
\[
|\nabla_{-x_d}(\omega_\delta-h_{1})(y)|\leq C\delta^{-1}(\omega_\delta(y)-h_{1}(y))+C\delta\|f_\delta\|_\infty\leq C\delta^{5-2\beta}.
\]
Thus \eqref{2.51} and $\beta<2$ yield for all $\delta$ sufficiently small that $\nabla_{-x_d}\omega_\delta(y)\geq 0$. This implies that $\nabla_{-x_d}\omega(-\delta^2 e_d)\geq 0$ for all $\delta$ sufficiently small. Finally the proof is finished after applying Lemma \ref{L.2.88}.
\end{proof}

\section{Sup-convolution}\lb{S.5}



In this section we prove several properties of sup-convolutions, first introduced by Caffarelli (see e.g. \cite{Caf87}). They will be used in the constructions of barriers in the next section. 
 
\medskip

For non-negative functions  $u$ in $C(B_1\times (0,T) )$ and $\varphi \in C^{2,1}_{x,t}(B_1\times (0,T) )$ with $0< \varphi \leq 1/2$,  define
\begin{equation}\label{sup}
v(x, t):=\sup_{B_{\varphi(x,t)}(x)}u(y,t)\quad\text{ in } B_{1/2}\times (0,T).
\end{equation}
The following lemma says that if $u$ is $(\eps,0)$-monotone, then the level surfaces of $v$ are Lipschitz graphs whenever $\eps/\varphi$ and $\nabla \varphi$ are not too big. 

\begin{lemma}{\rm (Lemma 5.4 \cite{CafSal})}
\lb{L.2.6}
Let $v$ be as given in \eqref{sup}. Suppose that $u(\cdot,t)$ is $(\eps,0)$-monotone with respect to $W_{\theta,\mu}$ for some $\theta\in (0,\frac\pi2]$ in $B_{1}$, and for some $x\in B_{1/2}$ and $\bar\theta\in (0,\frac\pi2)$ we have
\beq\lb{2.011}
\sin\bar\theta\leq \frac{1}{1+|\nabla_x\varphi(x,t)|}\left(\sin\theta-\frac{\eps\cos^2\theta}{2\varphi(x,t)}-|\nabla_x\varphi(x,t)|\right).
\eeq
Then $v(\cdot,t)$ is non-decreasing along all directions in $W_{\bar\theta,\mu}$ at $x$. 
\end{lemma}


The following lemma estimates $\Delta v$. The proof is similar to those in \cite{Caf87,kimzhang,CJK}.



\begin{lemma}\lb{L.2.7}
Suppose $-\Delta u= f\geq 0$ in $\Omega_u$ with continuous $f:\bbR^{d}\to\bbR$. Let $v$ be given by \eqref{sup}, then $v(x,t) = u(y(x,t),t)$ for some $y(x,t) \in B_{\varphi(x,t)}(x)$. Then there are dimensional constants $A_0,A_1>1$ such that if $\varphi$ satisfies 
\beq\lb{2.4}
\Delta\varphi\geq \frac{A_0|\nabla\varphi|^2}{|\varphi|} \quad \hbox{ in } B_1\times (0,T) ,
\eeq
then $v$ satisfies (in the viscosity sense)
\[
-\Delta v \leq (1+A_1 \|\nabla\varphi\|_\infty)f \circ y \quad \text{ in } \Omega_v \cap [ B_{1/2}\times (0,T)].
\] 
\end{lemma}
\begin{proof}


 Since $t$ stays fixed in the proof, we will omit its dependence from the notations of $u, v, \varphi$ and $y$.
We follow the idea of Lemma 9 in \cite{Caf87} and compute
\[\Delta v(0)=\underline{\lim}_{r\to 0}\left(\fint_{B_r}v(x)-v(0)dx\right), \quad\hbox{ where } \fint_{B_r}v(x)dx:=\frac{1}{|B_r|}\int_{B_r} v(x)dx.
\]

Let $ x_0\in B_{1/2}\cap\{v(\cdot,t)>0\}$, which may be set to be the origin. {If $y(0)$ is a local supremum of $u$, then there is nothing to prove since $\Delta v(0)=0$.} Otherwise $y(0)\in \partial B_{\varphi(0)}(0)$, by choosing an appropriate system of coordinates, we can assume for some $\gamma_1,\gamma_2\in\bbR$ that
\begin{align}\lb{2.2}
    v(0)=u(\varphi(0)e_d)\quad\text{ and }\quad \nabla\varphi(0)=\gamma_1 e_1+\gamma_2 e_d.
\end{align}
Recall that
\[
v(x)=\sup_{|\nu|\leq 1}u\left(x+\varphi(x)\nu\right)\geq 0.
\]
Let us estimate $v(x)$ from below by taking $\nu(x):=\frac{\nu_*(x)}{|\nu_*(x)|}$ where
\beq\lb{2.2'}
\nu_*(x):=e_d+\frac{\gamma_2 x_1-\gamma_1 x_n}{\varphi(0)}e_1+\frac{\gamma_3}{\varphi(0)}\Big(\sum_{i=2}^{d-1}\,x_i\,e_i\Big)
\eeq
and $\gamma_3\in\bbR$ satisfies
\beq\lb{2.3}
(1+\gamma_3)^2=(1+\gamma_2)^2+\gamma_1^2.
\eeq
With this choice of $\nu$, we define $y(x):=x+\varphi(x)\nu(x)$ and so
$y(0)=\varphi(0)e_d$.
Then direct computations yield (also see \cite{Caf87})
\beq\lb{B.10}
y(x)=Y_*(x)+\varphi(0)e_d+o(|x|^2)
\eeq
where $Y_*(x)$ denotes the first-order term that is
\beq\lb{B.5}
Y_*(x):=x+(\gamma_1 x_1+\gamma_2 x_n)e_d+(\gamma_2 x_1-\gamma_1 x_n)e_1+\gamma_3\sum_{i=2}^{d-1} x_ie_i.
\eeq
Hence $Y_*(x)$ is a rigid rotation plus a dilation, and \eqref{2.2} and \eqref{2.3} imply
\begin{equation}
    \label{B.1}
     \left|\frac{D (Y_*(x)-x)}{Dx}\right|\leq C\|\nabla\varphi\|_\infty.
\end{equation}
Then we have
\[
\begin{aligned}
    \fint_{B_r}v(x)-v(0)dx&\geq 
    \fint_{B_r}u(y(x))-u(y(0))dx\\
    &\geq 
    \fint_{B_r}u(y(x))-u(Y_*(x)+y(0))dx+
    \fint_{B_r}u(Y_*(x)+y(0))-u(y(0))dx. 
\end{aligned}
\]
Using \eqref{2.4} and following the computations done in Lemma 9 \cite{Caf87}, we find that the first integration in the above $\geq0$. 
Since $u$ is $C^2$ near $y(0)$ by the assumption, 
\[
\lim_{r\to0}\frac1{r^2}\fint_{B_r}u(Y_*(x)+y(0))-u(y(0))dx=\left(\left|\frac{D Y_*(x)}{D x}\right|_{x=0}\right)^2f(y(0)).
\]
Using \eqref{B.1} and $\Delta u (y(0))\leq 0$, we get
\beq\lb{2.6}
\begin{aligned}
\liminf_{r\to0}\frac{1}{r^2}\fint_{B_r}v(x)-v(0)dx
    &\geq \lim_{r\to0}\frac{1}{r^2}\fint_{B_r}u(Y_*(x)+y(0))-u(y(0))dx\\
    &\geq -(1+C \|\nabla\varphi\|_\infty)f (y(0)).
\end{aligned}
\eeq

Finally to show the conclusion, suppose $\phi$ is a smooth function such that $\phi$ touches $v$ from above at $0$. 
If $r>0$ is small enough, we have
$
\phi(x)\geq v(x)$ for $x\in B_r$, and thus
\[
\Delta\phi(0)=\lim_{r\to0}\frac{1}{r^2}\fint_{B_r}\phi(x)-\phi(0)dx\geq \limsup_{r\to0}\frac{1}{r^2}\fint_{B_r}v(x)-v(0)dx.
\]
This and \eqref{2.6} shows 
\[
\Delta\phi(0)\geq -(1+C \|\nabla\varphi\|_\infty)f(y(0))
\]
which finishes the proof.
\end{proof}

Below we show that if $u $ satisfies the free boundary condition in \eqref{1.1}, then ${v}$ satisfies some appropriate free boundary condition as well. 


\begin{lemma}\label{L.2.8}
Let $u,v$ be as given in Lemma ~\ref{L.2.7}, where $\varphi\in C^{2,1}_{x,t}$ satisfies 
\[
\varphi \leq \eps_1,\quad |\nabla\varphi|\leq \eps_2,\quad -1/2\leq  \varphi_t\leq \eps_3.
\]
In addition, if $u$ is a viscosity subsolution of \eqref{1.1} in $B_1\times (0,T)$,  
and if $\eps_1,\eps_2,|\eps_3|$ are small enough, then
$v$ is a viscosity subsolution of 
\[
\left\{
\begin{aligned}
   -\Delta v &\leq (1+A_1 \|\nabla\varphi\|_\infty)f \circ y\qquad\qquad\text{ in }\Omega_v\cap (B_{1/2}\times (0,T)),\\
v_t&\leq (1+2\eps_2)^2|\nabla v|^2{+}\vec{b}\cdot\nabla v+\left(\eps_1\|\nabla\vec{b}\|_\infty+2(\eps_1+\eps_2)\|\vec{b}\|_\infty+\eps_3+|\eps_3|/2\right)|\nabla v|\\
&\qquad\qquad\qquad\qquad\qquad\qquad\qquad\,\text{ on }\Gamma_v\cap (B_{1/2}\times (0,T)).
\end{aligned}
\right.
\]
\end{lemma}
\begin{proof}
{Suppose that for a smooth test function $\phi$, $v-\phi$ has a local maximum at $(x_0,t_0)\in \Gamma_v$ in $\overline{\Omega_v}\cap  B_{1/2}\times \{0\leq t\leq t_0\}$. We would like to verify the subsolution property for $\phi$.}

As done before, suppose $x_0=0$ and \eqref{2.2} holds, and let $\nu_*(x)$ be from \eqref{2.2'}, and $\nu(x):=\frac{\nu_*(x)}{|\nu_*(x)|}$. Then $v(0,t_0)=u(y_0,t_0)=0$ with $y_0:=\varphi(0,t_0)e_d$, $\nu(0)=e_d$, and $|\nabla\nu|\leq 1$. We now define 
\[
h(x,t):=x+\varphi(x,t)\nu(x)\quad (\text{ then }h(0,t_0)=y_0).
\]
If $\varphi$ has sufficiently small $C^1$ norm, $h$ is invertible and 
$h^{-1}$ is $C^{2,1}_{x,t}$. In particular  $\psi(y,t):=\phi(h^{-1}(y,t),t)$ is $C^{2,1}_{x,t}$ in a neighborhood of $(y_0,t_0)$.
Since $v(x,t)\geq u(h(x,t),t)$ and $v(0,t_0)=u(y_0,t_0)=0$,  $u-\psi$ has a local maximum in $\overline{\Omega_u}\cap \{t\leq t_0\}$ at $(y_0,t_0)$.

First, suppose that $-(\Delta\psi+f)(0,t_0)>0$. 
Since $u$ is a subsolution, $\psi$ satisfies
\beq\lb{2.80}
\psi_t(y_0,t_0)\leq |\nabla \psi(y_0,t_0)|^2+\vec{b}(y_0)\cdot\nabla \psi(y_0,t_0)
\eeq
in the classical sense.
By taking $\eps_1,\eps_2$ to be small enough, we get
\beq\lb{2.81}
\begin{aligned}
|\nabla \psi(y_0,t_0)-\nabla\phi(0,t_0)|&\leq \sup_{x\in B_{1/2}}\|D (h^{-1}(x,t_0))-I_d\||\nabla\phi(0,t_0)|\\
&\leq 2(\eps_1+\eps_2)|\nabla\phi(0,t_0)|
\end{aligned}
\eeq

Next we estimate $\phi_t(0,x_0)$. To do this, we first show that $\nabla\psi(y_0,t_0)$ is to the direction of $e_d$. Let us consider the set
\[
D:=\{x\,:\,|x-y_0|<\varphi(x,t_0)\},
\]
and then we have $0\in\partial{D}$ by $y_0\in \Gamma_u(t_0)$. 
Since $\varphi(\cdot,t_0)$ is $C^2$ and $\nabla\varphi(0,t_0)=\gamma_1e_1+\gamma_2e_d$ by \eqref{2.2}, $\partial{D}$ is $C^1$ and the inner normal direction at $0$ equals to $\gamma_1e_1+(1+\gamma_2)e_d$. Note that $v>0$ in $D$ by the definition of $v$ and $0\in\Gamma_v$. Thus we get $\phi(\cdot,t_0) > \phi(0,t_0)$ in $D\cap B_r(0)$ for some $r>0$, which implies that $\nabla\phi(0,t_0)$ is pointing to the direction of $\gamma_1e_1+(1+\gamma_2)e_d$.
In view of \eqref{B.10} and \eqref{B.5},  we have
\[
Dh(0,t_0)=
\begin{bmatrix}
1+\gamma_2 &  & && -\gamma_1\\
 & 1+\gamma_3 & & & \\
&  & \cdots & & \\
& & &1+\gamma_3& \\
\gamma_1 &  & & &1+\gamma_2
\end{bmatrix}.
\]
This and
\[
\frac{\gamma_1e_1+(1+\gamma_2)e_d}{\gamma_1^2+(1+\gamma_2)^2}|\nabla\phi (0,t_0)|=\nabla\phi (0,t_0)= \nabla\psi(y_0,t_0)\cdot Dh(0,t_0)
\]
yield that $\nabla\psi(y_0,t_0)=|\nabla\psi(y_0,t_0)|e_d$. 

Now, since $\varphi_t\leq \eps_3$, \eqref{2.81} shows that if $\eps_2$ is sufficiently small, 
\beq\lb{2.83}
\begin{aligned}
\phi_t(0,t_0)&=\psi_t(y_0,t_0)+\nabla\psi(y_0,t_0) \cdot h_t(0,t_0)\leq \psi_t(y_0,t_0)+[\nabla\psi(y_0,t_0)\cdot \nu(0)]\, \varphi_t(0,t_0)\\
&\leq \psi_t(y_0,t_0)+(\eps_3+2^{-1}|\eps_3|)|\nabla\phi(0,t_0)|.
\end{aligned}
\eeq
Then \eqref{2.80}, \eqref{2.81} and \eqref{2.83} yield at $(0,t_0)$,
\[
\phi_t\leq (1+2\eps_2)^2|\nabla\phi|^2+\vec{b}(y_0)\cdot \nabla\phi+\left(\eps_3+2^{-1}|\eps_3|+2\eps_2\|\vec{b}\|_\infty\right) |\nabla\phi|.
\]
Also using $|\vec{b}(y_0)-\vec{b}(0)|\leq \eps_1\|\nabla\vec{b}\|$ in the above  inequality and rearranging the terms, we obtain
\beq\lb{2.85}
\phi_t\leq(1+2\eps_2)^2|\nabla\phi|^2+\vec{b}(0)\cdot\nabla\phi+ \left(\eps_3+2^{-1}|\eps_3|+\eps_1\|\nabla\vec{b}\|_{\infty}+2\eps_2\|\vec{b}\|_{\infty}\right) |\nabla\phi|.
\eeq

Finally, if $-(\Delta\psi+f)(0,t_0)\leq 0$, it follows from the proof of Lemma \ref{L.2.7} with $\psi,\phi$ in place of $u,v$ that (note that $\phi(x,t_0)=\psi(x+\varphi(x,t_0)\nu(x),t_0)$ and we only used $v(x)\geq u(x+\varphi(x)\nu(x))$ in the proof before)
\[
-\Delta\phi(0)\leq (1+C \|\nabla\varphi\|_\infty)f(y_0)
\]
which finishes the proof.
\end{proof}

\begin{remark}
The conclusion of Lemma \ref{L.2.7} holds the same if $f=f(x,t)$ (then we replace $f(y(x,t))$ by $f(y(x,t),t)$). Similarly Lemma \ref{L.2.8} holds the same if $\vec{b}=\vec{b}(x,t)$ (in this case we don't need any regularity of $\vec{b}$ in $t$). The proofs are identical.
\end{remark}

%


Lastly, we describe a family of smooth functions $\varphi_\eta$, which will be used in the next section. The proof is {parallel} to that of  Lemma 10.10  in \cite{CafSal} with $\alpha:=4\kappa$. In the referenced lemma, the domain is assumed to be uniformly Lipschitz continuous in time.  However if we replace $L_1$ (the graph's Lipschitz constant in time) by $C/r=C\eps^{-{\gamma_1}}$, the same arguments apply to yield the desired estimates for our setting. 



\begin{lemma}\lb{L.4.1}
 Let $A_1>1$ be the dimensional constant from Lemma \ref{L.2.7}, and let $\gamma_1,{\gamma_2},\kappa\in (0,1)$ such that $\gamma_1-{\gamma_2}>4\kappa$. Take $r:=\eps^{\gamma_1}$ and $T\geq 4\eps^{4\kappa}$  with $\eps>0$ sufficiently small. 
{Suppose that $\Phi(x,t)$ is Lipschitz continuous with constant $C$ in space and with constant $C/r$ in time. } Then for
\[
\Sigma_{r,T}:=\{(x,t)\in B_1\times (-T,T)\,:\, |\Phi(x',t)-x_d|<2r\},
\]
there is $A_2=A_2(A_1,C)\geq 1$ such that for any $\eta\in [0,1]$, there exists a $C^2$ function $\varphi_\eta(x,t)$ in $\Sigma_{r,T}$ such that
\begin{enumerate}
    \item $0<\bar{c}\leq\varphi_\eta\leq 1+ \eta$ in $\Sigma_{r,T}$ for some universal constant $\bar c$,
    
    \smallskip
    
    \item $\varphi_\eta\Delta\varphi_\eta\geq A_1|\nabla\varphi_\eta|^2$ holds in $\Sigma_{r,T}$,
    \smallskip
    
    \item $\varphi_\eta\leq 1$ outside $\{(x,t)\in \Sigma_{r,T}\,:\,t>-T+\eps^{4\kappa},\,d(x,\partial B_{1})>\frac12\eps^{\kappa}\}$, 
    \smallskip

    \item $\varphi_\eta\geq 1+\eta(1-A_2\eps^{{\gamma_2}})$ in  $\{(x,t)\in \Sigma_{r,T}\,:\,t>-T+2\eps^{4\kappa},\,d(x,\partial B_{1})>\eps^{\kappa}\}$,
   \smallskip
    
    \item $|\nabla\varphi_\eta|\leq  A_2\eps^{{\gamma_2}-{\gamma_1}}$ and $0\leq \partial_t\varphi_\eta\leq  A_2\eps^{{\gamma_2}-{\gamma_1}}$ in $\Sigma_{r,T}$.
\end{enumerate}
\end{lemma}

In  the next section, we will choose  $\Phi_r$ to be the  Lipschitz function from Lemma \ref{L.3.11}, whose graph approximates the free boundary of $u$ up to order $r$.

\section{Flat free boundaries are Lipschitz}\lb{S.6}

In this section we will show by iteration that flat free boundaries are Lipschitz. Similar to \cite{CafSal} and \cite{CJK}, the proof is based on the construction of a family of subsolution, building on Section \ref{S.5}. These are constructed as a small perturbation of $u$ of \eqref{1.1} in a local domain $B_2(0)\times (-1,1)$. 

\medskip

{The family of subsolutions will represent the regularization mechanism of the flow, by the varying size of regularization given as a radius of the sub-convolution we apply to the solution. Due to the presence of the source and drift term with minimal regularity, and their competition with the regularization mechanism, there are additional terms to the perturbation: this makes the construction of barrier function rather versatile and technical. In an effort to make the construction more accessible for interested readers, we list the family of parameters in the next subsection. Readers may also choose to skip to our iterative statement, Proposition~\ref{L.2.12}, and the proof of our main theorem thereafter. }


\subsection{Parameters and Assumptions} 

Our barrier construction involves many parameters, which we put together here for the reader's convenience. First  we choose the minimal angle for the cone of monotonicity, from which we will apply our iteration arguments. In light of Lemma~\ref{L.3.4} we will assume that 
$$
f\in C^{\bar{\gamma}}(\R^d) \quad\hbox{ and }\quad \beta \in (1, 1+\bar \gamma).
$$

Let $\theta_{\beta}$ be as in \eqref{angle_1}, and let $\Theta_0\in (\theta_{\beta},\frac\pi2)$ be such that
\beq\lb{5.0}
\sin\theta_{\beta}< \sin\Theta_0-\cos^2\Theta_0/\bar{c}
\eeq
with $\bar c$ from Lemma \ref{L.4.1}. We will work with the cone angle $\theta \in (\Theta_0, \pi/2)$ in this section.

\medskip

Throughout this section we assume that $u$ satisfies the following for some $T\in (0,1]:$ 
\begin{itemize}
\item[(H-a)] $u(\cdot,t)$ is $(\eps,\eps^\al)$-monotone with respect to $W_{\theta,-e_d}$ in $B_1$ for some $\theta \in(\Theta_0, \pi/2)$ for all $t\in (-T,T)$, and with $\alpha$ satisfying 
\beq\lb{10.13}
0<\alpha <\min\left\{\frac{\bar\gamma^2}{2},\,\frac{\bar\gamma(1-\bar\gamma)}8,\,\frac{1-\bar\gamma^2}{16}\right\}
\eeq 
when $f$ is not a constant.

\smallskip
\item[(H-b)] $(0,0) \in \Gamma_u$ and $m:=\inf_{t\in (-T,T)}u(-e_d,t)>0$. 
\end{itemize}
Note that if $T<1$, a simple rescaling argument can reduce the problem to the case of $T=1$. So, for simplicity of notations, let us assume $T=1$ from now on. 

\medskip

Let us proceed with the next set of parameters, to be used in the next subsection for the construction of barrier functions.  For $\kappa:=\frac{2-\beta}{8}$, we choose $\gamma_1$ and  $\iota$ such that $\gamma_1$ and $\iota$ are respectively close to $1$ and $4\kappa$, and $\beta+ \iota <  2\gamma_1$. More specifically we choose 
\begin{equation}\lb{10.14}
\gamma_1:=\max\left\{\frac{3}{4}+\frac{\beta}{8},1-\frac{\bar\gamma}{2}\right\}<1, \quad \iota:=5\kappa=\frac{5}{4}-\frac{5\beta}{8} \quad \hbox{ and } \quad \gamma_2:= \gamma_1 {-}\iota,
\end{equation}
and so $\gamma_1 - \gamma_2 > 4\kappa$. With this choice of $\kappa, \gamma_1$ and $\gamma_2$,  let $\varphi_\eta$ be from Lemma \ref{L.4.1}, with some $\eta\in (0,1)$. 

\medskip

We also define $0<\al_1 < \al_2<1$ so that
\beq\lb{10.10}
1-\gamma_1 < \alpha_1 < 1-{(\beta + \iota)}/{2}, \quad   \al_2< \min\{1-\iota,\bar\gamma\}.
\eeq
Note that this is possible since $\beta+ \iota < 2\gamma_1 $ and $\max\{\iota,1-\bar\gamma\} < \gamma_1$.

Lastly we define universal constants: in this section $C$ or $c$ denotes constants
that only depend on $d,\alpha,\beta, m,\bar\gamma$, $\|u\|_{L^\infty(B_2(0)\times (-1,1))}$, $\|f\|_\infty$, $\|f\|_{C^{0,\bar\gamma}(B_2)}$, and $\|\vec{b}\|_{C^1(B_2)}$.

\medskip

\subsection{Construction of the base barrier $\bar{v}$}

Let us define $r:=\eps^{\gamma_1}$. We will construct our subsolution in the domain 
\[
\Sigma_{r,1}=\{(x,t)\in \calQ_1 \,:\, |\Phi_r(x',t)-x_d|<2r\},
\]
where $\Phi_r$ is as given in Lemma \ref{L.3.11} {that approximates $\Gamma_u$ in $r$-scale.}

\medskip

For a given $\sigma\in [\frac{\eps}{2},\eps]$, define
\beq\lb{5.1}
v(x,t):=\sup_{B_{\sigma\varphi_\eta(x,t)}(x)}u(y,t).
\eeq
{It then follows from \eqref{5.0}, Lemma \ref{L.2.6} and Lemma \ref{L.4.1} that $v(\cdot,t)$ is non-decreasing along all directions of $W_{\theta_{\beta},-e_d}$ when $\eps$ is sufficiently small.}
With this choice of $v$, we define the domains
\[
\Sigma^+:=\Sigma_{r,1}\cap\Omega_v,\quad \Sigma^+(t):=\{x\,:\,(x,t)\in\Sigma^+\},\quad \Sigma_{r,1}(t):=\{x\,:\,(x,t)\in\Sigma_{r,1}\}.
\]
and the bottom boundary of $\Sigma^+(t)$ as
\[
\partial_b \Sigma^+(t):= \left(\partial \Sigma_{r,1}(t)\cap\Omega_v\right)\backslash \partial B_1.
\]
Due to the presence of $f$, we need to adjust $v$ and the adjustments are superharmonic functions.

\medskip

For each $t\in (-1,1)$, let us define $w_1^t$ and $w_2^t$ by:
\begin{enumerate}
    \item $-\Delta w_1^t=0$ in $\Sigma^+(t)$  with $w_1^t=v(\cdot,t)$ on $\partial_b \Sigma^+(t)$ and zero elsewhere on the boundary.
    
    \smallskip
    
    \item $-\Delta w_2^t=1+\|f\|_\infty$ in $\Sigma^+(t)$ with zero boundary condition.
\end{enumerate}

\medskip

Consider a non-negative harmonic function $\phi$ in the annulus
$(B_{1}\backslash B_{1-2\eps^\kappa})\cap\Sigma^+(t)$.
Since $\e^{\kappa} << r$,  
if $\phi\geq v$ on $\partial_b\Sigma^+(t)\cap (B_{1-\eps^\kappa/2}\backslash B_{1-2\eps^\kappa}) $,  then Dahlberg's lemma yields a dimensional constant $c_*$ such that 
\beq\lb{10.16}
c_* w_1^t\leq \phi\quad \text{ on }\Sigma^+(t)\cap \partial B_{1-\eps^\kappa}
\eeq

\medskip

With this choice of $c_*$, we finally define our barrier function by
\beq\lb{10.15}
\bar{v}(\cdot,t):=(1+\eps^{\alpha+1})v(\cdot,t)-\eps^{\al_2} w_2^t+c_*\eps^{\al_1} w_1^t,
\eeq
where $\alpha_1, \alpha_2$ are given in \eqref{10.10}.
\begin{lemma}\lb{L.5.2}

For sufficiently small $\e>0$, $\bar{v}$ given by \eqref{10.15} satisfies the following in the viscosity sense: For any $e\in B_1$,
\[\left\{
\begin{aligned}
-\Delta\bar{v}&\leq f(x-\eps e) &&\quad\text{ in } \Sigma^+\cap (B_{1-\eps^\kappa}\times (-1,1)),\\
    \bar{v}_t&\leq |\nabla \bar{v}|^2+\vec{b}(x-\eps e)\cdot\nabla \bar{v}&&\quad\text{ on }\Gamma_{\bar{v}}\cap  (B_{1-\eps^\kappa}\times (-1,1)).
\end{aligned}
\right.\]
\end{lemma}

\begin{proof}
Since $\sigma|\nabla\varphi_\eta|\leq A_2\eps^{1-\iota}$ by Lemma \ref{L.4.1}, $\kappa<1$ and $\sigma\leq \eps$,  the proof of Lemma \ref{L.2.7} yields for small $\eps$,
\beq\lb{4.21}
-\Delta v(x,t)\leq (1+A_1A_2\eps^{1-\iota})\sup_{B_{\sigma\varphi_\eta(x,t)}(x)}f(y).
\eeq
Using $ \| f\|_{C^{0,\bar\gamma}}\leq C$ and $\varphi_\eta\leq 2$, the right-hand side of the above
\beq\nonumber
\begin{aligned}
&\leq 
(1+A_1A_2\eps^{1-\iota})(f(x-\eps e)+C\eps^{\bar{\gamma}})\\
&\leq (1+A_1A_2\eps^{1-\iota})f(x-\eps e)+C\eps^{\bar\gamma}.
\end{aligned}
\eeq
From \eqref{10.15} and the fact that $\alpha_2<\min\{1-\iota,\bar\gamma\}$, we obtain for all $\eps>0$ sufficiently small that
\begin{align*}
-\Delta\bar{v}- f(x-\eps e)&\leq- (1+\eps^{\alpha+1})\Delta v-\eps^{\al_2} - f(x-\eps e)\\
&\leq C\eps^{1-\iota}f(x-\eps e)+C\eps^{\bar{\gamma}}-\eps^{\al_2}\leq 0\qquad\text{ in }\Sigma^+(t)\cap B_1.
\end{align*}

It remains to show the appropriate free boundary condition. By Lemma \ref{L.2.6} and the choice of $\Theta_0$, for each $t\in (-1,1)$, $\Gamma_{v}(t)$ is a Lipschitz graph with Lipschitz constant less than $\cot\theta_{\beta}$ when $\eps$ is small. 
Then, using $u(-e_d,t)\geq m$, Corollary \ref{C.2.7} with $\delta:=\eps^{\gamma_1}<\eps^\kappa$ yields for some $C>0$,
\[
C\eps^{\gamma_1(2-\beta)} w_1^t\geq w_2^t\quad\text{ in }\Sigma^+(t)\cap B_{1-\eps^\kappa}.
\]
Thus we have, for $\eps$ sufficiently small,
\beq\lb{10.12}
c_*\eps^{\alpha_1} w_1^t>> \eps^{\alpha_1}w_2^t>> \eps^{\al_2}w_2^t\quad\text{ in }\Sigma^+(t)\cap B_{1-\eps^\kappa}.
\eeq
Next, $-\Delta v\leq 1+\|f\|_\infty$ by \eqref{4.21} for small $\eps$, the construction of $w_1^t$ and $w_2^t$, and Dahlberg's Lemma imply for some $C>1$,
\beq\lb{10.11}
Cw_1^t+w_2^t\geq  v(\cdot,t)\quad\text{ in }\Sigma^+(t)\cap B_{1-\eps^\kappa}.
\eeq
So for all $\eps$ small enough, $\al_2\geq \al_1$,
\eqref{10.15}, \eqref{10.12} and \eqref{10.11} show (for $c_1:=c_*/(4C)>0$)
\beq\lb{10.7}
\begin{aligned}
\bar{v}&\geq (1+\eps^{\alpha+1})v-(\eps^{\al_2}+c_*\eps^{\al_1}/(2C))w_2^t+c_*\eps^{\al_1}w_1^t/2+c_*\eps^{\al_1}v/(2C)\\
&\geq (1+2c_1\eps^{\al_1})v-C'\eps^{\al_1}w_2^t+c_*\eps^{\al_1}w_1^t/2\\
&\geq (1+2c_1\eps^{\al_1})v\qquad\qquad\text{ in }\Sigma^+(t)\cap B_{1-\eps^\kappa}.
\end{aligned}
\eeq

We then show that $\bar{v}$ has a linear growth near the free boundary.
For $x_0\in \Gamma_{\bar{v}}(t_0)\cap B_{1-\eps^\kappa}$ and $t_0 \leq 1$, since $x_0\in \Gamma_{\bar v}(t_0)=\Gamma_{v}(t_0)$, there exists $y_0\in \Gamma_{u}(t_0)\cap \overline{B_{\sigma\varphi_\eta(x_0,t_0)}(x_0)}$. By the definition of sup-convolution, $B_{\sigma\varphi_\eta(x_0,t_0)}(y_0)\subseteq \Omega_v(t_0)$. This means that $\Gamma_{{v}}(t_0)$ satisfies the interior ball property at $x_0$:
\[
B_{\bar{c}\eps/2}(y')\subseteq\Omega_{v}\quad\text{and}\quad x_0\in\partial B_{\bar{c}\eps/2}(y') \cap \Gamma_v(t_0)
\]
for some $y'$. Thus $\bar v\geq \frac{c_*\eps^{\alpha_1}}2 w_1^t$ (which easily holds for sufficiently small $\eps$ by Corollary \ref{C.2.7}) implies that $\bar{v}$ grows at least linearly at $(x_0,t_0)$. Moreover, we use $u(-e_d,t)>0$ and Lemma \ref{C.2.8} to obtain that 
\[
C\max_{B_{3\eps}(x_0)}w_1^{t_0}\geq \max_{B_{3\eps}(x_0)}v(\cdot,t_0)\geq {c}\,\eps^\beta
\]
for some universal ${c}>0$. 
It then follows from the interior ball property, and Dahlberg's Lemma that 
\beq\lb{10.21}
|\nabla\bar v(x_0,t_0)|\geq |\nabla w_1^{t_0}(x_0)|\geq {c}\,\eps^{\beta-1}\quad\text{ with possibly different universal ${c}>0$}.
\eeq

Now we check the viscosity subsolution property for $\bar{v}$ at the free boundary. Suppose that  a test function $\phi$  crosses $\bar v$ from above at $(x_0,t_0)$. 
The linear growth of $\bar{v}$ yields $|\nabla\phi(x_0,t_0)|\neq 0$.
Due to \eqref{10.7} and the fact that $\bar{v}(x_0,t_0)=v(x_0,t_0)=0$, $ (1+2c_1\eps^{\al_1})v-\phi$ has a local maximum at $(x_0,t_0)$ as well.
Using Lemma \ref{L.2.66}, Lemma \ref{L.2.8} and Lemma \ref{L.4.1} yields that $\phi$ satisfies
\[
\phi_t\leq (1+C\eps^{1-\iota})^2(1+2c_1\eps^{\al_1})^{-1}|\nabla \phi|^2+\vec{b}(\cdot-\eps e)\cdot\nabla \phi+C\eps^{1-\iota}|\nabla\phi|\quad\text{ at }(x_0,t_0)
\]
for some $C=C(A_2,\|\vec{b}\|_{C^1})$.
Since $1-\iota>\al_1$, we get for small $\eps>0$,
\beq\lb{10.9}
\phi_t\leq (1-c_1\eps^{\al_1})|\nabla \phi|^2+\vec{b}(\cdot-\eps e)\cdot\nabla \phi+C\eps^{1-\iota}|\nabla\phi|\quad\text{ at }(x_0,t_0).
\eeq

Next, since $\eps^{\al_2} w_2^t\leq v$ and $\bar{v}-\phi$ obtains a local maximum in $\overline{\Omega_v}\cap\{t\leq t_0\}$ at $(x_0,t_0)$, then $c_*\eps^{\al_1} w_1^t-\phi$ has a local maximum in the same domain at $(x_0,t_0)$ as well. Combining this with \eqref{10.21} shows
\[
|\nabla\phi(x_0,t_0)|\geq cc_*\eps^{\beta-1+\al_1}.
\]
So $\eps^{\al_1}|\nabla\phi(x_0,t_0)|>> \eps^{1-\iota}$ since {$\beta +  2\alpha_1 +\iota <2$ by \eqref{10.10}}. 
From this and \eqref{10.9}, we obtain
\[
\phi_t\leq |\nabla \phi|^2+\vec{b}(\cdot-\eps e)\cdot\nabla \phi\quad\text{ at }(x_0,t_0),
\]
and thus the subsolution property is verified.
\end{proof}

\subsection{Flat free boundary is Lipschitz}

Now we can prove the following main inductive proposition.
\medskip

\begin{proposition}\lb{L.2.12}
Under the assumptions {\rm (H-a)(H-b)} and for any fixed $\kappa\in (0,\frac{2-\beta}4)$, there exist $C>0$ and $j,\delta,\gamma_3\in (0,1)$ such that if $\eps>0$ is sufficiently small, $u$ is $(j\eps,\eps^\alpha(1-C\eps^\delta))$-monotone with respect to $W_{\theta-C\eps^{\gamma_3},-e_d}$ in $ B_{1-\eps^\kappa}\times (-1+2\eps^{4\kappa},1)$.
\end{proposition}
\begin{proof}

First we choose $\sigma$ and $\eta$ in the definition of the sup-convolution in \eqref{5.1}.
Since $\theta>\Theta_0\geq\pi/4$, we can take $j\in (0,1)$ so that
\[
\sigma:=\eps(\sin\theta-(1-j))\in ({\eps}/{2},\eps).
\]
 Define
\begin{equation}\label{gamma_3}
\gamma_3:=\min\left\{\frac{\al_1+\gamma_1-1}{2},\gamma_2\right\}.
\end{equation} 
Observe that $\gamma_3\in (0,1)$ by \eqref{10.10} and $\gamma_3+1-\gamma_1\in(0,\al_1)$. Choose $\eta=\eta_\eps>0$ such that \beq\lb{10.6}
\left(1+\eta\right)(\sin\theta-(1-j))= j\sin\theta-\eps^{\gamma_3}.
\eeq

By taking $\eps>0$ to be small enough, we have
\[
\eta\in \left(\frac{j\sin\theta-(\sin\theta-(1-j))}{2(\sin\theta-(1-j))},\frac{j\sin\theta-(\sin\theta-(1-j))}{\sin\theta-(1-j)}\right).
\]
It follows from Lemma \ref{L.4.1} (4), \eqref{gamma_3} and \eqref{10.6} that $\varphi_{\eta}$ in \eqref{5.1} then satisfies 
\beq\lb{10.2}
\sigma\varphi_\eta\geq \sigma(1+\eta(1-A_2\eps^{\gamma_2}))\geq \eps(j\sin\theta-C\eps^{\gamma_3})\, \text{ in }\Sigma_{r,1}\cap\{t>-1+2\eps^{4\kappa},\,d(x,\partial B_1)>\eps^\kappa\}
\eeq
for some $C=C(A_2)\geq1$, and by Lemma \ref{L.4.1} (3),
\beq\lb{10.3}
\varphi_\eta\leq 1\quad\text{ in }\Sigma_{r,1}\cap \{ t<-1+\eps^{4\kappa},\,d(x,\partial B_1)<\eps^\kappa/2\}.
\eeq

With above choice of $\sigma$ and $\eta$, we claim that
\beq\lb{10.5}
\bar{v}\leq \bar{u}:=u(x-j\eps e_d,t)\quad \text{ in } ( B_{1-\eps^{\kappa}}\times (-1,1))\cap \Sigma^+.
\eeq
Before showing this claim, we first discuss its consequence. 
It follows from \eqref{10.12} and \eqref{10.2} that for $(x,t)\in  (B_{1-\eps^\kappa}\times (-1+2\eps^{4\kappa},1))\cap \Sigma^+$,
\[
(1+\eps^{\alpha+1})\sup_{B_{ \eps(j\sin\theta-C\eps^{\gamma_3})}(x)}u(y,t)\leq \bar{v}(x,t)\leq u(x-j\eps e_d,t)
\]
which yields the conclusion for those $(x,t)$. Next for $(x,t)\in ( B_{1-\eps^\kappa}\times (-1+2\eps^{4\kappa},1))\backslash \Sigma^+$, we have $d(x,\Gamma_u(t))\geq \eps^{\gamma_1}$, and thus Lemma \ref{C.2.8} and the $(\eps,\eps^{\alpha})$-monotonicity yield that $u(x,t)\geq c\,\eps^{\gamma_1\beta}$.

Recall the choice of the parameters: $\beta\in (1,1+\bar\gamma)$, $\kappa=\frac{2-\beta}8$ and \eqref{10.14}.  Therefore by taking $\kappa_1:=\frac{1+\bar\gamma-\gamma_1\beta}{
4}\in(0,\frac12)$ and using \eqref{10.13}, we have $1-\kappa_1 > \kappa$ and $\gamma_1\beta+\alpha<1+\bar\gamma-2\kappa_1$. Then for small $\eps$,
\[
\eps^{\alpha} u(x,t)\geq c\,\eps^{\gamma_1\beta+\alpha}>>\eps^{1+\bar\gamma-2\kappa_1}\quad\text{ in $(B_{1-\eps^\kappa}\times (-1+2\eps^{4\kappa},1))\cap \Sigma^+$},
\]
and so the third remark of Lemma \ref{L.2.4} concludes the proof of Proposition~\ref{L.2.12}  with $\delta:=\kappa_1$.
 
\medskip

It remains to prove \eqref{10.5}. To do this, we claim that it suffices to show that for each $t\in (-1,1)$, 
\beq\lb{10.100}
\bar{v}(\cdot,t)\leq  \bar{u}(\cdot,t) \hbox{ on } (\partial B_{1-\eps^\kappa}\cap \Sigma^+(t))\cup (\partial_b\Sigma^+(t)\cap B_{1-\eps^\kappa}).
\eeq
Indeed, by \eqref{10.3} and (H-a), when $t<-1+\eps^{4\kappa}$, we have $\Sigma^+(t)=(\Omega_{\bar{v}}(t)\cap \Sigma_{r,1}(t))\subseteq \{\bar{u}(\cdot,t)>0\}$. Then \eqref{10.100} and the comparison principle for Laplacian yield $\bar{v}\leq  \bar{u}$ in $\Sigma_{r,1}\cap  B_{1-\eps^\kappa}\times \{t<-T+\eps^{4\kappa}\}$. This and \eqref{10.100} show that $\bar{v}$ and $\bar{u}$ are ordered on the parabolic boundary of $\Sigma_{r,1}$. 
In view of Lemma \ref{L.5.2} with $e:=je_d$, and the equation that $\bar u$ satisfies, we want to apply the comparison principle to conclude with \eqref{10.5}.
To do this rigorously, we replace $\theta$ by a slightly smaller $\theta'$ at the beginning of the proof, the supports of $\bar v$ and $\bar{u}$ are then separated (because if $z\in \Gamma_{\bar{v}}$, then the definitions of $v$ and $\bar{v}$, and (H-a) yield $z\in\Omega_{\bar v}$). The strict order of $\bar{v}$ and $\bar{u}$ in one of their support follows easily from the proof below.
Then the comparison principle Lemma \ref{L.cp} can now yield \eqref{10.5} after passing $\theta'\to\theta$.

\medskip

Now we show \eqref{10.100}. For any $t\in(-1,1)$ and $x\in \partial_b\Sigma^+(t)\cap B_{1-\eps^\kappa/2}$, since $x$ is at least $\eps^{\gamma_1}$-away from $\Gamma_v(t)$, Lemma \ref{l.2.2} and Lemma \ref{C.2.8} yield that $ \inf_{y\in B_{\eps}(x)}u(y,t)\approx u(x,t)\geq c\,\eps^{\gamma_1\beta}$. Also note that by \eqref{10.13} and \eqref{10.14} (when $f$ is not a constant), there exists $\kappa_2$ such that $\frac\alpha{\bar{\gamma}}<\kappa_2<\min\{\frac12,1-\gamma_1\}$.
Thus by Lemma \ref{L.3.4} and \eqref{10.6},
we have
\begin{align*}
v(x,t)&\leq\sup_{B_{(1+\eta)\sigma}(x)}u(y,t)\leq \sup_{B_{j\eps\sin\theta}(x)}u(y,t)-(j\eps\sin\theta-(1+\eta)\sigma)\inf_{y\in B_{j\eps\sin\theta}(x)}|\nabla u(y,t)|\\
&\leq \sup_{B_{j\eps\sin\theta}(x)}u(y,t)-C\eps^{\gamma_3+1-\gamma_1}\inf_{y\in B_{j\eps\sin\theta}(x)} u(y,t)\leq (1-C\eps^{\gamma_3+1-\gamma_1})\bar{u}(x,t).
\end{align*}
The last inequality is due to the full monotonicity of $u$ in the interior (the second remark after Lemma \ref{L.2.4}) since $\gamma_1\beta+\alpha<1+\bar{\gamma}-\kappa_1$.
Next using $w_t(x) \leq v(x,t)$, it follows that
\beq\lb{10.17}
\begin{aligned}
\bar{v}(x,t)& \leq (1+\eps^{\alpha+1}+c_*\eps^{\al_1})v(x,t)\\
&\leq (1+C\eps^{\al_1})(1-C\eps^{\gamma_3+1-\gamma_1})\bar{u}(x,t)\leq \bar{u}(x,t)
\end{aligned}
\eeq
when $\eps$ is sufficiently small.


Now we consider $(x,t)\in (\partial B_{1-\eps^\kappa}\times (-1,1))\cap \Sigma^+$. We define the following region that contains $x$:
\[
 \tilde\partial_l\Sigma^+(t):= (B_{1}\backslash B_{1-2\eps^\kappa})\cap\Sigma^+(t).
\]
The construction of $\varphi_\eta$ yields that $\varphi_\eta(\cdot,t)\leq 1$ in $\tilde\partial_l\Sigma^+(t)$. 
Since $B_{\sigma}(x+j\eps e_d)\subseteq B_{\eps\sin\theta}(x+\eps e_d)$ by the definition of $\sigma$, the $(\eps,\eps^\al)$-monotonicity assumption yields that
\beq\lb{10.4}
\bar{u}(\cdot,t)\geq (1+\eps^{\alpha+1})\sup_{B_\sigma(\cdot)}u(y,t)\geq (1+\eps^{\alpha+1})v(\cdot,t)\quad \text{ on }\tilde\partial_l\Sigma^+(t).
\eeq
 Due to Lemma \ref{L.5.2} and $\Delta w_1^t=0$, $\bar{u}-(1+\eps^{\alpha+1})v+\eps^{\al_2} w_2^t\geq 0$ is superharmonic. Note that \eqref{10.17} implies 
\[
\eps^{\al_1}w_1^t= \eps^{\al_1}v \leq \bar{u}-(1+\eps^{\alpha+1})v+\eps^{\al_2} w_2^t
\]
on $\partial_b\Sigma^+(t)\cap (B_{1-\eps^\kappa/2}\backslash B_{1-2\eps^\kappa}) $.
Therefore, the choice of $c_*$ and \eqref{10.16} yield
\[
c_*\eps^{\al_1} w_1^t(x)\leq \bar{u}(x,t)-(1+\eps^{\alpha+1})v(x,t)+\eps^{\al_2} w_2^t(x).
\]
We obtain $\bar{v}(x,t)\leq \bar{u}(x,t)$ again.
Overall, we showed \eqref{10.100} which implies \eqref{10.5} and finishes the proof.

\end{proof}

{\bf Proof of of Theorem \ref{T.2.1}:}
Let us fix $(x_0,t_0)\in \Gamma_u$, which we may assume to be the origin. 
Applying Lemma \ref{L.3.11} with some $r>0$, the free boundary at any time $t\in (-T,T)$ is contained in a $(r+CT/r)$-neighborhood of a $\cot\theta$-Lipschitz graph. Thus it can not move too far away from $t=0$ when $r$ and then $T$ are sufficiently small.
Then by the assumption, after rescaling and rotating, we can assume that the conditions of Proposition \ref{L.2.12} hold. 

Iterating Proposition \ref{L.2.12}, we generate a sequence of domains
\[
\calQ^k:= B_{R_k}\times (-T_k,1)
\]
where $T_k=1-2\Sigma_{n=1}^k (j^n\eps)^{4\kappa}$, $R_k:=1-\Sigma_{n=1}^k (j^n\eps)^{\kappa}$, in which $u$ is $(j^k\eps,\alpha_k)$-monotone with respect to the cone $W_{\theta_k,-e_d}$
where
\[
\theta_k:=\theta-C\Sigma_{n=1}^k (j^n\eps)^{\gamma_3},\quad \alpha_k=\eps^{\alpha}(1-C\Sigma_{n=1}^k (j^n\eps)^\delta).
\]
We claim that for each iteration, the constants $C,j,\gamma_3,\delta$ can be chosen the same. Indeed, by taking $\eps$ to be further small enough and $\theta>\Theta_0$, we have for all $k\geq 1$,
\[
T_k\geq 1/2,\quad R_k\geq 1/2, \quad \theta_k\geq \Theta_0,\quad \alpha_k\geq {\eps^\alpha}/{2}.
\]
The claim follows from the proof of Proposition \ref{L.2.12}.

Finally we obtain that $u$ is monotone non-decreasing in all directions of $W_{\Theta_0,-e_d}$ in $ B_{1/2}\times (-\frac12,1)$.
The last statement of the theorem follows from Corollary \ref{C.6.5} below.
The proof is then completed.

\hfill$\Box$

\subsection{$C^{1,\gamma}$ free boundary when $\vec{b}\equiv0$.}\lb{subs6.5}

In the zero-drift case, the support of $u$ increase over time. Using this fact, it is well-known that we can obtain an obstacle problem by integrating $u$ over time (for instance see \cite{EJ} for the classical setting). We will utilize this fact to derive further regularity result. 
\begin{corollary}\lb{C.6.5}
For $f \geq 0$ and $\vec{b}\equiv 0$, let $u$ be a viscosity solution to \eqref{1.1} in $\bbR^d\times (-2,2)$ with bounded support. Suppose the assumptions of Theorem \ref{T.2.1} hold in $\calQ_2$. Then there exists $0<\gamma<1$ such that  $\Gamma_u(t)$ is $C^{1,\gamma}$ in $B_1$ for each $t\in (-1,1)$.  
\end{corollary}
\begin{proof}
Since $\vec{b}\equiv 0$, $\Omega_u$ is non-decreasing in time.  and define   $w(x):=\int_{-2}^{t}u(x,s)ds.$  Since the positive set of $u$ expands over time, we have $\Omega_w(t) = \Omega_u(t)$ for each $t> -2$, so it suffices to show that $\Omega_w(t)$ is $C^{1,\gamma}$ in $B_1$ and for each $t\in (-1,1)$. 

\medskip

Since our solution is coming from a globally defined solution, it follows from \cite[Theorem 1.1]{kimsingular} and \cite{David_S} that the viscosity solution coincides with the weak solution of the divergence form equation
$$
(\chi_{\{u>0\}})_t - \Delta u = f \chi_{\{u>0\}} \hbox{ in } \R^d\times (-2,2).
$$

From this weak formulation one can then check that  $w(\cdot,t)$  solves the obstacle problem:
\[
 [1-F(x,t)]\chi_{\{w>0\}} - \Delta w =0 \hbox { in } \R^d. 
\]
for each $t >-2$, where $F(x,t):= (t-T(x))f(x)$ and the {\it hitting time} $T: \R^d \to [0,\infty]$ is given by 
\[
T(x):=\inf\left\{t\geq -2 \,:\, u(x,t)>0\right\}.
\]

Theorem \ref{T.2.1} and Proposition \ref{L.3.3} yields that $T(x)$ is H\"{o}lder continuous in $B_1$ near $\Gamma_w(t)$, and thus so is $F(x,t)$. Since we already know that the free boundary of $w$ has no cusp singularity, we can conclude from \cite[Theorem 7.1]{blank2001sharp} that 
$\Gamma_w(t) \cap B_1$ is  $C^{1,\gamma}$ for each $t\in (-1,1)$ for some $\gamma$. 
\end{proof}

\begin{remark}
We expect the corollary to hold for local solutions $u$ in $\calQ_2$ in general, but the corresponding proof requires coincidence of the notions used in weak and viscosity sense in bounded domains with fixed boundary data. We do not pursue it here.
\end{remark}

\subsection{Strict expansion along the streamline} 

We finish the section by establishing a uniform, yet sublinear, rate of expansion for the positive set $\Omega_u$ along the streamline (so for general Lipschitz $\vec{b}$).



\begin{definition}\lb{D.3.2}
We say that the set $\Omega_u$ is {\it strictly expanding relatively to the streamlines} in $\calQ_r$, if for all small $t>0$ there exists $r_t>0$ such that for any $(x_0,t_0)\in \Gamma_u\cap\calQ_r$ we have
\[
B_{r_t}(X(t;x_0))\cap\calQ_r\subseteq \Omega_u(t_0+t).
\]
\end{definition}

Note that this property is stronger than the conclusion of Lemma \ref{L.2.9}, but is weaker than non-degeneracy.
 
\medskip

If the free boundary is Lipschitz continuous, we can quantify the amount of expansion of the free boundary relatively to streamlines based on Lemma \ref{L.2.10}.

\begin{proposition}\lb{L.3.3}
Suppose that in $\calQ_1$, $u(\cdot,t)$ is non-decreasing with respect to $W_{\theta,-e_d}$ for some $\theta \in (\theta_{\beta}, \frac\pi2)$ and $\beta\in (1,2)$, and 
\beq\lb{6.90}
(0,0) \in \Gamma_u\quad\text{and}\quad \inf_{t\in (-1,1)}u(-e_d,t)>0.
\eeq
Then $\Omega_u$ expands strictly relatively to the streamlines in $\calQ_{1/2}$ with $r_t = ct^{1/(2-\beta)}$ for some $c>0$.

\end{proposition}
\begin{proof}
Let us only prove the lemma for $d\geq 3$ and the proof for $d=1,2$ is similar. Let $(x_0,t_0)\in\Gamma_u$, and after shifting, we assume $(x_0,t_0)=(0,0)$. Next we define $\bar u$ from \eqref{3.0} which solves \eqref{3.1}. Lemma \ref{L.2.9} yields that $0\in\overline{\{\bar{u}(\cdot,t)>0\}}$ for all $t>0$. Thus, the monotonicity assumption yields 
$$
W_{\theta,-e_d}\cap B_{1}\subseteq \{\bar{u}(\cdot,t)>0\} \hbox{ for } t>0.
$$ Since $\theta\geq\theta_{\beta}$, this,  \eqref{6.90} and Lemma \ref{L.2.10} imply that 
there exists $c>0$ such that
\beq\lb{6.13}
\bar{u}(\cdot,t )\geq cr^\beta\quad\text{ in $B_{2r}(-3re_d)$ for all }r\in(0,1).
\eeq

Now take $P:=-3re_d$ for some $r>0$, and for $c$ from \eqref{6.13} define
\[
\varphi(x,t):=c r^{d-2+\beta}(|x-P|^{2-d}-R(t)^{2-d})\quad\text{ with }R(t):=(c_1 r^{d-2+\beta}t+(2r)^d)^\frac{1}{d}
\]
with $c_1:=2cd(d-2)$. Then for each $t>0$, $\varphi(\cdot,t)$ is a non-negative harmonic function in $B_{R(t)}(P)\backslash B_r(P)$ such that $\varphi(\cdot,t)=0$ on $\partial B_{R(t)}(P)$ and $\varphi(\cdot,t)\leq cr^\beta$ in $B_{R(t)}(P)\backslash B_r(P)$. Thus, in view of \eqref{6.13}, we obtain
\beq\lb{6.5}
\varphi(x,t)\leq \bar{u}(x,t)\quad  \text{ in }( B_{2r}(P)\backslash B_r(P))\times \{0\} \cup \{(x,t)\,:\,t\in(0,1),x\in\partial B_r(P) \}.
\eeq

Note that $R(t_*)= 4r$ with
$
t_*:= (4^d-2^d)r^{2-\beta}/c_1\leq 1$ when $r$ is small. 
So by the definition of $\varphi$, if we can show $\varphi\leq \bar{u}$ for all $ t\in [0,t_*]$ and $x\in B_{R(t)}(P)\backslash B_r(P)$, then $B_{r}(0)\subseteq \Omega_{\bar{u}}(t_*)$ which concludes the proof.

\medskip

To do this, in view of \eqref{6.5}, $ f_0\geq 0$ and the comparison principle, it remains to show that $\varphi$ satisfies the appropriate boundary condition on $|x|=R(t)$.
Indeed, direct computation yields
\[
R'(t)=c_1 r^{d-2+\beta}R(t)^{1-d}/d,\quad |\nabla\varphi(x,t)|=c(d-2)r^{d-2+\beta}|x-P|^{1-d}.
\]
Also, by $|\vec{b}_0(x,t)|\leq \|\nabla\vec{b}\|_\infty|x|$ and the choice of $c_1$, we obtain for $x\in \partial B_{R(t)}(P)$ and $t\in [0,t_*]$ that
\[
R'(t)-|\nabla\varphi(x,t)|-
|\vec{b}_0(x,t)|\leq c(d-2)r^{d-2+\beta}R(t)^{1-d}- Cr\leq C'r^{\beta-1}-Cr
\]
which is non-positive  if $r>0$ is sufficiently small. This shows that $\varphi$ is a subsolution to \eqref{3.1}, with Dirichlet boundary condition on $ \partial B_r(P)\times [0,t_*]$, in $(\bbR^d\backslash B_r(P))\times [0,t_*]$. Now we can conclude.
\end{proof}

\section{Non-degeneracy}\lb{S.7}
The goal of the section is to show non-degeneracy result under additional assumptions. Let us illustrate the outline of the proof, in the setting where there is no drift, namely when $\vec{b}=0$. Due to the cone-monotonicity proven in the previous section, the free boundary of $u$ is a Lipschitz graph with respect to $e_d$ direction, and $u$ is non-decreasing with respect to a cone $W_{\theta,-e_d}$:
\begin{equation}\label{monotone}
 \sup_{|y-x|< r\e} u(y + \e e_d, t )\leq u(x,t+C\e) \hbox{ with } r=\sin\theta\hbox{ and a uniform constant } C.
\end{equation}
Our claim is that, if the above inequality is true in a unit space-time neighborhood of a free boundary point $x_0$ , then by the time $t=t(e_d)$ the free boundary reaches the point $x_0 + e_d$, the constant $r$ in \eqref{monotone} increases to a constant strictly larger than $1$ near the free boundary. In heuristic terms the claim states that the monotonicity of the solution propagates and improves over time in both space and time variable, as the positive set expands out toward $e_d$ direction. Observe that the claim implies that for some $r'>0$ we have
$$
\sup_{|y-x| < r' \e} u(y,t) \leq  u(x, t+C\e) \hbox{ in a small neighborhood of } (x_0+e_d,t(e_d)),
$$
providing uniform linear rate of expansion of the positive set of $u$, which then yields the non-degeneracy of $u$ due to the velocity law $V = |\nabla u|$.  

\medskip

Our claim above is proved in \cite{CJK} for the case $f=\vec{b}=0$. For the proof $u$ was compared with a subsolution of the form $\sup_{|y-x| \leq \varphi(x)\e} u(y,t)$, where $\varphi(x)$ is a chosen radius function first introduced by Caffrelli \cite{caff89}. The radius function $\varphi$ will be small on the boundary of the unit neighborhood but is larger near the point $x_0+e_d$, which yields the desired result. Of course to elaborate this idea the precise subsolution is more involved than stated, to accomodate a sizable perturbation by the radius function. For our problem we employ this idea but with significant modifications due to the presence of both $f$ ad $\vec{b}$, as we will see below (the barrier construction is given in the proof of Theorem~\ref{T.6.1}).


\medskip

Let us now proceed with the assumptions for this section. For any $\beta\in (1,\frac32)$,  
let  $\theta_\beta$ be given in Lemma \ref{L.3.1} so that  \eqref{angle_1} holds.  
We will assume that $u$ is a solution to \eqref{1.1} in $B_2\times (-1,1)$ with the following properties in $\calQ_1:= B_1\times (-1,1)$:
\begin{itemize}
\item[(H-a')] $u(\cdot,t)$ is non-decreasing with respect to $W_{\theta,-e_d}$ for some $\theta \in (\theta_{\beta}, \frac\pi2)$ and $\beta\in (1,\frac32)$;
\smallskip
\item[(H-b)] $(0,0) \in \Gamma_u$ and $m:=\inf_{t\in (-1,1)}u(-e_d,t)>0$;
\smallskip
\item[(H-c)] $u_t\geq \vec{b}\cdot\nabla u-C_0u$ for some $C_0>0$ (in the viscosity sense).
\end{itemize}

\medskip

Note that (H-a') is obtained from the previous sections, in particular from Theorem \ref{T.2.1} .  (H-b) defines $m$ as a parameter, since it is proportional to the rate the positivity set of $u$ expands over time.  The last condition (H-c) states that $u$ almost increases along the streamline. While we showed the monotonicity along the streamline  for the positive set in Lemma \ref{L.2.9}, it remains open whether this property holds for the solution $u$: the difficulty lies in the fact that, if we were to compare $u(x,t)$ with $u(X(t;x),t)$, the corresponding elliptic operator involves higher order derivatives of the drift $\vec{b}$, and thus one cannot directly compare the two functions based on the order of their support, {unless $\vec{b}$ is identically zero, or a constant vector field.} In the global setting, (H-c) can be derived for the initial value problem with smooth $\vec{b}$ and smooth positive $f$ (\cite{chu2022}).

\medskip

\subsection{Some properties of the expansion of positive sets.} 
Recall Definition \ref{D.3.2} about the expansion of $\Omega_u$. We observe that such property propagates backward in time.

\begin{lemma}\lb{L.6.2}
Let $c:=e^{-\|\nabla\vec{b}\|_\infty}$. 
If for some $t,r_t\in (0,1)$ sufficiently small,
\[
B_{r_t}(X(t;x_0))\subseteq \Omega_u(t_0+t)\quad \text{ for all }(x_0,t_0)\in \Gamma_u\cap \calQ_1,
\]
then
\[
B_{cr_t}(X(-t;x_0))\subseteq \Omega_u(t_0-t)^c\quad \text{ for all }(x_0,t_0)\in \Gamma_u\cap \calQ_{1/2}.
\]
\end{lemma}
\begin{proof}
Denoting $x_1:=X(-t;x_0)$ with $(x_0,t_0)\in \Gamma_u\cap\calQ_{1/2}$, Lemma \ref{L.2.9} yields that $x_1\in \Omega_u(t_0-t)^c\cap B_1$.
Suppose for contraction that there exists $x_2\in B_{cr_t}(x_1)$ such that $x_2\in\Gamma_u(t_0-t)$. If $t,r_t\in(0,1)$ are sufficiently small, then $(x_2,t_0-t)\in \calQ_1$. By the assumption, we have
\beq\lb{6.12}
B_{r_t}(X(t;x_2))\subseteq \Omega_u(t_0).
\eeq
Next since, for all $s\in (0,t)$,
\[
\frac{d}{ds}|X(s;x_1)-X(s;x_2)|\leq \|\nabla\vec{b}\|_\infty |X(s;x_1)-X(s;x_2)|,
\]
Gronwall's inequality yields
\[
|x_0-X(t;x_2)|=|X(t;x_1)-X(t;x_2)|\leq e^{\|\nabla\vec{b}\|_\infty t}|x_1-x_2|< e^{\|\nabla\vec{b}\|_\infty}cr_t= r_t.
\]
However this contradicts with \eqref{6.12} and $x_0\in\Gamma_u(t_0)$.
\end{proof}

Next we introduce a lemma that says characterizing the movement of the free boundary backward in time is the same as characterizing the growth of solutions forward in time.

\begin{lemma}\lb{L.6.3}
Let $r_1,r_2\in (0,1)$. Then the following is true for sufficiently small $\eps>0$:
Suppose  that there is $\tau>0$ such that 
\beq\lb{6.9}
B_{r_1\eps}(X(- \tau\eps;x)-r_2\eps e_d)\subseteq \Omega_u(t- \tau\eps)^c \hbox{ for all } (x,t)\in \Gamma_u \cap \calQ_1.
\eeq
Then, for some universal $C>0$,
\[
u(X(\tau\eps;x)+ r_2\eps e_d,t+\tau\eps)>0 \hbox{ in } \calQ_{1-C\eps}. 
\]
\end{lemma}
\begin{proof}
Let us fix $(x_0,t_0)\in \Gamma_u\cap \calQ_{1-C\eps}$. Suppose for contradiction that
\[
u(X(\tau\eps;x_0)+ r_2\eps e_d,t_0+\tau\eps)=0.
\]
Then, using (H-a') and $(X(\tau\eps;x_0),t_0+\tau\eps)\in \Omega_u$ by Proposition \ref{L.3.3}, there exists $h\in(0, r_2)$ such that $(X(\tau\eps;x_0)+ h\eps e_d,t_0+\tau\eps)\in\Gamma_u\cap \calQ_1$ if $C$ is large enough. So  \eqref{6.9} with $t=t_0+\tau\eps$ and $x=X(\tau\eps;x_0)+ h\eps e_d$ yields that
\beq\lb{6.10}
B_{{r_1\eps}}(X(-\tau\eps;X(\tau\eps;x_0)+ h\eps e_d)-r_2\eps e_d)\subseteq \Omega_u(t_0)^c.
\eeq

For $s\in [-\tau\eps,0]$, set 
\[
Y(s):=X(s;X(\tau\eps;x_0)+ h\eps e_d)-X(s;X(\tau\eps;x_0)).
\]
It is clear that $Y(0)=h\eps e_d$.
Using \eqref{ode} yields for all $s\in [-\tau\eps,0]$,
\begin{align*}
|Y(s)|\leq h\eps+\int_{s}^0\|\nabla\vec{b}\|_\infty |Y(\tau)|d\tau.
\end{align*}
Thus, Gronwall's inequality yields
\[
|Y(s)|\leq h\eps e^{\|\nabla\vec{b}\|_\infty \tau\eps}\leq 2h\eps.
\]
if $\eps$ is sufficiently small. Since $X(-\tau\eps;X(\tau\eps;x_0))=x_0$ and $Y(0)=h\eps e_d$, we get
\begin{align*}
&|X(-\tau\eps;X(\tau\eps;x_0)+h\eps e_d)-x_0- h\eps e_d|=|Y(-\tau\eps)-Y(0)|\\
&\qquad\qquad\leq \int_{-\tau\eps}^0\|\nabla\vec{b}\|_\infty|Y(s)|ds\leq 2h\tau\eps^2 \|\nabla\vec{b}\|_\infty
\end{align*}
which is less than $r_1\eps/2$ if $\eps$ is sufficiently small.
This and \eqref{6.10} imply that
\[
B_{{r_1\eps}/2}(x_0-(r_2-h)\eps e_d)\subseteq \Omega_u(t_0)^c.
\]
However since $u$ is non-decreasing along $-e_d$ direction and $h\leq r_2$, this
contradicts with $( x_0,t_0)\in\Gamma_u$, which leads to the conclusion. 
\end{proof}

\subsection{Uniform rate of expansion and non-degeneracy} 

Now we are ready to show that the support of our solution strictly expands with respect to streamlines. To show this we apply sup-convolutions as in Section \ref{S.5} to construct perturbed subsolutions, but our domain is no longer a thin strip near the free boundary. 
The construction of the barrier function in a thin strip domain was enough in Section \ref{S.6}, since there we showed the propagation of cone monotonicity over time, which came from the interior of the support, $\eps^\kappa$-away from the boundary.  Here we will show propagation of the interior non-degeneracy, which only holds unit distance away from the boundary. This necessitates our construction of the barrier different from the previous section. 

\medskip


\begin{theorem}\lb{T.6.1}
Assume {\rm (H-a')(H-b)(H-c)}. If $f$ is Lipschitz continuous, then
there exists $C_1>0$ such that 
\[
u(X(C_1\eps;x)+\eps e_d,t+C_1\eps)>0\quad\text{ for }(x,t)\in\Gamma_u\cap \calQ_{1/2} \hbox{ for sufficiently small } \e>0.
\]
\end{theorem}
\begin{proof}
{As mentioned earlier in this section, the proof relies on the comparison between $u$ and $v$, a sup-convolution of $u$ with a varying radius function. More precisely we will compare a perturbed version of these functions, $U$ and $V$.
 For the construction of $U$ and $V$, below we will work with $\theta>\theta_{\beta}$ that is slightly smaller than the one given in the assumption. }


\medskip

We first choose parameters ($t_*,r_*$ and $\sigma_i$ with $i=1,2,3$) to be used in the proof. By Lemma \ref{L.6.2} and Proposition \ref{L.3.3}, there exists $c>0$ such that for $0<t_*<1/3$  we have 
\beq\lb{6.3}
u(\cdot,t_0-t_*)=0\quad \text{ in }B_{2r_*}(X(-t_*;x_0))  \hbox{ where } r_*:=c\,t_*^{1/(2-\beta)}<\frac13,
\eeq 
for any free boundary point $(x_0,t_0) \in\Gamma_u$ in $\calQ_{2/3}$.

\medskip

Let $A_0,A_1\geq 1$ be from Lemma \ref{L.2.7}, $C_0$ from the assumption, and let $M_0\geq 1$ satisfying \eqref{M0} below which only depends on $d,A_0,\theta$.
We call $L:=(1+\|f\|_{C^1}+\|\vec{b}\|_{C^1})^2$, and define
\beq\label{6.4}
\sigma_1:=A_1M_0,\quad \sigma_2:=L(20M_0^2+2M_0((A_1+2)M_0+2)t_*/r_*),\quad \sigma_3:=(A_1+2)M_0+2.
\eeq
Note that $t_*^2<<r_*$ due to $\beta<\frac32$, so we can choose $t_*>0$ to be small enough that 
\beq\lb{6.6_cor}
t_*\leq \min\left\{\frac1{5\sigma_2},\frac{\sigma_1}{C_0\sigma_3},\frac{1}{\sigma_3}\right\}.
\eeq

%

\medskip

Let us fix the reference point $(x_0,t_0)\in \Gamma_u\cap \calQ_{2/3}$. 
After translations, we may assume that $t_0=t_*$ and $X(-t_*;x_0)=0$. Then $X(t):=X(t;0)$ satisfies
\begin{equation}\label{set_up}
(X(t_*),t_*)=(X(t_*;X(-t_*;x_0)),t_*)=(x_0,t_0)\in\Gamma_u.
\end{equation}
Define
$
\bar{u}(x,t):=u(x+X(t),t)
$ which solves \eqref{3.1} with $f_0,\vec{b}_0$ satisfying
\beq\lb{3.1''}
\|f_0\|_{C^1},\,\|\nabla \vec{b}_0\|_\infty,\,\|\partial_t \vec{b}_0\|_\infty\leq L,\quad |\vec{b}_0(x,t)|\leq L|x|.
\eeq
We will work in the cylindrical domain
$$
\Sigma:= (B_{r_*}(x_1)\setminus B_{r_{\delta,\theta}}(x_1)) \times [0, t_*]\quad\text{ where }x_1:=r_*e_d/5 \hbox{ and } r_{\delta,\theta}:= r_*\sin\theta/10.
$$

\medskip

\noindent$\circ$ {\it Construction of $U$ and $V$:} 

First we perturb $\bar{u}$ to define $U$.
Suppose $w^t$ satisfies $-\Delta w^t=1$ in $B_{r_*}(x_1)\cap \Omega_{\bar u}(t)$ and $w^t=0$ on $ \overline{(B_{r_*}(x_1)\cap\Omega_{\bar u}(t))^c}$.
Note that, from the cone-monotonicity assumption on $u$, it follows that $\Gamma_{\bar u}(t)$ is a Lipschitz graph with Lipschitz constant smaller than  $
\cot\theta_{\beta}$. 
 Corollary \ref{C.2.7}  and (H-b) then yield that 
\[
w^t\leq Cr_*^{2-\beta}\bar{u}(\cdot,t)\quad \text{ in }B_{r_*}(x_1) \hbox{ for some } C=C(m).
\]
Since $\beta<2$, after further taking $t_*$ to be sufficiently small (then $r_*=c\,t_*^{1/(2-\beta)}$ is small) depending only on $c,C,L$ and $M_0$, we have for all $t\in[0,t_*]$,
\beq\lb{6.60}
L(M_0+2)w^t\leq \bar{u}(\cdot,t)\quad\text{ in }B_{r_*}(x_1).
\eeq

We define
\beq\lb{6.14}
U(x,t):=\bar{u}(x,t)+L(M_0+2)\eps w^t(x).
\eeq
We claim that $U$ is a supersolution to
$$
    \left\{\begin{aligned}
        -\Delta U &=f_0(x,t)+L(M_0+2)\eps \quad &&\text{ in }\Sigma\cap\Omega_U,\\
        U_t&=(1-\eps)|\nabla U|^2+\vec{b}_0\cdot\nabla U\quad &&\text{ on } \Sigma\cap\Gamma_U.
    \end{aligned}\right.\leqno(P_\e)
$$
By the construction of $w^t$, it is direct to see the inequality in $\Sigma\cap\Omega_U$.

Let us check the supersolution property on the free boundary. 
Suppose $U-\phi$ for some $\phi\in C^{2,1}_{x,t}$ has a local minimum in $\{t\leq s_0\}$ at some $(y_0,s_0)\in\Gamma_U\cap\Sigma$ and $|\nabla\phi(y_0,s_0)|\neq 0$ and 
\beq\lb{6.71}
-(\Delta\phi+f_0+L(M_0+2)\eps)(y_0,s_0)<0.
\eeq
Because \eqref{6.60} yields $U\leq (1-\eps)^{-1}\bar{u}$ in $\Sigma$, we have that $\bar{u}-(1-\eps)\phi$ obtains a local minimum at $(y_0,s_0)\in\Gamma_{\bar u}$.
Note that \eqref{6.71} and \eqref{3.1''} yield
\[
-(\Delta (1-\eps)\phi+f_0)(y_0,s_0)<0.
\]
So using that $\bar{u}$ is a viscosity solution to \eqref{3.1}, we get
\[
\phi_t\geq (1-\eps)|\nabla\phi|^2+\vec{b}_0\cdot\nabla\phi\quad\text{ at }(y_0,s_0),
\]
which proves that $U$ is a supersolution to $(P_\e)$.

\medskip

We will use the following $\Phi$ to construct the radius function for $V$. Let $\Phi$ be the unique solution to
\begin{equation}\lb{6.1}
    \left\{\begin{aligned}
        & \Delta (\Phi^{-A_0+1})=0 &\text{ in }&B_1\backslash B_{\sin \theta/10}\\
        & \Phi=A_\theta &\text{ on }&\partial  B_{\sin \theta/10}\\
        & \Phi=(\sin{\theta})/2& \text{ on }&\partial B_{1}
    \end{aligned}
    \right.
\end{equation}
where $A_\theta$ is chosen sufficiently large so that
\begin{equation}\label{6.2}
    \Phi\left(-e_d/{5}\right)\geq 3.  
\end{equation}
We have $\Delta \Phi= \frac{A_0|\nabla\Phi|^2}{\Phi}$ in $B_1\backslash B_{\sin\theta/10}$, and there exists $M_0=M_0(d,A_0,\theta)\geq 1$ such that
\beq\lb{M0}
{M_0}^{-1}\leq \Phi\leq M_0,\quad \|\nabla \Phi\|_\infty \leq M_0 \quad \hbox{ in } B_1\backslash B_{\sin \theta/10}.
\eeq

\medskip

Let $\varphi(x):={r_*}\Phi(\frac{x-x_1}{r_*})$ where $x_1= r_* e_d/5$, and define 
\beq\lb{6.62}
V(x,t):=(1-\sigma_1\eps)\sup_{y\in B_{\eps(1-\sigma_2 t)\varphi(x)}(x)}\bar{u}(y+r_*\eps e_d,(1-\sigma_3\eps )t).
\eeq

 We now prove that $V$ is a viscosity subsolution of $(P_\e)$ in $\Sigma$. Recall that $\bar u$ satisfies \eqref{3.1}, and $f_0,\vec{b}_0$ given in \eqref{3.1'} satisfy \eqref{3.1''}.
Thus Lemma \ref{L.2.7} (with $f_0,\vec{b}_0$ in place of $f,\vec{b}$) yields
\[
-\Delta V(x,t)\leq (1-\sigma_1\eps)(1+A_1M_0\eps)f_0(y(x,t)+ r_*\eps e_d,(1- \sigma_3\eps)t)
\]
where $y(\cdot,\cdot)$ satisfies $|y(x,t)-x|\leq \eps(1-\sigma_2 t)\varphi(x)\leq r_* M_0\eps $. 
Using this, \eqref{6.4} and \eqref{3.1''} yields
\begin{align*}
&(1-\sigma_1\eps)(1+A_1M_0\eps)f_0(y(x,t)+ r_*\eps e_d,(1- \sigma_3\eps)t)\\
&\qquad\qquad\leq f_0(x,t)+L\eps ((1+M_0)r_*+\sigma_3t_*).
\end{align*}
Due to \eqref{6.14}, $-\Delta w^t=1$, $r_*\leq 1$ and $t_*\leq 1/\sigma_3$, we obtain 
\beq\lb{6.8}
-\Delta V(x,t)\leq f_0(x,t)+L\eps(M_0+2)\quad\text{ in }\Sigma.
\eeq

Next to prove that $V$ satisfies the free boundary condition on $(y_0,s_0)\in \Gamma_{\bar v}\cap\Sigma$, suppose that for a test function $\phi\in C_{x,t}^{2,1}$, $V-\phi$ has a local maximum in $\overline{\Omega_{V}}\cap\{t\leq s_0\}$ at $(y_0,s_0)$. 
So 
\[
\sup_{y\in B_{\eps(1-\sigma_2 t)\varphi(x)}(x)}\bar{u}(y+r_*\eps e_d,t)-\frac1{(1-\sigma_1\eps)}\phi(x,(1-\sigma_3\eps )^{-1}t)
\]
has a local maximum at $(y_0,(1-\sigma_3\eps )s_0)$ in $\overline{\Omega_{V}}\cap\{t\leq (1-\sigma_3\eps )s_0\}$.

Recall that
\[
M_0^{-1}r_*\leq \varphi\leq M_0r_*,\quad |\nabla\varphi|\leq M_0.
\]
It follows from Lemma \ref{L.2.8} and its remark (with $f_0,\vec{b}_0$ in place of $f,\vec{b}$, and $\eps_1:=M_0r_*\eps, \eps_2:=M_0\eps,\eps_3:=-\sigma_2r_*\eps/M_0$) that at $(y_0,(1-\sigma_3\eps )s_0)$,
\begin{align*}
(1- \sigma_3\eps)^{-1}\phi_t&\leq (1-\sigma_1\eps)^{-1}(1+2\eps_2)^2|\nabla \phi|^2+\vec{b}_0\cdot\nabla \phi\\ 
&\quad+\left(\eps_1\|\nabla\vec{b}_0\|_{L^\infty(\Sigma')}+2(\eps_1+\eps_2)\|\vec{b}_0\|_{L^\infty(\Sigma')}-\eps_3/2\right)|\nabla \phi|.
\end{align*}
Using \eqref{6.4}, \eqref{3.1''} and $(y_0,s_0)\in \Sigma\subseteq B_{2r_*}\times [0,t_*]$ yields for $\eps$ sufficiently small,
\begin{align*}
\phi_t
&\leq (1-\eps)|\nabla \phi|^2+(1- \sigma_3\eps)\vec{b}_0(y_0,(1-\sigma_3\eps )s_0)\cdot\nabla \phi+(1- \sigma_3\eps)\left(9M_0L-\sigma_2/(2M_0)\right)r_*\eps|\nabla \phi|\\
&\leq (1-\eps)|\nabla \phi|^2+\vec{b}_0(y_0,s_0)\cdot\nabla \phi+
(1-\sigma_3\eps)\left(\sigma_3 L t_*+Lr_* +9M_0Lr_*-\sigma_2r_*/(2M_0)\right)\eps|\nabla \phi|\\
&\leq (1-\eps)|\nabla \phi|^2+\vec{b}_0(y_0,s_0)\cdot\nabla \phi.
\end{align*}

\noindent$\circ$ {\it Comparison of $V$ and $U$:} We are going to show next that
\begin{equation}\label{comparison}
V \leq U \quad\hbox{ in } \Sigma.
\end{equation}
By the comparison principle applied to $(P)_\e$, it is enough to show that $V \prec U$ on the parabolic boundary of the domain. Below we always consider $(x,t)\in  B_{r_*}(x_1)\times [0,t_*]=:\Sigma'$ unless otherwise stated.

\medskip

We claim that  $V \prec \bar{u}$ on the parabolic boundary of $\Sigma$, which will suffice due to the fact that $\bar{u} \leq U$ by definition.  From \eqref{6.3} that $\bar{u}(\cdot,0)=u(\cdot,0)=0$ in $B_{2r_*}(0)\supseteq B_{9r_*/5}(x_1)$. Because 
\[
(1-\sigma_2t)\varphi(x)\eps+r_*\eps\leq (1+M_0)r_*\eps\leq 4r_*/5
\]
in $\Sigma'$ when $\eps$ is small, we obtain
\[
V(x,0)=0=\bar{u}(x,0)\quad\hbox{ in } B_{r_*}\left( x_1\right).
\]
Moreover, the same holds for small $t>0$, and so $U$ and $V$ cannot cross on the initial boundary of $\Sigma$.

Next we consider the inner lateral boundary of $\Sigma$. Due to $\bar u(0,t_*)=u(X(t_*),t_*)=0$ and the monotonicity of support along streamlines (Lemma \ref{L.2.9}),
\[
\bar{u}(0,t)=u(X({t}),t)=0\quad \text{ for }t\in [0,{t_*}].
\]  
Then since $\bar{u}$ is non-decreasing along all directions of $W_{\theta,-e_d}$, we get $\bar{u}=0$ in $  B_{r_*\sin\theta/5}\left(x_1+ r_*e_d\eps\right)\times [0,{t_*}]$. Thus by taking $\eps>0$ to be small enough such that
\[
(1-\sigma_2t)\varphi(x)\eps\leq M_0r_*\eps\leq  r_*\sin\theta/10=r_{\delta,\theta},
\]
we get
\[
V(\cdot,\cdot)\leq \sup_{B_{r_{\delta,\theta}}}\bar{u}(\cdot+ r_*\eps e_d,(1- \sigma_3\eps)\cdot)=0\,(=\bar{u}) \quad\text{ in }  B_{r_{\delta,\theta}}(x_1)\times [0,{t_*}].
\]

Now it remains to show that $V\prec  \bar{u}\, (\leq U)$ on the outer lateral boundary $ \partial B_{r_{\delta}}(x_1) \times [0,t_{\delta}]$. To do this, we use both the assumptions (H-a') and (H-c).
Indeed, it is not hard to derive from the latter that
$e^{C_0t}u(X(t;x),t)$ is non-creasing in $t$. 
In particular, writing $x_t:=x+X(t)$ for $(x,t)\in\Sigma'$, we get
\[
\bar{u}(x,t)=u(x_t,t)\geq e^{-C_0\sigma_3 \eps t}u(X(- \sigma_3\eps t;x_t),t- \sigma_3\eps t).
\]
This and the cone-monotonicity then yield
\beq\lb{6.7}
\bar{u}(x,t)\geq e^{-C_0\sigma_3 \eps t}\sup_{y\in B_{r_*\eps \sin\theta}}u(y+r_*\eps e_d+X(- \sigma_3\eps t;x_t),t- \sigma_3\eps t).
\eeq

Note that $X(0;x_t)=x_t=X(0;X(t))+x$.
Therefore
\beq\lb{6.22}
\begin{aligned}
   & \left|X(- \sigma_3\eps t;x_t)-X(- \sigma_3\eps t;X(t))-x\right|\\
   &\qquad\qquad=\left|(X(- \sigma_3\eps t;x_t)-X(0;x_t)-(X(-\sigma_3 \eps t;X(t))-X(0;X(t)))\right|
    \\
    &\qquad\qquad\leq \int_{- \sigma_3\eps t}^0|\vec{b}(X(-s;x_t))-\vec{b}(X(-s;X(t)))|ds.
\end{aligned}
\eeq
By direct computations, for $s\in [- \sigma_3\eps t,0]$,
\begin{align*}
&|\vec{b}(X(-s;x_t))-\vec{b}(X(-s;X(t)))|\leq \|\nabla\vec{b}\|_\infty(|X(-s;x_t)-X(-s;X(t))|)\\
&\qquad\qquad \leq\|\nabla\vec{b}\|_\infty\|\vec{b}\|_\infty |2s|\leq 2\sigma_3\|\nabla\vec{b}\|_\infty\|\vec{b}\|_\infty  t_*\eps.
\end{align*}
Then, if $\eps$ is small enough, \eqref{6.22} yields 
\[
\left|X(- \eps t;x_t)-X(- \eps t;X(t))-x\right|\leq  2\|\nabla\vec{b}\|_\infty\|\vec{b}\|_\infty  t_*^2\eps^2\leq r_*\eps\sin\theta/2.
\]
Combining this with \eqref{6.7} implies
\beq\lb{6.23}
\bar{u}(x,t)\geq (1-\sigma_1\eps)\sup_{y\in B_{r_*\eps\sin\theta/2}(x)} u(y+ r_*\eps e_d+X(t- \sigma_3\eps t),t- \sigma_3\eps t).
\eeq
Here we also used $e^{-C_0\sigma_3\eps t}
\geq 1-\sigma_1\eps$ for $t\in [0,t_*]$  by \eqref{6.6_cor}, and $X(- \sigma_3\eps t;X(t))=X(t-\sigma_3 \eps t)$.

Take $(x,t)$ on the outer lateral boundary of $\Sigma$ (then $(x,t)\in \partial B_{r_*}(x_1)\times [0,t_*]$). Since $\varphi(x)=r_*(\sin\theta) /2$,  \eqref{6.6_cor} yields
\[
\eps(1-\sigma_2 t)\varphi(x)=\eps(1-\sigma_2 t)r_*\sin\theta/2\leq  r_*\eps\sin\theta/2.
\]
Thus \eqref{6.23} yields that $V\leq \bar{u}$ on $\partial B_{r_*}(x_1)\times [0,t_*]$. If $V(x,t)>0$ for some $(x,t)\in\partial B_{r_*}(x_1)\times [0,t_*]$, it is easy to get $V<\bar{u}$ at $(x,t)$ from the above proof. In addition, the separation of supports follows from the fact that $u$ is monotone with respect to $W_{\theta_0,-e_d}$ for $\theta<\theta_0$.
In summary, we conclude that 
$$V\prec \bar{u}\quad\hbox{ on } \partial B_{r_*}(x_1)\times [0,t_*].
$$

\medskip

{Now we will use \eqref{comparison} to conclude the theorem.}

\smallskip

\noindent$\circ$ {\it Proof of the Theorem:} Note that \eqref{6.2}  yields
\[
\varphi(0)=r_*\Phi\left(-e_d/{5}\right)\geq 3{r_*}.
\]
Hence we have
 \begin{align*}
     B_{r_*\eps/5}(-{r_*}\eps e_d)\subseteq B_{ {12}{r_*}\eps/5 }(0)+ {r_*}\eps  e_d\subseteq B_{\eps\varphi(0) (1-\sigma_2 t_*)}+ {r_*}\eps e_d.
 \end{align*}
With this, by \eqref{6.6_cor} and \eqref{comparison}, we get
\beq\lb{6.15}
\begin{aligned}
\bar u(0,t_*)+L(M_0+2)\eps w^{t_*}(0)&=U(0,t_*)\geq V(0,t_*)\\
&\geq \sup_{|z|\leq {r_*\eps}/{5}}(1-\sigma_1\eps )\bar{u}(z-{r_*}\eps  e_d,t_*- \sigma_3t_*\eps ).
\end{aligned}
\eeq
Due to $(X(t_*),t_*)\in\Gamma_u$ by \eqref{set_up}, and the definition of $w^{t_*}$, we get $\bar u(0,t_*)= w^{t_*}(0)=0$. Thus \eqref{6.15} yields
\[
u(z+X(-\sigma_3t_*\eps ;X(t_*))-r_*\eps e_d,t_*-\sigma_3t_*\eps )=0\quad\text{ for all }z\in B_{r_*\eps/5}.
\]

In summary, after translations, we proved for all $(x,t)\in\Gamma_u\cap \calQ_{2/3}$,
\[
B_{{r_*\eps}/{5}}(X(-\sigma_3 t_*\eps;x)-r_*\eps e_d/2)\subseteq \Omega_u(t-\sigma_3 t_*\eps)^c.
\]
The proof is now completed by invoking Lemma \ref{L.6.3}.
\end{proof}

\subsection*{Proof of Theorem B}

\medskip

We now show the non-degeneracy result, Theorem B. The proof is a consequence of Theorem \ref{T.6.1}, closely following the arguments given in \cite{CJK}. {We will prove that $u$ grows at least linearly near the free boundary (Theorem~\ref{T.6.8}), which readily delivers the desired result. }

\medskip

Heuristically speaking, the strict expansion of the positive set $\Omega_u$ along the streamline, along with the velocity law $V = |\nabla u| - \vec{b}\cdot\nu$, should provide a lower bound for $|\nabla u|$ on the free boundary. One needs to ensure however that $u$ does not change too much over time, to be able to relate the rate of expansion of the positive set with the size of the pressure variable. This is where we need a Carleson-type estimate (see also Lemma 2.5-2.6 in \cite{CJK}), and its proof is parallel to that of Corollary 2.2 in \cite{CJK}. 
\medskip

{Let us denote $(HS)$ by the particular case of \eqref{1.1} with $f, \vec{b} \equiv 0$. }

\begin{lemma}\lb{C.7.2}
Suppose that $v$ is a subsolution to $(HS)$ in $\calQ_2$ satisfying $C^{-1}\leq\frac{v(-e_d,t)}{v(-e_d,0)}\leq C$ for some $C\geq 1$. Also suppose that $\Gamma_v(t) \cap B_1(0)$ can be represented by $x_n=g_t(x')$ where $g_t:\bbR^{d-1}\to\bbR$ is Lipschitz with $\|g\|_{\Lip} \leq c_d$ for some dimensional constant $c_d>0$. 
Then there is $\delta_d>0$ such that the following holds for $0<\delta < \delta_d$:  for any $x_0\in\Gamma_v(0)\cap B_1$, $x_1\in \Omega_v(0)$ and $x_2\in \Omega_v(0)^c$ such that
\[
\frac{\delta}{2}\leq |x_1-x_0|,\quad |x_2-x_0|,\quad d(x_1,\Gamma_v(0)),\quad d(x_2,\Gamma_v(0))\leq \delta,
\]
we have for some $M$ depending on $C$ that
\beq\lb{7.2}
\dfrac{\delta^2}{v(x_1,0)}\leq M\, T(x_2) ,\quad \hbox{ where }T(x):= \sup\{t\geq0\,:\, v(x,t)=0\}.
\eeq
\end{lemma}

Now we are ready to prove the non-degeneracy result. Note that Theorem \ref{T.2.2} follows directly from Theorem \ref{T.6.8} and Lemma \ref{L.2.88}, since Theorem \ref{T.2.1} yields (H-a') in $\calQ_1$.

\begin{theorem}\lb{T.6.8}
Assume the conditions of Theorem \ref{T.6.1}. 
Moreover, suppose that $\theta\geq\arccot c_d$ (with $c_d$ from Lemma \ref{C.7.2}) and $C^{-1}\leq \frac{u(-e_d,t)}{u(-e_d,0)}\leq C$ for all $t\in (-1,1)$ and for some $C\geq 1$. Then there exist $\delta_0,c_0>0$ such that for all $\delta\in (0,\delta_0)$, 
\begin{equation*}
     u(x-\delta e_d,t)\geq c_0\delta \quad \hbox{ for all } (x,t)\in \Gamma_u\cap \calQ_{1/2}.
\end{equation*} 
\end{theorem}

\begin{proof}
 We will only show the conclusion for $(x,t)=(0,0)$, which is on $\Gamma_u$ by our setting. 
Let $C_1$ from Theorem \ref{T.6.1}, 
and choose $c_1:= (2C_1)^{-1}$ . Then $\bar{u}(x,t):=u(x+X(t)+c_1t e_d,t )$  satisfies 
\begin{equation}\lb{7.6}
    \left\{\begin{aligned}
        -\Delta \bar u &=f'(x,t)\quad  &&\text{ in }\{\bar u>0\},\\
        \bar u_t&=|\nabla \bar u|^2+\vec{b}'(x,t)\cdot\nabla \bar u\quad  &&\text{ on }\partial\{\bar u>0\},
    \end{aligned}\right.
\end{equation}
where 
\beq\lb{7.13}
f'(x,t):=f(x+X(t)+c_1t e_d) \hbox{ and } \vec{b}'(x,t):=\vec{b}(x+X(t)+c_1t e_d)-\vec{b}(X(t)) +c_1 e_d.
\eeq

\medskip

For each $t\in (-1/2,1/2)$, let $w_1(\cdot,t)$ be the unique non-negative harmonic function in ${\Omega}_{\bar u}(t)\cap B_1$ such that $w_1(\cdot,t)=0$ on $\bar{\Gamma}(t)$, and $w_1(\cdot,t)=\bar{u}(\cdot,t)$ on $ \Gamma_{\bar u}(t)\cap \partial B_1$. 
It follows from Lemma 11.12 \cite{CafSal} that any harmonic function is monotone along the monotonicity direction of its Lipschitz domain, if sufficiently close to its domain boundary where it assumes zero boundary data. In particular, we have $\nabla_{-x_d} w_1(\cdot,t)\geq 0$ in $B_r$ for some $r\in(0,1)$.
Let us fix one such $r$ that also satisfies 
\[
r < \min\left\{1, \frac{c_1}{(1+c_1)}(\|\nabla\vec{b}\|_\infty)^{-1}\right\}.
\]

Next, for $w_2:=\bar{u}-w_1$, it follows from Corollary \ref{C.2.7} that there exists $C_2>1$ such that $w_2\leq (C_2-1)w_1$. So we get
\beq\lb{7.3}
w_1\leq \bar{u}\leq C_2w_1.
\eeq

We claim that 
$C_2{w}_1$ is a subsolution to $(HS)$ in $\calQ_r $. Since $w_1$ is harmonic in its support, it suffices to verify the free boundary condition. Suppose there is a smooth function  $\phi\in C^{2,1}_{x,t}$  such that $C_2w_1-\phi$ has a local maximum zero in $\overline{\Omega_{w_1}}\cap\{t\leq t_0\}$ at $(x_0,t_0) \in \Gamma_{w_1}$. 
By \eqref{7.3}, $\bar{u}-\phi$ also obtains a local maximum in $\overline{\Omega_{\bar u}}\cap\{t\leq t_0\}$ at $(x_0,t_0)$, and therefore \eqref{7.6} and Lemma \ref{L.2.66} yield
\beq\lb{7.5}
\phi_t(x_0,t_0)\leq |\nabla\phi(x_0,t_0)|^2+\vec{b}'(x_0,t_0)\cdot\nabla\phi(x_0,t_0)
\eeq
when $|\nabla\phi(x_0,t_0)|\neq 0$. While when $\nabla\phi(x_0,t_0)= 0$, Lemma \ref{L.2.9} yields \eqref{7.5} again.
Hence, to conclude, it is enough to show that 
\begin{equation}\label{inequality_00}
\vec{b}'(x_0,t_0)\cdot\nabla\phi(x_0,t_0) \leq 0.
\end{equation}
By the assumption on $r$, we have for all $(x,t)\in \calQ_r$,
\[
|\vec{b}(x+X(t)+c_1t e_d)-\vec{b}(X(t))| \leq r(1+c_1)\|\nabla\vec{b}\|_\infty\leq c_1. 
\]
So $
\langle \vec{b}'(x,t),e_d\rangle\leq \frac\pi4$, where the notation \eqref{2.0} is used. 
By the $W_{\theta,-e_d}$-monotonicity of $\bar{u}$, we get $\phi(\cdot,t_0)\geq\phi(x_0,t_0)$ in $x_0+W_{\theta,-e_d}$ which implies $\langle\nabla\phi(x_0,t_0),-e_d\rangle\leq  {\pi}/{2}-\theta$. Consequently, also using $\theta\geq \frac{\pi}{4}$ and $\langle \vec{b}'(x,t),e_d\rangle\leq \frac\pi4$, we verified \eqref{inequality_00}. This concludes that  $w_1$ is a subsolution to $(HS)$ in $\calQ_r$.


\medskip

Lemma \ref{C.7.2} now yields that for all $\delta>0$ sufficiently small,
\[
\delta^2/w_1(-\delta e_d,0)\leq M\sup\{t\geq 0\,:\, w_1(\delta e_d,t)=0\}.
\]
Thus, \eqref{7.3} along with the definition of $\bar{u}$ yields
\beq\lb{7.10}
\delta^2/{u}(-\delta e_d,0)\leq M\sup\{t\geq 0\,:\, {u}(\delta e_d+X(t)+ c_1t e_d,t)=0\}.
\eeq
Lastly we apply Theorem \ref{T.6.1} with $\eps:=2\delta$. It follows that
if $\delta$ is sufficiently small,
\[
{u}(\delta e_d+X(2C_1\delta)+ 2C_1\delta c_1 e_d,2C_1\delta)=u(X(C_1\eps)+\eps e_d,C_1\eps)>0.
\]
Therefore 
\[
\sup\{t\geq 0\,:\, {u}(\delta e_d+X(t)+ c_1t e_d,t)=0\}\leq 2C_1\delta.
\]
From this and \eqref{7.10},
we obtain
$
{u}(-\delta e_d,0)\geq \delta/(2 C_1M)
$, 
which finishes the proof.
\end{proof}







\end{document}